\documentclass[a4paper, 10pt, linktocpage,pagebackref]{article}
\usepackage[latin1]{inputenc} 
\usepackage[T1]{fontenc} 
\usepackage[english]{babel} 
\usepackage{amsmath}
\usepackage{amsthm}
\allowdisplaybreaks
\usepackage{amssymb}
\usepackage{amsfonts}
\usepackage{mathrsfs}
\usepackage{stmaryrd}
\usepackage{color}
\usepackage{cases}
\usepackage{dsfont}
\usepackage{setspace}
\usepackage{listings}
\usepackage{lscape}
\usepackage{graphicx}
\usepackage{subfigure}
\newtheorem{Lemma}{Lemma}
\newtheorem{Proposition}{Proposition}
\newtheorem{Remark}{Remark}
\newtheorem{Theorem}{Theorem}
\newtheorem{Corollary}{Corollary}
\newtheorem{Definition}{Definition}
\usepackage{lettrine}
\usepackage{geometry}
\usepackage{hyperref}
\usepackage{caption}
\usepackage{fancyhdr}
\usepackage[numbers]{natbib}

\hypersetup{
colorlinks=true, 
linkcolor=red, 
citecolor=blue, 
urlcolor=green } 

\begin{document}

\title{Error estimates of finite difference schemes for the Korteweg-de Vries equation}
\author{Cl\'ementine Court\`es$^1$, Fr\'ed\'eric Lagouti\`ere$^2$, Fr\'ed\'eric Rousset$^1$}
\date{\today}
\footnotetext[1]{
Laboratoire de Math\'ematiques d'Orsay, Univ. Paris-Sud, CNRS, Universit\'e Paris-Saclay, 91405 Orsay, France
}
\footnotetext[2]{Univ Lyon, Universit\'e Claude Bernard Lyon 1, CNRS UMR 5208, Institut
Camille Jordan, F-69622 Villeurbanne, France}
\maketitle

%
%
%
%
\begin{abstract}
This article deals with the numerical analysis of the Cauchy problem for the Korteweg-de Vries equation 
with a finite difference scheme. We consider the explicit Rusanov scheme for the hyperbolic flux term and a 4-points $\theta$-scheme for the dispersive term. We prove the convergence under a hyperbolic Courant-Friedrichs-Lewy condition when $\theta\geq \frac{1}{2}$ and under an "Airy" Courant-Friedrichs-Lewy condition when $\theta<\frac{1}{2}$.
 More precisely, we get the first order convergence rate for strong solutions in the Sobolev space $H^s(\mathbb{R})$, 
 $s \geq 6$ and extend this result to the non-smooth case for initial data in $H^s(\mathbb{R})$, with $s\geq \frac{3}{4}$, to the price of a loss in the convergence order. Numerical simulations indicate that the orders of convergence may be optimal when $s\geq3$. 
 \end{abstract}

%
%
%
%
\section{Introduction}
We are interested in the Korteweg-de Vries equation (called the KdV equation thereafter), 
which is a model for wave propagation on shallow water surfaces in a channel and was first established by D.J. Korteweg and G. de Vries in 1895 \cite{Korteweg_de_Vries_1895}. We focus on the numerical analysis of the Cauchy problem
\begin{subequations}
\begin{numcases}{}
\partial_tu(t,x)+\partial_x\left(\frac{u^2}{2}\right)(t,x)+\partial_x^3u(t,x)=0,\ \ (t,x)\in[0,T]\times\mathbb{R},\label{EQ_INIT}\\
u_{|_{t=0}}(x)=u_0(x), \ \ \hspace*{3.8cm} x\in\mathbb{R},\label{DONNEE_INIT}
\end{numcases}
\end{subequations}
for which the local well-posedness in Sobolev spaces $H^{s}(\mathbb{R})$ is well-established: 
in particular, well-posedness was proved for $s\geq 2$ in \cite{Saut_Temam_1976}, $s>\frac{3}{2}$ in \cite{Bona_1975}, $s>\frac{3}{4}$ in \cite{Kenig_1991}, $s\geq 0$ in \cite{Bourgain_1993}, $s>-\frac{5}{8}$ in \cite{Kenig_1993} 
(note that one of the first existence results was obtained by proving the convergence of a semi-discrete scheme \cite{Sjoberg_1970}). 
Due to the conservation of the $L^2$ norm, this yields global well-posedness for any $s \geq 0$. Note that global well-posedness is even known below $L^2$ (see \cite{Colliander_2003}, for example). There are two antagonist effects in the KdV equation: the Burgers nonlinearity tends to create singularities (shock waves, which yield a blow up in finite time) whereas the linear term tends to smooth the solution due to dispersive effects (and creates dispersive oscillating waves of Airy type). In some sense the above global well-posedness results come from the fact that dispersive effects dominate. 
 
\indent Given the practical importance of the KdV equation in concrete physical situations, there exists a wide range of numerical schemes to solve it. A very classical numerical approach is \emph{the finite difference method}, which consists in approximating the exact solution $u$ by a numerical solution $(v_j^n)_{(n,j)}$ in such a way that $v_j^n\approx u(t^n,x_j)$ in which $t^n=n\Delta t$, $x_j=j\Delta x$ with $\Delta t$ and $\Delta x$ respectively the time and space steps. In most cases, the convergence is ensured only if a stability condition between $\Delta t$ and $\Delta x$ is satisfied. Let us mention for instance the explicit leap-frog scheme designed by Zabusky and Kruskal in \cite{Zabusky_Kruskal_1965} with periodic boundaries conditions, or the Lax-Friedrichs scheme studied by Vliegenthart in \cite{Vliegenthart_1971}. Both are formally convergent to the second order in space under a very restrictive stability condition $\Delta t=\mathcal{O}(\Delta x^3)$. The price to pay to avoid a so restrictive stability condition $\Delta t=\mathcal{O}(\Delta x^3)$ is to design formally an implicit scheme, as in \cite{Winther_1980}, for example, with a twelve-points implicit finite difference scheme with three time levels or in \cite{Taha_Ablowitz_1984} with a pentagonal implicit scheme.  
The analysis and the rigorous justification of the stability condition started in \cite{Vliegenthart_1971}, where Vliegenthart computed rigorously the amplification factor for a linearized equation. More recently, Holden, Koley and Risebro in \cite{Holden_2014} prove the convergence of the Lax-Friedrichs scheme with an implicit dispersion under the stability condition $\Delta t=\mathcal{O}(\Delta x^{\frac{3}{2}})$ if $u_0\in H^3(\mathbb{R})$ and $\Delta t=\mathcal{O}(\Delta x^{2})$ if $u_0\in L^2(\mathbb{R})$ (without convergence rate). More precisely, they obtain the strong convergence without rate of the numerical scheme towards a classical solution if $u_0\in H^3(\mathbb{R})$ and a strong convergence towards a weak solution $L^2(0,T;L^2_{\text{loc}}(\mathbb{R}))$ if $u_0\in L^2(\mathbb{R})$.
 
The aim of this paper is to prove rigorously the convergence of some finite difference schemes 
for the KdV equation by analyzing the rate of convergence 
and in particular its dependence with respect to the regularity of the initial datum. 
We will get a rate of convergence for rough initial data by combining precise stability
estimates for the scheme with information coming from the study of the Cauchy problem for the KdV equation
and in particular some dispersive smoothing effects. 
 
The approach of this paper could be extended to third order dispersive perturbations of hyperbolic systems. It was indeed successfully extended in \cite{Burtea_Courtes} to the $abcd$-system
 \begin{equation*}
\left\{
\begin{array}
[c]{l}%
\left( I-b\partial_{x}^{2}\right) \partial_{t}\eta+\left( I+a\partial
_{x}^{2}\right) \partial_{x}u+\partial_{x}\left( \eta u\right) =0,\\
\left( I-d\partial_{x}^{2}\right) \partial_{t}u+\left( I+c\partial
_{x}^{2}\right) \partial_{x}\eta+\frac{1}{2}\partial_{x}u^{2}=0.
\end{array}
\right. 
\end{equation*}
This system, which was introduced by Bona, Chen and Saut in \cite{Bonaetal}, is a more precise long wave asymptotic model for free surface incompressible fluids. Note that the result of \cite{Burtea_Courtes} is weaker than the result in the present paper in the sense that only the first order convergence for smooth initial data is proven. The extension to rougher initial data as in the present paper would require some significant progress in the study of the Cauchy problem at the continuous level.
 

Let us mention that many types of other numerical methods can be used to solve the KdV equation 
The equation being Hamiltonian (the Hamiltonian is the energy), symplectic schemes based on compact finite differences that conserve the energy have been designed. We refer for example to \cite{Kanazawa_2012}, \cite{Li_Visbal_2006}, \cite{Ascher_2005}.
Splitting methods (the equation being split into the linear Airy part and the nonlinear Burgers part) are also widely studied.
For example, a rigorous analysis of such schemes has been performed in \cite{Holden_Tao_2010}, \cite{Holden_Lubich_2013}.
One can also use spectral methods see \cite{Nouri_Sloan_1989} for example or \cite{Hofmanova_Schratz_2016} where a Fourier pseudo spectral method is combined with an exponential-type time-integrator. A quite widespread discretization is related to finite element type schemes, see for example \cite{Baker_1983}, \cite{Dougalis_Karakashian_1985}, \cite{Bona_Chen_Karakashian_2013} for Galerkin methods. In the recent work \cite{Dutta_Koley_Risebro_2015} where the convergence of a Galerkin-type implicit scheme is established for $L^2$ initial data. The focus is on the strong convergence in $L^2(0,T; L^2_{\text{loc}}(\mathbb{R}))$ of the fully discrete solution to a weak solution of \eqref{EQ_INIT} by a method which gives in the same way a direct and constructive existence theorem of \eqref{EQ_INIT}. Our approach is different because we want to highlight the {\em convergence rate}, with a Courant-Friedrichs-Lewy type condition (CFL-type condition) as optimal as possible.

In the present paper, we discretize Equation \eqref{EQ_INIT} together with the initial datum \eqref{DONNEE_INIT} in a finite difference way and our aim is to determine the convergence rate of this numerical scheme. We exhibit the error estimate on the convergence error by a method which suits both non-linear term and dispersive term of KdV.

Let us introduce some notations and present the finite difference scheme here under study.
\paragraph{Notations and numerical scheme} We use a uniform time- and space-discretization of \eqref{EQ_INIT}. Let $\Delta t$ be the constant time step and $\Delta x$ the constant space step. We note $t^n=n\Delta t$ for all $n\in\llbracket 0,N\rrbracket=\{0,1,..,N\}$ where $N=\lfloor \frac{T}{\Delta t}\rfloor$ (where $\lfloor . \rfloor$ denotes the integer part) and $x_j=j\Delta x$ for all $j\in\mathbb{Z}$.\\

\textbf{Numerical scheme.} Let $c\in\mathbb{R}_+^*$ and $\theta\in[0,1]$. We denote by $(v_j^n)_{(n,j)\in\mathbb{N}\times\mathbb{Z}}$ the discrete unknown defined by the following scheme with parameters $c$ and $\theta$ :
\begin{multline}
\frac{v_j^{n+1}-v_j^n}{\Delta t}+\frac{\left(v_{j+1}^n\right)^2-\left(v_{j-1}^n\right)^2}{4\Delta x}+\theta\frac{v_{j+2}^{n+1}-3v_{j+1}^{n+1}+3v_j^{n+1}-v_{j-1}^{n+1}}{\Delta x^3}\\+(1-\theta)\frac{v_{j+2}^n-3v_{j+1}^n+3v_j^n-v_{j-1}^n}{\Delta x^3}=c\left(\frac{v_{j+1}^n-2v_j^n+v_{j-1}^n}{2\Delta x}\right), \ \ n\in \llbracket0,N\rrbracket,\ j\in\mathbb{Z}
\label{EQ_SCHEME}
\end{multline}
with
\begin{equation}
v_j^0=\frac{1}{\Delta x}\int_{x_j}^{x_{j+1}}u_0(y)dy, \ \ j\in\mathbb{Z}.
\label{EQ_33333_bis_bis}
\end{equation}

If $\theta=0$, we recognize the explicit scheme whereas $\theta=1$ corresponds to the implicit scheme (with respect to the dispersive term).
Without the dispersive term $\theta\frac{v_{j+2}^{n+1}-3v_{j+1}^{n+1}+3v_j^{n+1}-v_{j-1}^{n+1}}{\Delta x^3}$$+(1-\theta)\frac{v_{j+2}^n-3v_{j+1}^n+3v_j^n-v_{j-1}^n}{\Delta x^3}$, we recognize the Rusanov scheme applied to the Burgers equation, which consists in a centered hyperbolic flux $\frac{\left(v_{j+1}^n\right)^2-\left(v_{j-1}^n\right)^2}{4\Delta x}$ and an added artificial viscosity $c\left(\frac{v_{j+1}^n-2v_j^n+v_{j-1}^n}{2\Delta x}\right)$ in order to ensure the stability of the scheme. In the following, the constant $c$ will be called the Rusanov coefficient. 

Without the non-linear term and the right-hand side, we recognize the $\theta$-right winded finite difference scheme for the Airy equation 
\begin{equation*}
\frac{v_j^{n+1}-v_j^n}{\Delta t}+\theta\frac{v_{j+2}^{n+1}-3v_{j+1}^{n+1}+3v_{j}^{n+1}-v_{j-1}^{n+1}}{\Delta x^3}+(1-\theta)\frac{v_{j+2}^{n}-3v_{j+1}^{n}+3v_{j}^{n}-v_{j-1}^{n}}{\Delta x^3}=0, \ \ n\in\llbracket0,N\rrbracket,\ j\in\mathbb{Z}.
\end{equation*}

\begin{Remark}
System \eqref{EQ_SCHEME} is invertible, for any $\Delta t, \Delta x > 0$ and any $\theta \in [0, 1]$. This will be proved in Proposition \ref{invertible} below. 
\end{Remark}
\begin{Remark}
All the results are valid with a variable time step $\Delta t^n$ and a variable Rusanov coefficient $c^n$. For simplicity, we will keep them constant.
\end{Remark}
\begin{Remark}
The choice of the right winded scheme for the dispersive part is dictated by the result in \cite{Proceeding_HYP} on numerical schemes applied to high-order dispersive equations $\partial_t u + \partial_x^{2p+1}u=0$, with $p\in\mathbb{N}$, which brought to light that right winded schemes are stable under a CFL-type condition for $p$ odd (including the Airy equation) and left winded schemes are stable under a CFL-type condition for $p$ even.
\end{Remark}
\begin{Remark}
This scheme \eqref{EQ_SCHEME}-\eqref{EQ_33333_bis_bis} is a generalization of the one studied by Holden, Koley and Risebro \cite{Holden_2014}. Indeed, they consider the Lax-Friedrichs scheme for the hyperbolic flux term together with the implicit scheme for the dispersive term, which consists in taking $c\Delta t=\Delta x$ and $\theta=1$ in Scheme \eqref{EQ_SCHEME}-\eqref{EQ_33333_bis_bis}.
\end{Remark}

\textbf{Discrete operators.} For the convenience of notations, 
we will use the notations introduced in \cite{Holden_2014} and define the following discrete operators. For any sequence $(a_j^n)_{(n,j)\in\mathbb{N}\times\mathbb{Z}}$,
\begin{equation}
D_-(a)_j^n=\frac{a_{j}^n-a_{j-1}^n}{\Delta x},\ \ \ D_+(a)_j^n=\frac{a_{j+1}^n-a_{j}^n}{\Delta x}, \ \ \ D(a)_j^n=\frac{D_+(a)_j^n+D_-(a)_j^n}{2}.
\label{EQ_DEF_OPERATOR_22}
\end{equation}
Equation \eqref{EQ_SCHEME} rewrites 
\begin{equation}
\frac{v_j^{n+1}-v_j^n}{\Delta t}+D\left(\frac{v^2}{2}\right)_j^n+\theta D_+D_+D_-\left(v\right)_j^{n+1}+(1-\theta)D_+D_+D_-\left(v\right)_j^{n}=\frac{c\Delta x}{2}D_+D_-\left(v\right)_j^n.
\label{EQ_SCHEME_2}
\end{equation}
Eventually, for all $a=\left( a_{j}\right) _{j\in\mathbb{Z}}\in\ell^{\infty}\left(
\mathbb{Z}\right) $ we introduce the spatial shift operators:%
\begin{equation}\label{Shift_2018}
\left( \mathcal{S}^{\pm}a\right) _{j}:=a_{j\pm1}.
\end{equation}

\textbf{Function spaces.} In the following, we denote by $H^{r}(\mathbb{R})$, with $r\in\mathbb{R}$, the Sobolev space whose norm is
\begin{equation}
||u||_{H^r(\mathbb{R})}=\left(\int_{\mathbb{R}}\left(1+|\xi|^2\right)^r|\widehat{u}(\xi)|^2\right)^{\frac{1}{2}},
\label{REM_1_BISBIS}
\end{equation}
where $\widehat{u}$ is the Fourier transform of $u$. If there is ambiguity, an '$x$' will be added in $H^r_x$ for the Sobolev space with respect to the space variable.

We study the convergence in the discrete space $\ell^{\infty}(\llbracket0,N\rrbracket; \ell^2_{\Delta}(\mathbb{Z}))$ whose scalar product and
norm are defined by

\[
\left\langle a,b\right\rangle :=\Delta x\sum_{j\in\mathbb{Z}}a_{j}b_{j}\text{,
}
\]
and
\begin{equation}
||a||_{\ell^{\infty}(\llbracket0,N\rrbracket, \ell^2_{\Delta}(\mathbb{Z}))}=\underset{n\in\llbracket0,N\rrbracket}{\sup}||a^n||_{\ell^2_{\Delta}}=\underset{n\in\llbracket0,N\rrbracket}{\sup}\left(\sum_{j\in\mathbb{Z}}\Delta x|a^n_j|^2\right)^{\frac{1}{2}},
\label{NORME_L2}
\end{equation}
for all $a=(a^n)_{n\in\llbracket0,N\rrbracket}=(a^n_j)_{(n,j)\in\llbracket0,N\rrbracket\times\mathbb{Z}}$ and $b=(b^n)_{n\in\llbracket0,N\rrbracket}=(b^n_j)_{(n,j)\in\llbracket0,N\rrbracket\times\mathbb{Z}}$.
This norm is a relevant discrete equivalent for the $L^{\infty}(0,T; L^2(\mathbb{R}))$-norm.\\

\textbf{Convergence error.} Let $u$ be the exact solution of \eqref{EQ_INIT}-\eqref{DONNEE_INIT}. From $u$, we construct the following sequence
\begin{equation}\label{EQ_82_barre_FRED_L}
\left\{
\begin{split}
&\left[u_{\Delta}\right]_j^{n}=\frac{1}{\Delta x[\min \left(t^{n+1}, T\right)-t^n]}\int_{t^n}^{\min\left(t^{n+1}, T\right)}\int_{x_j}^{x_{j+1}}u(s,y)dyds,\ \ \ \ \ \ \ \ \text{\ if\ }(n,j)\in\llbracket1,N\rrbracket\times\mathbb{Z},\\
&\left[u_{\Delta}\right]_j^{0}=\frac{1}{\Delta x}\int_{x_j}^{x_{j+1}}u_0(y)dy,\hspace*{6.2cm}\text{\ if\ }j\in\mathbb{Z}.
\end{split}
\right.
\end{equation}
From the averaged exact sequence $\left(\left[u_{\Delta}\right]_j^n\right)_{(n,j)}$ and the numerical one $\left(v_j^n\right)_{(n,j)}$, we define two piecewise constant functions $u_{\Delta}$ and $v_{\Delta}$ by, for all $n\in\llbracket0,N\rrbracket$ and $j\in\mathbb{Z}$,
\begin{equation}
\ \ \left\{
\begin{split}
&u_{\Delta }(t,x)=\left(u_{\Delta}\right)_j^{n},\\
&v_{\Delta }(t,x)=v_j^n,
\end{split}
\right. \ \ \ \ \ \text{\ if\ }(t,x)\in[t^n,\min\left(t^{n+1}, T\right))\times[x_j,x_{j+1}).
\label{NOTATION_BAR}
\end{equation}

\noindent We define the convergence error by the following difference 
\begin{equation}
e_j^n=v_{\Delta}(t^n,x_j)-u_{\Delta}(t^n,x_j), \ \ (n,j)\in\llbracket0,N\rrbracket \times \mathbb{Z}.
\label{EQ_8_barre_FRED_L}
\end{equation}
Thanks to Definition \eqref{NORME_L2}, the convergence error satisfies
\begin{equation*}
||e||_{\ell^{\infty}(\llbracket0,N\rrbracket; \ell^2_{\Delta}(\mathbb{Z}))}=||v_{\Delta}-u_{\Delta}||_{L^{\infty}(0, T; L^2(\mathbb{R}))}.
\end{equation*}

\textbf{Consistency error.} We denote by $\left(\epsilon_j^n\right)_{(n,j)\in\llbracket0,N\rrbracket\times\mathbb{Z}}$ the consistency error defined by the following relation
\begin{multline}
\epsilon_j^n=\frac{\left(u_{\Delta}\right)^{n+1}_j-\left(u_{\Delta}\right)^{n}_j}{\Delta t}+D\left(\frac{u_{\Delta}^2}{2}\right)_j^n+\theta D_+D_+D_-\left(u_{\Delta}\right)_j^{n+1}\\+(1-\theta)D_+D_+D_-\left(u_{\Delta}\right)_j^{n}-\frac{c\Delta x}{2}D_+D_-\left(u_{\Delta}\right)_j^n, \ \ (n,j)\in\llbracket0,N\rrbracket \times \mathbb{Z}.
\label{CONSISTENCY_DEF}
\end{multline}

\paragraph{Main result} In our first main result we handle the case of smooth enough initial data, $u_{0}\in H^s\left(\mathbb{R}\right)$, $s \geq 6$. 
\begin{Theorem}[Convergence rate in the smooth case]
Let $s\geq6$ and $u_0\in H^s(\mathbb{R})$. Let $T>0$ and $c>0$ such that the unique global solution $u$ of \eqref{EQ_INIT}-\eqref{DONNEE_INIT} satisfies 
\begin{equation}
\underset{t\in[0,T]}{\mathrm{sup}}||u(t,\cdot)||_{L^{\infty}(\mathbb{R})}< c.
\label{DEF_c_INTRO}
\end{equation}
Let $\beta_0\in(0,1)$ and $\theta\in[0,1]$. There exists $\widehat{\omega_0}>0$ such that, for every $\Delta x\leq \widehat{\omega}_0$ and $\Delta t$ satisfying \begin{subequations}\label{EQ_SUR_DELTAT_ET_DELTAX}
\begin{numcases}{}
4\left(1-2\theta\right)\frac{\Delta t}{\Delta x^3}\leq 1-\beta_0,\label{CFL_1}\\
\left[c+\frac{1}{2}\right]\frac{\Delta t}{\Delta x}\leq 1-\beta_0,\label{CFL_2}
\end{numcases}
\end{subequations}
the finite difference scheme \eqref{EQ_SCHEME}-\eqref{EQ_33333_bis_bis} with parameters $c$ and $\theta$ and time and space steps 
$\Delta t$, $\Delta x$ satisfies, for any $\eta\in(0,s-\frac{3}{2}]$,
\begin{equation}
\label{errorestimatetheo}
||e||_{\ell^{\infty}(\llbracket0,N\rrbracket;\ell^2_{\Delta}(\mathbb{Z}))}\leq \Lambda_{T, \|u_0\|_{H^{\frac{3}{4}}}}\left(1+\|u_0\|^2_{H^{\frac{1}{2}+\eta}}\right)\left(\frac{\|u_0\|_{H^6}}{c+\frac{1}{2}}+\|u_0\|_{H^4}+\|u_0\|_{H^{\frac{3}{2}+\eta}}\|u_0\|_{H^1}\right)\Delta x,
\end{equation}
where $\Lambda_{T, \|u_0\|_{H^{\frac{3}{4}}}}$ is defined by
\begin{equation}
\Lambda_{T, \|u_0\|_{H^{\frac{3}{4}}}}=\exp\left(\frac{C}{2}\left(1+c^2\right)\left(1+\frac{(1-\beta_0)^2}{(c+\frac{1}{2})^2}\right)\left(T+(T^{\frac{3}{4}}+T^{\frac{1}{2}})||u_0||_{H^{\frac{3}{4}}}e^{\kappa_{\frac{3}{4}}T}\right)\right)Ce^{\kappa T}\sqrt{T\left\{1+\frac{1-\beta_0}{c+\frac{1}{2}}\right\}},
\label{LAMBDA_0}
\end{equation}
in which $C$ is a constant, $\kappa_{\frac{3}{4}}$ and $\kappa$ depend only on $\|u_0\|_{L^2(\mathbb{R})}$. 
In Estimate \eqref{errorestimatetheo}, $e^n$ is defined as in 
\eqref{EQ_8_barre_FRED_L}-\eqref{NOTATION_BAR}-\eqref{EQ_82_barre_FRED_L}. 
\label{THEOREM_MAIN111}
\end{Theorem}

\begin{Remark}\label{REM_INtro_}
Conditions \eqref{CFL_1}-\eqref{CFL_2} are Courant-Friedrichs-Lewy-type conditions (in short, CFL conditions).\\ 
Assumption $\left[c+\frac{1}{2}\right]\frac{\Delta t}{\Delta x}\leq1-\beta_0$ seems to be only technical, and probably may be replaced with the classical hyperbolic CFL condition $c\Delta t\leq \Delta x$. Indeed, experimental results suit with Theorem \ref{THEOREM_MAIN111} with this classical CFL condition, see Section \ref{NUMERICAL_RESULTS}.
\end{Remark}
\begin{Remark}
Thereafter, $\eta$ should be chosen as small as possible, then norms $||u_0||_{H^{s+\eta}(\mathbb{R})}$ should be regarded as $||u_0||_{H^{s+}(\mathbb{R})}$.
\end{Remark}
\noindent Thus, the scheme \eqref{EQ_SCHEME}-\eqref{EQ_33333_bis_bis} is convergent to the first order in space in the $\ell^{\infty}(\llbracket0,N\rrbracket; \ell^{2}_{\Delta}(\mathbb{Z}))$-norm. 

\noindent In our second main result, we improve the previous result to handle non-smooth initial data $u_{0} \in H^s(\mathbb{R})$, $s \geq 3/4.$ To perform the analysis, we first have to approximate in a suitable way the initial datum.
Let $\chi$ be a $\mathcal{C}^{\infty}$-function such that
$$ 0\leq \chi\leq 1, \quad \chi\equiv 1 \mbox{ in } \mathcal{B}\left(0,\frac{1}{2}\right), \quad \mbox{Supp } \chi
 \subset \mathcal{B}\left(0,1\right), \quad \chi(-\xi)=\chi(\xi), \, \forall \xi \in \mathbb{R}.$$ 
Let $\varphi$ be such as $\widehat{\varphi}\left(\xi\right)=\chi\left(\xi\right)$, where $\widehat{\varphi}$ stands for the Fourier transform of $\varphi$, and for all $\delta>0$, we define $\varphi^{\delta}$ such that $\widehat{\varphi^{\delta}}\left(\xi\right)=\chi\left(\delta\xi\right)$, which implies $\varphi^{\delta}=\frac{1}{\delta}\varphi\left(\frac{.}{\delta}\right)$.
 Eventually, \begin{itemize}
 \item we shall still denote by $u$ the exact solution of \eqref{EQ_INIT} starting from the initial datum $u_0$.
 \item Let $u^{\delta}$ be the solution of \eqref{EQ_INIT} with $u_0^{\delta}=u_0\star\varphi^{\delta}$ as initial datum, where $\star$ stands for the convolution product. 
 \item We denote then by $(v_j^{n})_{(n,j)\in\llbracket0,N\rrbracket\times\mathbb{Z}}$ the numerical solution obtained by applying the numerical scheme \eqref{EQ_SCHEME} from the initial datum $(u_0^{\delta})_{\Delta}$: \begin{equation}\label{init_peu_reg}
 v_j^{0}=(u_0^{\delta})_{\Delta}=\frac{1}{\Delta x}\int_{x_j}^{x_{j+1}}u_0\star\varphi^{\delta}(y)dy.
 \end{equation}
 \end{itemize}

\begin{Theorem}[Convergence rate in the non-smooth case]
Let $s\geq\frac{3}{4}$ and $u_0\in H^s(\mathbb{R})$. Let $T>0$ and $c>0$ such that the unique global solution $u$ of \eqref{EQ_INIT}-\eqref{DONNEE_INIT} satisfies
\begin{equation*}
\underset{t\in[0,T]}{\mathrm{sup}}||u(t,\cdot)||_{L^{\infty}(\mathbb{R})}<c.
\end{equation*}
Let $\beta_0\in(0,1)$ and $\theta\in[0,1]$. There exists $\delta > 0$ and $\widehat{\omega_0}>0$ such that for every 
$\Delta x\leq \widehat{\omega_0}$ and $\Delta t$ satisfying 
\begin{equation}
\left\{
\begin{split}
&4\left(1-2\theta\right)\frac{\Delta t}{\Delta x^3}\leq1-\beta_0,\\
&\left[c+\frac{1}{2}\right]\frac{\Delta t}{\Delta x}\leq 1-\beta_0,
\end{split}
\right.\label{25_oct_17_CFL_cond}
\end{equation}
the finite difference scheme \eqref{EQ_SCHEME}-\eqref{init_peu_reg} with parameters $c$ and $\theta$ and time and space steps 
$\Delta t$, $\Delta x$ satisfies, for any $\eta\in(0,s-\frac{1}{2}]$,
\begin{equation*}
||e||_{\ell^{\infty}(\llbracket0,N\rrbracket;\ell^2_{\Delta}(\mathbb{Z}))}\leq \Gamma_{T, \|u_0\|_{H^{\frac{3}{4}}}}\left(1+\|u_0\|_{H^{\frac{1}{2}+\eta}}^2\right)\left(\frac{1}{c+\frac{1}{2}}+1+\|u_0\|_{H^{\min(1,s)}}\right) \|u_0\|_{H^s} \Delta x^{q}, 
\end{equation*}
where 
\begin{itemize}
\item $q=\frac{s}{12-2s}$ if $\frac{3}{4}\leq s\leq 3$,
\item$q=\frac{\min(s,6)}{6}$ if $3< s$,
\end{itemize}
and $\Gamma_{T, \|u_0\|_{H^{\frac{3}{4}}}}$ is defined by
\begin{equation*}
\Gamma_{T, \|u_0\|_{H^{\frac{3}{4}}}}=C\left[\Lambda_{T, \|u_0\|_{H^{\frac{3}{4}}}}+ \exp\left(\frac{T^{\frac{3}{4}}C_{\frac{3}{4}}e^{\kappa_{\frac{3}{4}}T}}{4}\|u_0\|_{H^{\frac{3}{4}}}\right)\right],
\end{equation*}
where $\Lambda_{T, \|u_0\|_{H^{\frac{3}{4}}}}$ is defined by \eqref{LAMBDA_0}, $C$ and $C_{\frac{3}{4}}$ are two constants and $\kappa_{\frac{3}{4}}$ depends only on $\|u_0\|_{L^{2}(\mathbb{R})}$.
In the error estimate above, $e^n$ is defined as in 
\eqref{EQ_8_barre_FRED_L}-\eqref{NOTATION_BAR}-\eqref{EQ_82_barre_FRED_L}. 
\label{MAIN_THEOREM_445454545454}
\end{Theorem}

If $u_0\in H^{m}(\mathbb{R})$ with $m\geq6$, then Theorem \ref{MAIN_THEOREM_445454545454} implies an order of convergence equal to 1 and we get back the result of Theorem \ref{THEOREM_MAIN111}.
Note that the results are valid for any $T>0$ in agreement with the fact that at this level of regularity we have global solutions
keeping their regularity. 

To prove Theorem \ref{THEOREM_MAIN111}, we prove consistency and stability of the scheme.
It is in the control of the consistency error that we need the exact solution to be smooth.
The most challenging part of the proof is the study of the stability of the scheme in order to take advantage
of the fact that the exact solution remains smooth on the whole $[0, T]$. The main idea is to transpose at the discrete
level the well-known weak-strong stability property for hyperbolic conservation laws that relies on 
a relative entropy estimate, see \cite{Dafermos_Livre} for a detailed presentation. This method is classical for the study of hyperbolic systems, see for exemple \cite{Cances_2013} for the numerical approximation of systems of conservation laws, \cite{Tzavaras_2005} for a relaxation hyperbolic system or \cite{Leger_Vasseur_2011} for the approximation of shocks and contact discontinuities. An important outcome 
of this approach is that in the stability estimate, 
the exponential amplification factor only involves the norm $\int_{0}^T \| \partial_{x} u(t,.) \|_{L^\infty}dt$ of the exact
solution, which is bounded thanks to the dispersive properties of the equation. This allows to get the convergence of the scheme on the full interval of time $[0, T]$
and also to handle less smooth initial data at the price of deteriorating the convergence order
as stated in Theorem \ref{MAIN_THEOREM_445454545454}. Indeed in order to prove Theorem \ref{MAIN_THEOREM_445454545454}, we replace the initial datum $u_{0}$ with a smoother one $u_{0}^\delta$
and just use the triangular inequality
$$ 
\|v_{\Delta}- u_{\Delta}\|_{L^\infty(0, T; L^2_{x})} \leq \|v_{\Delta} - u_{\Delta}^\delta\|_{L^\infty(0, T; L^2_{x})} 
+ \|u_{\Delta}^\delta - u_{\Delta}\|_{L^\infty(0, T; L^2_{x})},
$$
where $u_{\Delta}^\delta$ is the discretization of the exact solution $u^\delta$ of the KdV equation with initial 
datum $u_{0}^\delta$. We then use the stability in $L^2$ for exact solutions of the KdV equation and the stability estimate
of Theorem \ref{THEOREM_MAIN111}. The amplification factor $\int_{0}^T \| \partial_{x} u^\delta (t,.)\|_{L^\infty}dt$ is finite and can be bounded independently of $\delta$ as soon
as the initial datum is in $H^s(\mathbb{R})$, with $s\geq 3/4$ because of the Strichartz estimate that ensures that at this level
of regularity, the exact solution is actually also such that $\partial_{x} u \in L^4(0, T; L^\infty(\mathbb{R}))$.
We then end the proof by optimizing these estimates in terms of $\delta$ and $\Delta x$.
 
\begin{Remark}\label{estimstab}
We suppose $u_0\in H^s(\mathbb{R})$, with $s\geq \frac{3}{4}$ in Theorem \ref{MAIN_THEOREM_445454545454} because some difficulties are attached to get
a convergence rate for rough initial data.
If we are interested only in the convergence of the scheme (and not in the rate of convergence), 
it is well-known that we can construct weak solutions of KdV for $L^2$ initial data by a compactness
argument by using the Kato smoothing effect which writes
$$\int_{-T}^{T}\int_{-R}^{R}|\partial_xu(t,y)|^2dydt\leq C(T,R).$$ 
The convergence proof in \cite{Dutta_Koley_Risebro_2015} relies on a discrete analogous inequality
for the scheme. It is proved that the solution of the scheme satisfies for $L^2$ initial data :
$$\Delta t\sum_{n\Delta t\leq T}||\partial_xu^{n+1}||^2_{L^2(-R,R)}\leq C(||u^0||_{L^2(\mathbb{R})}, R) \text{,\ \ for\ \ }n\Delta t\leq T$$
and some compactness arguments allow to prove the convergence of the scheme.

%
In order to get a precise convergence rate, we need at the discrete level a counterpart of a quantitative stability estimate for two solutions namely an estimate under the form
\begin{equation}
\label{pfff} \|u-v \|_{L^\infty(0, T; L^2(\mathbb{R}))} \leq C(T, \|u\|_{X_{T}}, \|v\|_{X_{T}})\|u_{0}- v_{0}\|_{L^2(\mathbb{R})}
\end{equation}
where $u$, $v$ are two solutions of KdV and $X_{T}$ is some well chosen functional space. 
It is known that such an estimate is true for KdV for $L^2$ initial data for $X_{T}$ some well chosen 
Bourgain space (some more details will be given in Section \ref{sectioncontinue}).
These spaces are designed to capture in an optimal way all the dispersive information coming from the linear part.
The discrete counterpart of these spaces is at the moment unclear. Our approach here
relies on a discrete version of a non-symmetric form of \eqref{pfff} which reads
$$ 
\|u-v \|_{L^\infty(0, T; L^2(\mathbb{R}))} \leq C(T, \|\partial_{x}u\|_{L^1(0, T; L^\infty(\mathbb{R}))}) \|u_{0}- v_{0}\|_{L^2(\mathbb{R})}
$$
and is true if $v_{0} \in L^2$ and $u_{0} \in H^s$, $s \geq 3/4$ (again, we shall give more details
in Section \ref{sectioncontinue}).

 \end{Remark}

\paragraph{Outline of the paper} 
 In Section \ref{sectioncontinue}, we state precisely the results of the Cauchy theory of KdV that we shall
 use in this paper. Then, in Section \ref{SECTION_CONSISTENCY}, we analyze the consistency error of the scheme (postponing the more technical part to the Appendix \ref{APPENDIX_A}). The aim of Section \ref{SECTION_3} is to derive the crucial $\ell^2_{\Delta}$-stability inequality. We study the discrete equation verified by the convergence error and we obtain the $\ell^2_{\Delta}$ estimates, whose proof is detailed in Appendix \ref{PROOF_Stability_Ineq_10_avril}. Eventually, the rate of convergence is determined in Section \ref{SECTION_4}. \\
Section \ref{SECTION_NON_REGULIERE} is devoted to the study of the convergence rate for a non smooth solution. A convolution product by mollifiers enables us to counteract the lack of regularity. It requires several general approximation estimates between initial data and regularized initial data which are gathered in Subsection \ref{SUBSECTION_1_SECTION_NON_REGULIERE}. The proof of Theorem \ref{MAIN_THEOREM_445454545454} is developed in Subsection \ref{SUBSECTION_2_SECTION_NON_REGULIERE}. Some numerical 
results illustrate the theoretical rate of convergence in Section \ref{NUMERICAL_RESULTS}. 

\paragraph{Notation} Thereafter, the letter $C$ represents a positive number that may differ from line to line and that can be chosen independently of $\Delta t$, $\Delta x$, $u$, $u_0$, $T$ and $\delta$. We denote by $\kappa$ all numbers depending only on $\|u_0\|_{L^2(\mathbb{R})}$.

%
%
%
%

\section{Known results on the Cauchy problem for the KdV equation}
\label{sectioncontinue}

Let us recall the definition of Bourgain spaces. For $s \in \mathbb{R}$ and $b \geq 0$, 
a tempered distribution $u(t,x)$ on $\mathbb{R}^2$ is said to 
belong to $X^{s, b}$ if its following norm is finite 
\begin{equation*}
||u||_{X^{s, b}}=\left(\int_{\mathbb{R}}\int_{\mathbb{R}}\left(1+ 
|\xi|\right)^{2s}\left(1+| \tau-\xi^3|\right)^{2b}|\tilde{u}\left(\tau, \xi\right)|^2d\xi d\tau\right)^{\frac{1}{2}},
\end{equation*}
where $\tilde{u}$ is the space and time Fourier transform of $u$.
We shall also use a localized version of this space: $u \in X^{s,b}(I)$, where $I \subset \mathbb{R}$
is an interval, if $\|u\|_{X^{s,b}(I)}<+ \infty$, where
$$ 
\|u\|_{X^{s,b}(I)} = \inf\{\|\overline{u} \|_{X^{s,b}}, \, \overline{u}_{/I} = u \}.
$$
By using results from \cite{Kenig_1991}, \cite{Bourgain_1993}, \cite{Kenig_1993}, see for example the book \cite{Linares_Ponce_2015}, 
we get the following theorem.
\begin{Theorem}
\label{theobourgain}
Consider $s \geq 0$, $1>b>1/2$. There exists a unique global solution $u$ of \eqref{EQ_INIT}-\eqref{DONNEE_INIT}, with $u_0\in H^s(\mathbb{R})$, such that for every $T \geq 0$, $u\in \mathcal{C}([0, T]; H^{s}(\mathbb{R})) \cap X^{s,b}([0, T])$.
 Moreover, there exists $\kappa_{s}>0$, depending only on $s$ and on the norm $\|u_{0}\|_{L^2}$, and $C_{ s}>0$, depending only on $s$, such that, for any $T\geq0$,
\begin{align*}
&\bullet\ \sup_{t\in[0, T]}\|u(t)\|_{H^s(\mathbb{R})} \leq C_{s} \|u_{0}\|_{H^s(\mathbb{R})} e^{ \kappa_{s} T}, \\
&\bullet\ \mbox{if\ } s\geq\frac{3}{4}, \quad \| \partial_{x} u \|_{L^i(0, T; L^\infty(\mathbb{R}))} \leq T^{\frac{4-i}{4i}} \|u_{0}\|_{H^{3 \over 4}(\mathbb{R})} C_{\frac{3}{4}} e^{ \kappa_{\frac{3}{4}} T},\quad \mbox{for
 i } \in\{1,2\}.
\end{align*}
\end{Theorem}
The growth rate in the above estimates is not optimal. 
 
Note that a local well-posedness result for $s>3/4$ follows directly from \cite{Kenig_1991}.
In the present paper, we will be only interested in $s \geq 3/4$, 
nevertheless, to get the global well-posedness for $s \in [3/4, 1)$, we need to go through the $L^2$ local
well-posedness result. 

\begin{proof}
Let us just briefly explain how we can 
organize now classical arguments to get the result. We refer for example to \cite{Kenig_1993}, \cite{Linares_Ponce_2015} 
for the details. 
The existence is proven by a fixed point argument on the following truncated problem:
$$
v\mapsto F(v)
$$
such that
$$ 
F(v)(t)= \chi(t)e^{- t \partial_{x}^3} u_{0} - \chi(t)\int_{0}^t e^{-(t- \tau) \partial_{x}^3 }
\partial_{x} \left( \chi\left({\tau \over \delta}\right) \frac{v^2}{2}(\tau)\right)\, d\tau,
$$
where $\chi$ is a smooth compactly supported function taking its values in $[0, 1]$ that is equal to $1$ on $[-1, 1]$
and supported in $[-2, 2]$. For $|t| \leq \delta \leq 1/2 $, a fixed point of the above equation
is a solution of the original Cauchy problem, denoted by $u$. 

To see that there exists such a fixed point, fix $C>0$, that does not depend on $u_0$, 
such that $$ \|\chi(t)e^{- t \partial_{x}^3} u_{0}\|_{X^{0, b}} \leq C \|u_{0}\|_{L^2}.$$
We can first prove that $F$ is a contraction on a suitable ball of $X^{0, b}$, provided $8 C^2 \|u_{0}\|_{L^2} \delta^\beta \leq 1$ 
for some $\beta >0$ (that is related to $1>b>1/2$) that does not depend on $\delta$ nor $u_0$.
In particular, for the fixed point, denoted by $v$, we can ensure that 
$$
\|v\|_{X^{0,b}} \leq 2 C \|u_{0}\|_{L^2}.
$$
 
Next, by using again the Duhamel formula, we can obtain, for $s \geq 0$, $$ \|v\|_{X^{s,b}} \leq c_{s} \|u_{0}\|_{H^s} + c_{s}\delta^\beta \|v\|_{X^{0,b}} \|v\|_{X^{s,b}} 
\leq c_{s} \|u_{0}\|_{H^s} + 2 c_{s} C \|u_{0}\|_{L^2} \delta^\beta \|v\|_{X^{s,b}},$$
where $c_{s}$ depends only on $s$. In particular, by choosing $\delta$, possibly smaller than previously, such that 
$ 2 c_{s} C \|u_{0}\|_{L^2} \delta^\beta \leq 1/2$, 
we thus obtain that
$$ 
\|v\|_{X^{s,b}} \leq 2 c_{s} \|u_{0}\|_{H^s}.
$$
Next, by using that the $X^{s,b}$ norm for $b>1/2$ controls the $\mathcal{C}(\mathbb{R}, H^s)$ norm (see for example
\cite{Tao_Livre_2006} 
lemma 2.9 page 100), we obtain that
$$ 
\|v\|_{\mathcal{C}([0, \delta]; H^s(\mathbb{R}))} \leq \|v\|_{\mathcal{C}(\mathbb{R}; H^s(\mathbb{R}))} 
 \leq B_{s} \|u_0\|_{H^s(\mathbb{R})},
 $$
where $B_{s}$ depends only on $s$.
Since the existence time $\delta$ depends only on the $L^2$-norm of the initial datum and that
the $L^2$-norm is conserved for the KdV equation, we can iterate the above argument
to get a global solution (thus denoted by $u$). 
Moreover, in a quantitative way, by choosing $n = \lfloor T/\delta \rfloor +1$ and iterating $n$ times, we obtain that 
$$ 
\ \|u\|_{\mathcal{C}([0, T]; H^s)} + \|u\|_{X^{s,b}[0, T]} \leq B_{s}^n \|u_{0}\|_{H^s} \leq C_{s} \|u_{0}\|_{H^s} e^{ \kappa_{s} T},
$$
where $\kappa_{s}$ depends only on $s$ and $\|u_{0}\|_{L^2}$ while $C_{s}$ depends only on $s$.
 
Finally, since the Strichartz estimate in the KdV context (see \cite{Kenig_1991}) reads
$$ 
\| |\partial_{x}|^{1 \over 4}e^{-t \partial_{x}^3} u_{0}\|_{L^4_{t}(\mathbb{R}; L^\infty_{x})}
\leq C \|u_{0}\|_{L^2},
$$
by using the embedding properties of the Bourgain spaces (see again \cite{Tao_Livre_2006} lemma 2.9 page 100),
we obtain that
$$
\| \partial_{x} u \|_{L^4_{t}([0, \delta]; L^\infty_{x})} \leq \| \partial_{x} v\|_{L^4_{t}(\mathbb{R}; L^\infty_{x})}
\leq \| v\|_{X^{{3\over 4}, b}} \leq C \|u_{0}\|_{H^{3 \over 4}}.
$$
Again by iterating this estimate, we finally obtain that
$$ 
\| \partial_{x} u \|_{L^4_{t}(0, T; L^\infty_{x})} \leq C_{\frac{3}{4}} \|u_{0}\|_{H^{3 \over 4} } e^{\kappa_{\frac{3}{4}} T}
$$
and the desired estimate follows from the H\"older inequality.
\end{proof}
 
\section{Consistency error estimate}
\label{SECTION_CONSISTENCY}

This section is devoted to the computation of the consistency error defined by Equation \eqref{CONSISTENCY_DEF}.
As a starting point, by using Theorem \ref{theobourgain}, we obtain
the following estimates on the averaged solution $u_{\Delta}$.
\begin{Lemma}
Let $u$ be the exact solution of \eqref{EQ_INIT}-\eqref{DONNEE_INIT} from $u_0\in H^{s}(\mathbb{R})$, $s>\frac{1}{2}$ and $u_{\Delta}$ be defined by \eqref{NOTATION_BAR}. Then there exists $C>0$, depending only on $s$, and $\kappa_s>0$, depending only on $s$ and $\|u_0\|_{L^2}$, such that, for any $T\geq 0$ and any $n\in\llbracket0,N\rrbracket$ with $N=\lfloor \frac{T}{\Delta t}\rfloor,$

\begin{equation*}
\hspace*{-10.9cm}\bullet\ ||\left(u_{\Delta}\right)^n||_{\ell^{\infty}}\leq Ce^{\kappa_{s}T}\|u_0\|_{H^{s}},
\label{EQ_33_NOUVELLE_VERSION}
\end{equation*}
\begin{equation}
\bullet\ \text{if\ }s\geq\frac{3}{4},\ \Delta t||D_+\left(u_{\Delta}\right)^n||^i_{\ell^{\infty}}\leq \int_{t^n}^{t^{n+1}}||\partial_xu(s,.)||^i_{L_x^{\infty}}ds\leq T^{\frac{4-i}{4i}}Ce^{\kappa_{\frac{3}{4}}T}\|u_0\|_{H^{\frac{3}{4}}(\mathbb{R})}, \ \ \ \textit{for\ } i\in\{1,2\}.
\label{BORNE_2222222222}
\end{equation}
\label{LEMMA_F_G}
\end{Lemma}
\begin{proof}
The Sobolev embedding $H^{s}(\mathbb{R})\hookrightarrow L^{\infty}(\mathbb{R})$, for $s>\frac{1}{2}$ yields the inequality
\begin{equation*}
\begin{split}
||\left(u_{\Delta}\right)^n||_{\ell^{\infty}}\leq \frac{1}{\Delta t}\int_{t^n}^{t^{n+1}}||u(t,.)||_{L^{\infty}(\mathbb{R})}dt\leq C \underset{t\in[0,T]}{\mathrm{sup}}||u(t,.)||_{H^s(\mathbb{R})}.
\end{split}
\end{equation*}
Theorem \ref{theobourgain} implies
\begin{equation*}
\begin{split}
||\left(u_{\Delta}\right)^n||_{\ell^{\infty}}\leq CC_{s}\|u_0\|_{H^s(\mathbb{R})}e^{\kappa_{s}T},
\end{split}
\end{equation*}
which proves the first estimate of Lemma \ref{LEMMA_F_G}.

\noindent To prove  \eqref{BORNE_2222222222} for $i=1$, we use a Taylor expansion:
\begin{equation*}
\begin{split}
\Delta t\left|\left|D_+\left(u_{\Delta}\right)^n\right|\right|_{\ell^{\infty}}&= \Delta t\left|\left|\frac{1}{\Delta t\Delta x^2}\int_{t^n}^{t^{n+1}}\int_{x_j}^{x_{j+1}}u(s,y+\Delta x)-u(s,y)dyds\right|\right|_{\ell^{\infty}}
\leq \int_{t^n}^{t^{n+1}}\left|\left|\partial_x u(s,.)\right|\right|_{L^{\infty}_x}ds.
\end{split}
\end{equation*}
For $i=2$, the same Taylor expansion gives, thanks to the Cauchy-Schwarz inequality,
\begin{equation*}
\begin{split}
\Delta t\left|\left|D_+\left(u_{\Delta}\right)^n\right|\right|^2_{\ell^{\infty}}&=\Delta t\left|\left|\frac{1}{\Delta x^2\Delta t}\int_{t^n}^{t^{n+1}}\int_{x_j}^{x_{j+1}}\int_{y}^{y+\Delta x}\partial_xu(s,z)dzdyds\right|\right|^2_{\ell^{\infty}}
\leq \int_{t^n}^{t^{n+1}}\left|\left|\partial_x u(s,.)\right|\right|^2_{L^{\infty}_x}ds.
\end{split}
\end{equation*}
Theorem \ref{theobourgain} concludes the proof.

\end{proof}

\begin{Remark}
The Sobolev regularity of the initial datum is at least $H^{\frac{3}{4}}(\mathbb{R})$ in Theorem \ref{MAIN_THEOREM_445454545454} because we need to control $\int_{0}^T||\partial_xu(t,.)||^i_{L^{\infty}(\mathbb{R})}dt$, for $i\in\{1, 2\}$ in some of the proofs. This is explicitly needed in Lemma \ref{LEMMA_F_G}, Theorem \ref{theobourgain} and in the definition of $\Lambda_{T, \|u_0\|_{\frac{3}{4}}}$ in \eqref{LAMBDA_0}.
\end{Remark}

As a consequence, we control the $\ell^2_{\Delta}$-norm of the consistency error $\epsilon^n$ defined in \eqref{CONSISTENCY_DEF} in terms of the initial datum thanks to the following proposition.
\begin{Proposition}
\label{propositionun}

Let $s \geq 6$ and $\eta\in(0,s-\frac{3}{2}]$. There exists $C > 0$ such that, for any $u_0\in H^s(\mathbb{R})$ there exists $\kappa>0$, depending only on $\|u_0\|_{L^2}$, such that for any $T \geq 0$ one has 
\begin{equation}
||\epsilon^n||_{\ell^{\infty}(\llbracket0,N\rrbracket;\ell^2_{\Delta}(\mathbb{Z}))}\leq Ce^{\kappa T}\left(1+||u_0||^2_{H^{\frac{1}{2}+\eta}}\right)\left\{\Delta t\ ||u_0||_{H^{6}}+\Delta x\ \left[||u_0||_{H^{4}}+||u_0||_{H^{\frac{3}{2}+\eta}}||u_0||_{H^1}\right]\right\}. 
\label{BORNE_33333333333}
\end{equation}
\label{LEMMA_ON_ESPILON_AAAAAA}
\end{Proposition}

The proof is postponed until  Appendix \ref{APPENDIX_A}.

%
%
%
%

\section{Stability estimate}
\label{SECTION_3}

The stability property will be proved in stating a discrete weak-strong stability type inequality : Equation \eqref{latest_ineq} in the following. 
This inequality gives an upper bound of the convergence error at time $n+1$ with respect to the convergence error at time $n$. Note however that this estimate is not totally usable in this form, as it involves, on the right-hand term, {\em derivatives} of the convergence error at time $n$. 
This will be made more explicit in Section \ref{SECTION_4}.

\subsection{Preliminary results}
\label{SECTION_2}

We here collect some discrete "Leibniz's rules" (Lemma \ref{FIIIRST_LEEEMA}), $\ell^2$-norm identities (Lemma \ref{LEMMA_SWITCH}) and discrete integrations by parts formulas (Lemma \ref{LEMMA_BASIC}) which will be used in Subsection \ref{SUBSECTION_32}. 
As they are classical and quite simple, we here ommit their proofs. 
\begin{Lemma}
Let $(a_j)_{j\in \mathbb{Z}}$ and $(b_j)_{j\in \mathbb{Z}}$ be two sequences and let $D,\ D_+,\ D_-$ be the discrete operators defined in \eqref{EQ_DEF_OPERATOR_22}. One has, for any $j\in\mathbb{Z}$:
\begin{equation}
\hspace*{-4.71cm}\bullet \hspace*{3.7cm}
D_+D_-\left(a\right)_j=D_-D_+\left(a\right)_j,\label{EQ_1_LEMMA_BASIC}
\end{equation}
\begin{subequations}
\begin{numcases}{\hspace*{-4.2cm}\bullet\hspace*{2cm}}
D_+\left(ab\right)_j=a_{j+1}D_+\left(b\right)_j+b_jD_+\left(a\right)_j,\label{EQ_UUTTIILLEESS_11}\\
D_-\left(ab\right)_j=a_{j-1}D_-\left(b\right)_j+b_jD_-\left(a\right)_j.\label{EQ_UUTTIILLEESS_22}
\end{numcases}
\end{subequations}
\begin{equation}
\hspace*{-4.1cm}\bullet \hspace*{3.2cm}
D(ab)_j=D(a)_jb_{j+1}+a_{j-1}D(b)_j, \label{EQ_2_LEMMA_BASIC}
\end{equation}
\begin{equation}
\hspace*{-2.75cm}\bullet \hspace*{2.3cm}
D(ab)_j=b_jD(a)_j+\frac{a_{j+1}}{2}D_+(b)_j+\frac{a_{j-1}}{2}D_-(b)_j, \label{EQ_2bis_LEMMA_BASIC}
\end{equation}
\begin{subequations}
\begin{numcases}{\hspace*{-3.5cm}\bullet \hspace*{1.8cm}}
a_jD_+\left(a\right)_j=\frac{1}{2}D_+\left(a^2\right)_j-\frac{\Delta x}{2}\left(D_+\left(a\right)_j\right)^2,\label{EQ_A_1_A_1_A_1_A_1_A_1_A}\\
a_jD_-\left(a\right)_j=\frac{1}{2}D_-\left(a^2\right)_j+\frac{\Delta x}{2}\left(D_-\left(a\right)_j\right)^2.\label{EQ_B_2_B_2_B_2_B_2_B_2_B}
\end{numcases}
\end{subequations}
\label{FIIIRST_LEEEMA}
\end{Lemma}

\begin{Lemma}
For $\left(a_j\right)_{j\in\mathbb{Z}}$ a sequence in $\ell_{\Delta}^2(\mathbb{Z})$, one has
\begin{equation}
\hspace*{-3.95cm}\bullet \hspace*{4.2cm}
\left|\left|D_+\left(a\right)\right|\right|_{\ell^2_{\Delta}}=\left|\left| D_-\left(a\right)\right|\right|_{\ell^2_{\Delta}},
\label{EQ_D_+_egal_D_-}
\end{equation}
\begin{equation}
\hspace*{-2.75cm}\bullet \hspace*{2.85cm}
\left|\left| D\left(\frac{a^2}{2}\right)\right|\right|_{\ell^2_{\Delta}}=\left|\left| D\left(a\right)\left(\frac{\mathcal{S}^{+}a+\mathcal{S}^{-}a}{2}\right)\right|\right|_{\ell^2_{\Delta}},
\label{EQ_3_LEMMA_3}
\end{equation}
\begin{equation}
\hspace*{-2.23cm}\bullet \hspace*{1.83cm}
\left|\left|D_+D_-\left(a\right)\right|\right|_{\ell^2_{\Delta}}^2=\frac{4}{\Delta x^2}\left|\left|D_+\left(a\right)\right|\right|_{\ell^2_{\Delta}}^2-\frac{4}{\Delta x^2}\left|\left|D\left(a\right)\right|\right|_{\ell^2_{\Delta}}^2.
\label{EQ_2_LEMMA_3}
\end{equation}
\label{LEMMA_SWITCH}
\end{Lemma}
Applying \eqref{EQ_2_LEMMA_3} to $D_+\left(a\right)_j$ rather than $a_j$ enables to state
\begin{Corollary}Let $\left(a_j\right)_{j\in\mathbb{Z}}$ be a sequence in $\ell_{\Delta}^2(\mathbb{Z})$. One has
\begin{equation}
\left|\left|D_+D_+D_-\left(a\right)\right|\right|_{\ell^2_{\Delta}}^2 =\frac{4}{\Delta x^2}\left|\left| D_+D_-\left(a\right)\right|\right|_{\ell^2_{\Delta}}^2-\frac{4}{\Delta x^2}\left|\left|D_+D\left(a\right)\right|\right|_{\ell^2_{\Delta}}^2.
\label{EQ_2_LEMMA_4}
\end{equation}
\end{Corollary}

\begin{Lemma}
Let $\left(a_j\right)_{j\in\mathbb{Z}}$ and $\left(b_j\right)_{j\in\mathbb{Z}}$ be two sequences in $\ell_{\Delta}^2(\mathbb{Z})$. One has
\begin{equation}
\hspace*{-4.43cm}\bullet \hspace*{4.13cm}
\left\langle D_+\left(a\right), b\right\rangle=-\left\langle a, D_-\left(b\right)\right\rangle,\label{EQ_3_LEMMA_BASIC}
\end{equation}
\begin{equation}
\hspace*{-4.65cm}\bullet \hspace*{4.35cm}
\left\langle D\left(a\right), b\right\rangle=-\left\langle a, D\left(b\right)\right\rangle,\label{EQ_4_LEMMA_BASIC}
\end{equation}
\begin{equation}
\hspace*{-4.3cm}\bullet \hspace*{3.95cm}
\left\langle a, D_+\left(a\right)\right\rangle=-\frac{\Delta x}{2}\left|\left| D_+\left(a\right)\right|\right|_{\ell^2_{\Delta}}^2,
\label{EQ_5_LEMMA_BASIC}
\end{equation}
\begin{equation}
\hspace*{-3.2cm}\bullet \hspace*{2.5cm}
\left\langle D_+\left(a\right), a\mathcal{S}^{+}a\right\rangle=-\frac{\Delta x^2}{3}\left\langle D_+\left(a\right), \left(D_+(a)\right)^2\right\rangle,\label{EQ_6_LEMMA_BASIC}
\end{equation}
\begin{equation}
\hspace*{-3.2cm}\bullet \hspace*{2.6cm}
\left\langle D\left(a\right), \mathcal{S}^{-}a\mathcal{S}^{+}a\right\rangle=-\frac{4\Delta x^2}{3}\left\langle D\left(a\right), \left(D\left(a\right)\right)^2\right\rangle,\label{EQ_7_LEMMA_BASIC}
\end{equation}
\label{LEMMA_BASIC}
\begin{equation}
\hspace*{-4.25cm}\bullet \hspace*{4.05cm}
\left\langle a, D\left(ab\right)\right\rangle=\left\langle D_+\left(b\right), \frac{a\mathcal{S}^{+}a}{2}\right\rangle,
\label{EQ_1_LEMMA_1}
\end{equation}
\begin{equation}
\hspace*{-1.3cm}\bullet \hspace*{0.9cm}
\left\langle D_+D_-\left(a\right), D\left(ab\right)\right\rangle=-\frac{1}{\Delta x^2}\left\langle D_+\left(b\right), a\mathcal{S}^{+}a\right\rangle+\frac{1}{\Delta x^2}\left\langle D\left(b\right), \mathcal{S}^{-}a\mathcal{S}^{+}a\right\rangle.
\label{EQ_2_LEMMA_1}
\end{equation}
\label{FIIIRST_PROPR}
\end{Lemma}
With \eqref{EQ_6_LEMMA_BASIC} and \eqref{EQ_7_LEMMA_BASIC}, taking $(b)_{j\in\mathbb{Z}}=(\frac{a_j}{2})_{j\in\mathbb{Z}}$ in \eqref{EQ_1_LEMMA_1} and \eqref{EQ_2_LEMMA_1} gives the following corollary.

\begin{Corollary}
Let $\left(a_j\right)_{j\in\mathbb{Z}}$ be a sequence in $\ell_{\Delta}^2(\mathbb{Z})$. One has
\begin{equation}
\hspace*{-3.9cm}\bullet \hspace*{3.5cm}
\left\langle a, D\left(\frac{a^2}{2}\right)\right\rangle=-\frac{\Delta x^2}{12}\left\langle D_+\left(a\right), \left(D_+\left(a\right)\right)^2\right\rangle,
\label{EQ_0_LEMMA_3}
\end{equation}
\begin{equation}
\hspace*{-1.9cm}\bullet \hspace*{1.4cm}
\left\langle D\left(\frac{a^2}{2}\right), D_+D_-\left(a\right)\right\rangle=\frac{1}{6}\left\langle D_+\left(a\right), \left(D_+\left(a\right)\right)^2\right\rangle-\frac{2}{3}\left\langle D\left(a\right), \left(D\left(a\right)\right)^2\right\rangle.
\label{EQ_1_LEMMA_3}
\end{equation}
\label{LEMMA_3}
\end{Corollary}

%
%

\subsection{The $\ell^2_{\Delta}$-stability inequality }
\label{SUBSECTION_32}
We focus on the derivation of the $\ell^2_{\Delta}$-stability inequality \eqref{latest_ineq}, which corresponds to a discrete 
weak-strong estimate.\\
Combining \eqref{EQ_SCHEME_2}, \eqref{EQ_8_barre_FRED_L} and \eqref{CONSISTENCY_DEF}, we obtain 
\begin{align}
&e_{j}^{n+1}+\theta \Delta tD_+D_+D_-\left(e\right)_j^{n+1}\label{SCHEMA_REMANIE}\\
&=e_j^n-(1-\theta)\Delta tD_+D_+D_-\left(e\right)_j^n-\Delta tD\left(\frac{e^2}{2}\right)_j^n-\Delta tD\left(u_{\Delta}e\right)_j^n+\frac{c\Delta x\Delta t}{2}D_+D_-\left(e\right)_j^n-\Delta t\epsilon_j^n, \ (n,j)\in\llbracket0,N\rrbracket\times\mathbb{Z}.\nonumber
\end{align}

\begin{Definition}For more simplicity, we denote by $\mathcal{A}_{\theta}$ the dispersive operator
\begin{equation}
\mathcal{A}_{\theta}=I+\theta \Delta tD_+D_+D_-,
\label{DEF_opA}
\end{equation}
where $I$ is the identity operator in $\ell^2_{\Delta}(\mathbb{Z})$.
\end{Definition}
\begin{Proposition}[$\ell^2_{\Delta}$-stability inequality]
Let $(e_j^n)_{(j,n)}$ be the convergence error defined by \eqref{EQ_8_barre_FRED_L} with respect to Scheme \eqref{EQ_SCHEME}-\eqref{EQ_33333_bis_bis}. For every $\theta\in[0,1], \Delta t>0$ and $\Delta x>0$, for every $(n,j)\in\llbracket0,N\rrbracket\times\mathbb{Z}$ and $\gamma \in[0,\frac{1}{2})$ and $\sigma\in\{0,1\}$, one has
\begin{small}
\begin{equation}
\begin{split}
&\left|\left|\mathcal{A}_{\theta}e^{n+1}\right|\right|_{\ell^{2}_{\Delta}}^2\leq \left|\left|\mathcal{A}_{\theta}e^n\right|\right|_{\ell^{2}_{\Delta}}^2+\Delta tA_{a}||e^n||^2_{\ell^2_{\Delta}}+\Delta t\left|\left|\mathcal{A}_{-(1-\theta)}e^n\right|\right|_{\ell^{2}_{\Delta}}^2+\Delta t||\epsilon^n||^2_{\ell^2_{\Delta}}\left\{1+4\frac{\Delta t}{\Delta x}+\Delta t\right\}\\
&+\Delta t\left\langle A_{b}, \left[D_+\left(e\right)^n\right]^2\right\rangle+\Delta t^2A_{c}\left|\left|D\left(e\right)^n\right|\right|_{\ell^{2}_{\Delta}}^2+\Delta tA_{d}\left|\left|D_+D_-\left(e\right)^n\right|\right|_{\ell^{2}_{\Delta}}^2+\Delta tA_{e}\left|\left|D_+D\left(e\right)^n\right|\right|_{\ell^{2}_{\Delta}}^2+\Delta tA_{f}\left|\left|D_+D_+D_-\left(e\right)^n\right|\right|_{\ell^{2}_{\Delta}}^2,
\end{split}
\label{latest_ineq}
\end{equation}
\end{small}
where the coefficients $A_{i}$, for $i \in \{a, b, c, d, e, f\}$, are defined in Equations \eqref{Cf}-\eqref{Cb}. 
\begin{subequations}
\begin{align}
A_{a}&=||u_{\Delta}^n||_{\ell^{\infty}}^2+||D_+\left(u_{\Delta}\right)^n||_{\ell^{\infty}}\left(2-\theta+\frac{\Delta t}{\Delta x}\left[2c+\frac{2}{3}||e^n||_{\ell^{\infty}}+\frac{3}{2}||(u_{\Delta})^n||_{\ell^{\infty}}\right]\right)\nonumber\\
&\hspace*{6cm}+\frac{\Delta t^2}{\Delta x^2}||D_+(u_{\Delta})^n||_{\ell^{\infty}}^2+\frac{\Delta t}{\Delta x}(||u_{\Delta}^n||_{\ell^{\infty}}^2+2c^2),\label{Cf}\\
A_{b}&=\left(\frac{\Delta x}{6}D_+\left(e\right)^n-c\boldsymbol{1}\right)\left(\Delta x-c\Delta t\right)+(1-\theta)\Delta t||D_+\left(u_{\Delta}\right)^n||_{\ell^{\infty}}^{2-\sigma}\boldsymbol{1},
\label{Ca}\\
&\hspace*{-1.4cm}\text{with\ }\boldsymbol{1}=(1, 1, 1, ...),\nonumber\\A_{c}&=||e^n||_{\ell^{\infty}}^2\left[1+\Delta x\right]+||(u_{\Delta})^n||_{\ell^{\infty}}^2-c^2+2||e^n||_{\ell^{\infty}}||(u_{\Delta})^n||_{\ell^{\infty}}+\frac{2c}{3}||e^n||_{\ell^{\infty}},
\label{Cc}\\
A_{d}&=(1-\theta)\Delta t\left[||D_+(u_{\Delta})^n||_{\ell^{\infty}}^{\sigma}+\frac{\Delta x}{2}||D_-\left(u_{\Delta}\right)^n||_{\ell^{\infty}}\right],
\label{Cd}\\
A_{e}&=2(1-\theta)\Delta t\left\{||\left(u_{\Delta}\right)^n||_{\ell^{\infty}}+||e^n||_{\ell^{\infty}}+\left[\frac{\Delta x^{\frac{1}{2}-\gamma}+||e^n||_{\ell^{\infty}}+9||e^n||_{\ell^{\infty}}^2\Delta x^{\gamma-\frac{1}{2}}}{2}\right]\right\}-\Delta x,
\label{Ce}\\
A_{f}&=\Delta t\left\{(1-2\theta)+\frac{(1-\theta)\Delta x^2}{2}\left[c+\frac{\Delta x^{\frac{1}{2}-\gamma}+||e^n||_{\ell^{\infty}}+9||e^n||_{\ell^{\infty}}^2\Delta x^{\gamma-\frac{1}{2}}}{2}\right]\right.\nonumber\\
&\left.\hspace*{7cm}+\Delta t(1-\theta)||D_+\left(u_{\Delta}\right)^n||_{\ell^{\infty}}\right\}-\frac{\Delta x^3}{4}.
\label{Cb}
\end{align}
\end{subequations}
\label{PROP_DISCRETE_RELATIVE_ENTROPY_INEQ}
\end{Proposition}

\begin{Remark} One of our purposes, here below, will be to control the right-hand side terms $A_{i}$ with $i\in\{b, c, d, e, f\}$ only in terms of $u_{\Delta}$ and not $v$. This is why this inequality can be viewed as a weak-strong inequality.
\end{Remark}

The proof of Proposition \ref{PROP_DISCRETE_RELATIVE_ENTROPY_INEQ} is detailed in Appendix \ref{PROOF_Stability_Ineq_10_avril}.

%
%
%
%
\section{Rate of convergence}
\label{SECTION_4}

\indent In the left-hand side of the $\ell^2_{\Delta}$-stability inequality \eqref{latest_ineq}, $e_j^{n+1}$ appears in the operator $\mathcal{A}_{\theta}$. The study of this dispersive operator is the aim of Subsection \ref{SUBSECTION_40}. \\
\indent In the right-hand side of \eqref{latest_ineq}, $D_+(e)_j^n$, $D_+D_-(e)_j^n$ appear in factor of some terms $A_i$. Since we have no control on these derivatives of the convergence error, we reorganize terms $A_{i}$ in Subsection \ref{_Simplification_of_Inequality_} to obtain non-positive terms : the $B_{i}$ and $C_{i}$ terms of Corollaries \ref{Cor_Grand} and \ref{Cor_petit}. \\
\indent In Subsection \ref{SUBSECTION_42}, the correct CFL hypothesis enables to cancel extra terms $B_{i}$ and $C_{i}$ and an induction method concludes the convergence proof.

\subsection{Properties of the operator $\mathcal{A}_{\theta}$}
\label{SUBSECTION_40}

\begin{Proposition}\label{invertible}
For every $\Delta t>0$ and $\Delta x>0$, $\mathcal{A}_{\theta}$ is 
\begin{itemize}
\item continuous (with a norm depending on $\frac{\Delta t}{\Delta x^3}$) from $\ell^2_{\Delta }$ to $\ell^2_{\Delta }$,
\item invertible. 
\end{itemize}
Moreover, one has the following inequalities, for any sequence $ \left(a_j\right)_{j\in\mathbb{Z}}\in\ell^2_{\Delta}(\mathbb{Z})$

\begin{equation}
||a||^2_{\ell^2_{\Delta}}\leq ||\mathcal{A}_{\theta}a||_{\ell^2_{\Delta}}^2\leq \left\{1+\frac{16\theta\Delta t}{\Delta x^3}\left[1+\frac{4\theta\Delta t}{\Delta x^3}\right]\right\}||a||_{\ell^2_{\Delta}}^2.
\label{CONTINOUS_INVERSE}
\end{equation}
\label{PROP_CONTINOUS_INVERSE}
\end{Proposition}
\begin{Remark}
Inequality \eqref{CONTINOUS_INVERSE} implies that the inverse of $\mathcal{A}_{\theta}$ is continuous from $\ell^2_{\Delta }$ to $\ell^2_{\Delta }$ with a norm independent of $\frac{\Delta t}{\Delta x^3}$. 
\end{Remark}
\begin{proof}
Given $a\in\ell^{2}_{\Delta}\left(\mathbb{Z}\right)$, we may define the function $\widehat{a}\in L^2\left(0,1\right)$ by
\begin{equation*}
\widehat{a}\left(\xi\right)=\sum_{k\in\mathbb{Z}}a_ke^{2i\pi k\xi}, \ \ \ \xi\in(0,1),
\end{equation*}
(the sequence $a$ is seen as the Fourier-series of the function $\widehat{a}$). Parseval identity yields
\begin{equation}
\sum_{j\in\mathbb{Z}}\Delta x|a_j|^2=\Delta x\int_0^1|\widehat{a}\left(\xi\right)|^2d\xi.
\label{PARSEVAL}
\end{equation}
We extend the shift operators $\mathcal{S}^{\pm}$ and define furthermore the general shift operator $\mathcal{S}^{\ell}$ with $\ell\in\mathbb{Z}$ by
\begin{equation*}
\mathcal{S}^{\ell}a=\left(a_{j+\ell}\right)_{j\in\mathbb{Z}},
\end{equation*}
the associated function verifies
\begin{equation*}
\widehat{\mathcal{S}^{\ell}a}\left(\xi\right)=e^{-2i\pi\ell\xi}\widehat{a}\left(\xi\right), \ \ \xi \in(0,1).
\end{equation*}

The function associated to $\mathcal{A}_{\theta}a$ is
\begin{equation*}
\begin{split}
\widehat{\mathcal{A}_{\theta}a}\left(\xi\right)&=\widehat{a}+\theta\frac{\Delta t}{\Delta x^3}\widehat{a}\left(e^{-4i\pi\xi}-3e^{-2i\pi\xi}+3-e^{2i\pi\xi}\right),\ \ \ \xi\in(0,1),\\
&=\widehat{a}\left\{1+\theta\frac{\Delta t}{\Delta x^3}\left[-2ie^{-i\pi\xi}\sin\left(3\pi\xi\right)+6ie^{-i\pi\xi}\sin\left(\pi\xi\right)\right]\right\}, \ \ \ \xi\in(0,1).
\end{split}
\end{equation*}
As $\sin\left(3\pi\xi\right)=3\sin\left(\pi\xi\right)-4\sin^3\left(\pi\xi\right)$, we obtain
\begin{equation*}
\begin{split}
\widehat{\mathcal{A}_{\theta}a}\left(\xi\right)&=\widehat{a}\left\{1+8i\theta\frac{\Delta t}{\Delta x^3}e^{-i\pi\xi}\sin^3\left(\pi\xi\right)\right\}.
\end{split}
\end{equation*}
The operator $\mathcal{A}_{\theta}$ is thus inversible and its inverse is defined by $\widehat{\mathcal{A}^{-1}_{\theta}a}(\xi)=\frac{1}{1+8i\theta\frac{\Delta t}{\Delta x^3}e^{-i\pi\xi}\sin^3\left(\pi\xi\right)}\widehat{a}(\xi).$\\
Moreover, this operator and its inverse are continuous since
\begin{equation*}
||\mathcal{A}_{\theta}a||_{\ell^2_{\Delta }}^2=\Delta x\int_{0}^1\left|1+8i\theta\frac{\Delta t}{\Delta x^3}e^{-i\pi\xi}\sin^3\left(\pi\xi\right)\right|^2|\widehat{a}(\xi)|^2d\xi,
\end{equation*}
and the module $\left|1+8i\theta\frac{\Delta t}{\Delta x^3}e^{-i\pi\xi}\sin^3\left(\pi\xi\right)\right|^2$ satisfies
\begin{equation*}
\begin{split}
\left|1+8i\theta\frac{\Delta t}{\Delta x^3}e^{-i\pi\xi}\sin^3\left(\pi\xi\right)\right|^2&=\left(1+8\theta\frac{\Delta t}{\Delta x^3}\sin^4\left(\pi\xi\right)\right)^2+\left(8\theta\frac{\Delta t}{\Delta x^3}\cos\left(\pi\xi\right)\sin^3\left(\pi\xi\right)\right)^2\\
&=1+16\theta\frac{\Delta t}{\Delta x^3}\sin^4\left(\pi\xi\right)\left(1+4\theta\frac{\Delta t}{\Delta x^3}\sin^2\left(\pi\xi\right)\right)\\
&\in[1,1+16\theta\frac{\Delta t}{\Delta x^3}\left(1+4\theta\frac{\Delta t}{\Delta x^3}\right)].
\end{split}
\end{equation*}
Thus, the operator $\mathcal{A}_{\theta}$ verifies
\begin{equation*}
\Delta x\int_0^1|\widehat{a}(\xi)|^2d\xi\leq ||\mathcal{A}_{\theta}a||_{\ell^2_{\Delta }}^2\leq \left\{1+16\theta\frac{\Delta t}{\Delta x^3}\left(1+4\theta\frac{\Delta t}{\Delta x^3}\right)\right\}\Delta x \int_0^1|\widehat{a}(\xi)|^2d\xi.
\end{equation*}
We conclude by using Identity \eqref{PARSEVAL}.
\end{proof}
\begin{Remark} The norm of the inverse operator $\mathcal{A}_{\theta}^{-1}$ is upper bounded by $1$ (independent of $\frac{\Delta t}{\Delta x^3}$). This independence is crucial to be able to impose a hyperbolic Courant-Friedrichs-Lewy condition (\ $[c+\frac{1}{2}]\frac{\Delta t}{\Delta x}<1 $) for $\theta\geq \frac{1}{2}$, to establish Equation \eqref{EQ_*} for example.
\end{Remark}

The operator $\mathcal{A}_{\theta}$ enables us to control not only the $\ell^2_{\Delta}$-norm (as proved in Proposition \ref{PROP_CONTINOUS_INVERSE}) but also an $h^2_{\Delta}$-discrete norm and $h^3_{\Delta}$-discrete norm as in the following proposition. 
\begin{Proposition}Let $\mathcal{A}_{\theta}$ be the operator defined by \eqref{DEF_opA}, then for any sequence $(a_j)_{j\in\mathbb{Z}}$, one has
\begin{equation*}
||\mathcal{A}_{\theta}a||_{\ell^{2}_{\Delta}}^2=||a||_{\ell^{2}_{\Delta}}^2+\theta\Delta t\Delta x||D_+D_-(a)||_{\ell^{2}_{\Delta}}^2+\theta^2\Delta t^2||D_+D_+D_-(a)||_{\ell^{2}_{\Delta}}^2.
\end{equation*}
\label{h_2_discret_norm}
\end{Proposition}
\begin{proof}
We develop the square of the $\ell^2_{\Delta}$-norm of $\left(\mathcal{A}_{\theta}a_j\right)_{j\in\mathbb{Z}}$ :
\begin{equation*}
\begin{split}
\left|\left|a+\theta\Delta tD_+D_+D_-(a)\right|\right|_{\ell^2_{\Delta}}^2=\left|\left|a\right|\right|_{\ell^2_{\Delta}}^2+2\theta\Delta t\left\langle a, D_+D_+D_-(a)\right\rangle+\theta^2\Delta t^2\left|\left|D_+D_+D_-(a)\right|\right|_{\ell^2_{\Delta}}^2.
\end{split}
\end{equation*}
Let us focus on the cross term. Discrete integration by parts \eqref{EQ_3_LEMMA_BASIC} together with \eqref{EQ_5_LEMMA_BASIC} (with $D_-(a)_j$ instead of $a_j$) give
\begin{equation*}
2\theta\Delta t\left\langle a, D_+D_+D_-(a)\right\rangle=-2\theta\Delta t\left\langle D_-(a), D_+D_-(a)\right\rangle=\theta\Delta t\Delta x\left|\left|D_+D_-(a)\right|\right|_{\ell^2_{\Delta}}^2,
\end{equation*}
which concludes the proof.
\end{proof}

The following proposition enables to deal with the term $\mathcal{A}_{-(1-\theta)}e_j^n$ in Equation \eqref{latest_ineq}.
\begin{Proposition}
For $\theta\in[0, 1]$, assume the CFL condition
$
\Delta t(1-2\theta)\leq\frac{\Delta x^3}{4}
$
is satisfied. 
Then, for any sequence $\left(a_j\right)_{j\in\mathbb{Z}}$, it holds
\begin{equation}
\left|\left| \mathcal{A}_{-(1-\theta)}a\right|\right|_{\ell^2_{\Delta}}^2\leq ||\mathcal{A}_{\theta}a||_{\ell^2_{\Delta}}^2.
\end{equation}
\label{eq_a_et_b}
\end{Proposition}
\begin{proof}
We develop the expression:
\begin{multline*}
\left|\left| \mathcal{A}_{-(1-\theta)}a\right|\right|_{\ell^2_{\Delta}}^2=\left|\left|a-(1-\theta)\Delta tD_+D_+D_-\left(a\right)\right|\right|_{\ell^2_{\Delta}}^2=\left|\left| a+\theta\Delta tD_+D_+D_-\left(a\right)\right|\right|_{\ell^2_{\Delta}}^2-2\Delta t\left\langle a, D_+D_+D_-\left(a\right)\right\rangle\\+\Delta t^2(1-2\theta)\left|\left|D_+D_+D_-\left(a\right)\right|\right|_{\ell^2_{\Delta}}^2.
\end{multline*}
By applying Relations \eqref{EQ_3_LEMMA_BASIC} and \eqref{EQ_5_LEMMA_BASIC} (with $D_-\left(a\right)_j$ instead of $a_j$), the previous equation becomes
\begin{equation*}
\left|\left|\mathcal{A}_{-(1-\theta)}a\right|\right|_{\ell^2_{\Delta}}^2=\left|\left|\mathcal{A}_{\theta}a\right|\right|_{\ell^2_{\Delta}}^2-\Delta x\Delta t\left|\left|D_+D_-\left(a\right)\right|\right|_{\ell^2_{\Delta}}^2+\Delta t^2(1-2\theta)\left|\left|D_+D_+D_-\left(a\right)\right|\right|_{\ell^2_{\Delta}}^2.
\end{equation*}
If $\theta\geq\frac{1}{2}$, Proposition \ref{eq_a_et_b} is proved.\\
If $\theta<\frac{1}{2}$, thanks to Identity \eqref{EQ_2_LEMMA_4}, we have\begin{equation*}
\left|\left|\mathcal{A}_{-(1-\theta)}a\right|\right|_{\ell^2_{\Delta}}^2=\left|\left|\mathcal{A}_{\theta}a\right|\right|_{\ell^2_{\Delta}}^2-\Delta x\Delta t\left|\left|D_+D_-\left(a\right)\right|\right|_{\ell^2_{\Delta}}^2+\frac{4\Delta t^2(1-2\theta)}{\Delta x^2}\left|\left|D_+D_-\left(a\right)\right|\right|_{\ell^2_{\Delta}}^2-\frac{4\Delta t^2(1-2\theta)}{\Delta x^2}\left|\left|D_+D\left(a\right)\right|\right|_{\ell^2_{\Delta}}^2.\end{equation*}

Since $\Delta t(1-2\theta)\leq \frac{\Delta x^3}{4}$, the term $\frac{4\Delta t^2(1-2\theta)}{\Delta x^2}$ is upper bounded by $\Delta t\Delta x$, which transforms the previous equation into
\begin{equation*}
\left|\left|\mathcal{A}_{-(1-\theta)}a\right|\right|_{\ell^2_{\Delta}}^2\leq\left|\left|\mathcal{A}_{\theta}a\right|\right|_{\ell^2_{\Delta}}^2-\Delta x\Delta t\left|\left|D_+D_-\left(a\right)\right|\right|_{\ell^2_{\Delta}}^2\\+\Delta t\Delta x\left|\left| D_+D_-\left(a\right)\right|\right|_{\ell^2_{\Delta}}^2-\frac{4\Delta t^2(1-2\theta)}{\Delta x^2}\left|\left|D_+D\left(a\right)\right|\right|_{\ell^2_{\Delta}}^2.
\end{equation*}
The conclusion of the proposition is a straightforward consequence, since $1-2\theta>0$.

\end{proof}

\subsection{Simplification of Inequality \eqref{latest_ineq}}\label{_Simplification_of_Inequality_}

The previous study of the dispersive operator $\mathcal{A}_{\theta}$ enables us to reorganize terms in the $\ell^2_{\Delta}$-stability inequality \eqref{latest_ineq} in a way simpler to study : signs of new terms are easier to identify. The reorganization is not exactly the same for $\theta\geq\frac{1}{2}$ and $\theta<\frac{1}{2}$, as seen in the following two corollaries of Proposition \ref{PROP_DISCRETE_RELATIVE_ENTROPY_INEQ}.
\begin{Corollary}[Corollary of Proposition \ref{PROP_DISCRETE_RELATIVE_ENTROPY_INEQ}] Consider Scheme \eqref{EQ_SCHEME}-\eqref{EQ_33333_bis_bis}.
Let $(e_j^n)_{(j,n)}$ be the convergence error defined by \eqref{EQ_8_barre_FRED_L}. Then, for every $n\in\llbracket0,N\rrbracket$, $\gamma \in[0,\frac{1}{2})$ and $\theta\geq \frac{1}{2}$, one has
\begin{equation}
\begin{split}
||\mathcal{A}&_{\theta}e^{n+1}||_{\ell^{2}_{\Delta}}^2\leq ||\mathcal{A}_{\theta}e^n||_{\ell^{2}_{\Delta}}^2\left[1+\Delta tE_{a}\right]+\Delta t||\epsilon^n||^2_{\ell^2_{\Delta}}\left\{1+4\frac{\Delta t}{\Delta x}+\Delta t\right\}\\
&+\Delta t\left\langle B_{b}, \left[D_+\left(e\right)^n\right]^2\right\rangle+\Delta t^2B_{c}\left|\left|D\left(e\right)^n\right|\right|_{\ell^{2}_{\Delta}}^2+\Delta tB_{e}\left|\left|D_+D\left(e\right)^n\right|\right|_{\ell^{2}_{\Delta}}^2+\Delta tB_{f}\left|\left|D_+D_+D_-\left(e\right)^n\right|\right|_{\ell^{2}_{\Delta}}^2.
\end{split}
\label{EQ_CN+Sup}
\end{equation}

with 
\begin{subequations}
\begin{align}
E_{a}&=||u_{\Delta}^n||_{\ell^{\infty}}^2\left(1+\frac{\Delta t}{\Delta x}\right)+||D_+\left(u_{\Delta}\right)^n||_{\ell^{\infty}}\left(7+\frac{\Delta t}{\Delta x}\left[2c+\frac{2}{3}||e^n||_{\ell^{\infty}}+\frac{3}{2}||(u_{\Delta})^n||_{\ell^{\infty}}\right]\right)\nonumber\\
&\hspace*{6cm}+||D_+(u_{\Delta})^n||_{\ell^{\infty}}^2\left[\sqrt{2}\frac{\sqrt{\Delta t}}{\sqrt{\Delta x}}+\frac{\Delta t^2}{\Delta x^2}\right]+1+2c^2\frac{\Delta t}{\Delta x},
\label{DEF_Ba}\\
B_{b}&=\left(\frac{\Delta x}{6}D_+\left(e\right)^n-c\boldsymbol{1}\right)\left(\Delta x-c\Delta t\right),\label{Cb2}\\
B_{c}&=||(u_{\Delta})^n||_{\ell^{\infty}}^2+\left\{||e^n||_{\ell^{\infty}}^2\left[1+\Delta x\right]+2||e^n||_{\ell^{\infty}}||(u_{\Delta})^n||_{\ell^{\infty}}+\frac{2c}{3}||e^n||_{\ell^{\infty}}\right\}-c^2,\label{Cc2}\\
B_{e}&=2(1-\theta)\Delta t\left\{||\left(u_{\Delta}\right)^n||_{\ell^{\infty}}+||e^n||_{\ell^{\infty}}+\frac{1}{2}+\left[\frac{\Delta x^{\frac{1}{2}-\gamma}+||e^n||_{\ell^{\infty}}+9||e^n||_{\ell^{\infty}}^2\Delta x^{\gamma-\frac{1}{2}}}{2}\right]\right\}-\Delta x,\label{Ce2}\\
B_{f}&=\Delta t\left\{(1-2\theta)+\frac{(1-\theta)\Delta x^2}{2}\left[c+\frac{1}{2}+\frac{\Delta x^{\frac{1}{2}-\gamma}+||e^n||_{\ell^{\infty}}+9||e^n||_{\ell^{\infty}}^2\Delta x^{\gamma-\frac{1}{2}}}{2}\right]\right\}-\frac{\Delta x^3}{4}.\label{Cf2}
\end{align}
\end{subequations}
\label{Cor_Grand}
\end{Corollary}
\begin{Remark}
Corollary \ref{Cor_Grand} is, in fact, true for all $\theta\neq 0$ (if $\theta<\frac{1}{2}$ we have to add the dispersive CFL condition hypothesis $\Delta t(1-2\theta)\leq\frac{\Delta x^3}{4}$), but we essentially use it for $\theta\geq \frac{1}{2}$.
\end{Remark}
\begin{proof}
We choose $\sigma=0$ in Inequality \eqref{latest_ineq}.\\
$\bullet$ First, we upper bound $||\mathcal{A}_{-(1-\theta)}e^n||_{\ell^{2}_{\Delta}}^2$ in \eqref{latest_ineq} by $||\mathcal{A}_{\theta}e^n||_{\ell^{2}_{\Delta}}^2$ thanks to Proposition \ref{eq_a_et_b}.\\
$\bullet$ We tranform $A_{b}$ in \eqref{Ca} into
\begin{equation*}
A_{b}=B_{b}+(1-\theta)\Delta t||D_+(u_{\Delta})^n||^2_{\ell^{\infty}}\boldsymbol{1},
\end{equation*}
with
\begin{equation}
B_{b}=\left(\frac{\Delta x}{6}D_+\left(e\right)^n-c\boldsymbol{1}\right)\left(\Delta x-c\Delta t\right).
\label{Bb}
\end{equation}
The $A_{b}$-term in \eqref{latest_ineq} thus is
\begin{equation}
\Delta t\left\langle A_{b}, \left(D_+e^n\right)^2\right\rangle=\Delta t\left\langle B_{b}, \left(D_+e^n\right)^2\right\rangle+(1-\theta)\Delta t^2||D_+u_{\Delta}^n||^2_{\ell^{\infty}}||D_+e^n||^2_{\ell^2_{\Delta}}.
\label{Avant_Gagliordo}
\end{equation}
For any sequence $(a_j)_{j\in\mathbb{Z}}$, the following Gagliardo-Nirenberg inequality 
\begin{equation*}
||D_+(a)||_{\ell^{2}_{\Delta}}^2\leq ||a||_{\ell^2_{\Delta}}||D_+D_-(a)||_{\ell^2_{\Delta}}
\end{equation*}
is valid even with the $\ell^2_{\Delta}$-norm. 
We will use it on $||D_+(e)^n||^2_{\ell^2_{\Delta}}$ in \eqref{Avant_Gagliordo}, to obtain
\begin{equation*}
(1-\theta)\Delta t^2||D_+(u_{\Delta})^n||^2_{\ell^{\infty}}||D_+e^n||^2_{\ell^2_{\Delta}}\leq (1-\theta)\Delta t^2||D_+(u_{\Delta})^n||_{\ell^{\infty}}^2\frac{||e^n||_{\ell^2_{\Delta}}\sqrt{\theta\Delta t\Delta x}||D_+D_-(e)^n||_{\ell^2_{\Delta}}}{\sqrt{\theta\Delta t\Delta x}}.
\end{equation*}
Proposition \ref{h_2_discret_norm} enables to make $||\mathcal{A}_{\theta}e^n||_{\ell^2_{\Delta}}^2$appear and
\begin{equation*}
\begin{split}
(1-\theta)\Delta t^2||D_+(u_{\Delta})^n||^2_{\ell^{\infty}}||D_+e^n||^2_{\ell^2_{\Delta}}&\leq \frac{(1-\theta)}{\sqrt{\theta}}\frac{\sqrt{\Delta t}}{\sqrt{\Delta x} }\Delta t||D_+(u_{\Delta})^n||_{\ell^{\infty}}^2||\mathcal{A}_{\theta}e^n||_{\ell^2_{\Delta}}^2.
\end{split}
\end{equation*}
$\bullet$ As a third step, we transform the $A_{d}$-term of \eqref{latest_ineq} (recall that $\sigma=0$):
\begin{equation*}
\Delta tA_{d}\left|\left|D_+D_-(e)^n\right|\right|_{\ell^2_{\Delta}}^2=(1-\theta)\Delta t^2\left|\left|D_+D_-(e)^n\right|\right|_{\ell^2_{\Delta}}^2+\frac{(1-\theta)}{2\theta}\Delta t||D_+(u_{\Delta})^n||_{\ell^{\infty}}\theta\Delta t\Delta x||D_+D_-(e)^n||_{\ell^2_{\Delta }}^2.
\end{equation*}
Relation \eqref{EQ_2_LEMMA_4} allows to rewrite the term $(1-\theta)\Delta t^2\left|\left|D_+D_-(e)^n\right|\right|_{\ell^2_{\Delta}}^2$:
\begin{equation*}
(1-\theta)\Delta t^2\left|\left|D_+D_-(e)^n\right|\right|_{\ell^2_{\Delta}}^2=(1-\theta)\Delta t^2\left|\left|D_+D(e)^n\right|\right|_{\ell^2_{\Delta}}^2+(1-\theta)\frac{\Delta t^2\Delta x^2}{4}\left|\left|D_+D_+D_-(e)^n\right|\right|_{\ell^2_{\Delta}}^2.
\end{equation*}
Proposition \ref{h_2_discret_norm} gives
\begin{equation*}
\frac{(1-\theta)}{2\theta}\Delta t||D_+(u_{\Delta})^n||_{\ell^{\infty}}\theta\Delta t\Delta x||D_+D_-(e)^n||_{\ell^2_{\Delta }}^2\leq\frac{(1-\theta)}{2\theta}\Delta t||D_+(u_{\Delta})^n||_{\ell^{\infty}}||\mathcal{A}_{\theta}e^n||_{\ell^2_{\Delta}}^2.
\end{equation*}
$\bullet$ Eventually, we focus on the $A_{f}$-term in \eqref{latest_ineq}. We decompose $A_{f}$ into
\begin{equation*}
A_{f}=A_{g}+\Delta t^2(1-\theta)||D_+(u_{\Delta})^n||_{\ell^{\infty}}
\end{equation*}
with 
\begin{equation}
A_{g}=\Delta t\left\{(1-2\theta)+\frac{(1-\theta)\Delta x^2}{2}\left[c+\frac{\Delta x^{\frac{1}{2}-\gamma}+||e^n||_{\ell^{\infty}}+9||e^n||_{\ell^{\infty}}^2\Delta x^{\gamma-\frac{1}{2}}}{2}\right]\right\}-\frac{\Delta x^3}{4}
\label{Bf}
\end{equation}
which leads to the following inequality (thanks to Proposition \ref{h_2_discret_norm}):
\begin{equation*}
\begin{split}
\Delta tA_{f}||D_+D_+D_-(e)^n||_{\ell^2_{\Delta}}^2&=\Delta tA_g||D_+D_+D_-(e)^n||_{\ell^2_{\Delta}}^2+\frac{(1-\theta)}{\theta^2}\Delta t||D_+(u_{\Delta})^n||_{\ell^{\infty}}||\theta\Delta tD_+D_+D_-(e)^n||_{\ell^2_{\Delta}}^2\\
&\leq\Delta tA_g||D_+D_+D_-(e)^n||_{\ell^2_{\Delta}}^2+\frac{(1-\theta)}{\theta^2}\Delta t||D_+(u_{\Delta})^n||_{\ell^{\infty}}||\mathcal{A}_{\theta}e^n||_{\ell^2_{\Delta}}^2.
\end{split}
\end{equation*}
Thanks to all the previous relations, we rewrite Inequality \eqref{latest_ineq} as
\begin{equation*}
\begin{split}
&||\mathcal{A}_{\theta}e^{n+1}||_{\ell^{2}_{\Delta}}^2\leq ||\mathcal{A}_{\theta}e^n||_{\ell^{2}_{\Delta}}^2\left[1+\Delta tB_{a}\right]+\Delta t||\epsilon^n||^2_{\ell^2_{\Delta}}\left\{1+4\frac{\Delta t}{\Delta x}+\Delta t\right\}+\Delta t\left\langle B_{b}, \left(D_+\left(e\right)^n\right)^2\right\rangle\\
&+\Delta t^2A_{c}\left|\left|D\left(e\right)^n\right|\right|_{\ell^{2}_{\Delta}}^2+\Delta t\left[A_{e}+(1-\theta)\Delta t\right]\left|\left|D_+D\left(e\right)^n\right|\right|_{\ell^{2}_{\Delta}}^2+\Delta t\left[A_g+(1-\theta)\frac{\Delta t\Delta x^2}{4}\right]\left|\left|D_+D_+D_-\left(e\right)^n\right|\right|_{\ell^{2}_{\Delta}}^2,
\end{split}
\end{equation*}
with \begin{equation*}
\begin{split}
B_{a}&=||u_{\Delta}^n||_{\ell^{\infty}}^2\left(1+\frac{\Delta t}{\Delta x}\right)+||D_+\left(u_{\Delta}\right)^n||_{\ell^{\infty}}\left(2-\theta+\frac{1-\theta}{2\theta}+\frac{1-\theta}{\theta^2}+\frac{\Delta t}{\Delta x}\left[2c+\frac{2}{3}||e^n||_{\ell^{\infty}}+\frac{3}{2}||(u_{\Delta})^n||_{\ell^{\infty}}\right]\right)\\
&+||D_+(u_{\Delta})^n||_{\ell^{\infty}}^2\left[\frac{(1-\theta)}{\sqrt{\theta}}\frac{\sqrt{\Delta t}}{\sqrt{\Delta x}}+\frac{\Delta t^2}{\Delta x^2}\right]+1+2c^2\frac{\Delta t}{\Delta x}.
\end{split}
\end{equation*}
For $\theta\in[\frac{1}{2}, 1]$, one has $B_{a}\leq E_{a}$ with $E_{a}$ defined in \eqref{DEF_Ba}.
Finally, we define $B_{c}:=A_{c}$ and $B_{e}:= A_{e}+(1-\theta)\Delta t$ and $B_{f}:=A_g+(1-\theta)\frac{\Delta t\Delta x^2}{4}$.

\end{proof}
\begin{Corollary}[Corollary of Proposition \ref{PROP_DISCRETE_RELATIVE_ENTROPY_INEQ}] Consider Scheme \eqref{EQ_SCHEME}-\eqref{EQ_33333_bis_bis}.
Let $(e_j^n)_{(j,n)}$ be the convergence error defined by \eqref{EQ_8_barre_FRED_L}. Then, for every $n\in\llbracket0,N\rrbracket$, $\gamma \in[0,\frac{1}{2})$ and $\theta<\frac{1}{2}$, one has, if $\Delta t(1-2\theta)\leq \frac{\Delta x^3}{4}$
\begin{equation*}
\begin{split}
||\mathcal{A}_{\theta}e^{n+1}||_{\ell^2_{\Delta}}^2&\leq ||\mathcal{A}_{\theta}e^n||_{\ell^2_{\Delta}}^2\left[1+E_a\Delta t\right]
+\Delta t||\epsilon^n||^2_{\ell^2_{\Delta}}\left\{1+4\frac{\Delta t}{\Delta x}+\Delta t\right\}\\
&+\Delta t\left\langle C_{b}, \left[D_+(e)^n\right]^2\right\rangle+\Delta t^2C_{c}||D(e)^n||_{\ell^2_{\Delta}}^2+\Delta tC_{d}||D_+D_-(e)^n||_{\ell^2_{\Delta}}^2+\Delta tC_{e}||D_+D(e)^n||_{\ell^2_{\Delta}}^2,
\end{split}
\end{equation*}
with

\begin{subequations}
\begin{align}
E_a&=||u_{\Delta}^n||_{\ell^{\infty}}^2\left(1+\frac{\Delta t}{\Delta x}\right)+||D_+\left(u_{\Delta}\right)^n||_{\ell^{\infty}}\left(7+\frac{\Delta t}{\Delta x}\left[2c+\frac{2}{3}||e^n||_{\ell^{\infty}}+\frac{3}{2}||(u_{\Delta})^n||_{\ell^{\infty}}\right]\right)\nonumber\\
&\hspace*{6cm}+||D_+(u_{\Delta})^n||_{\ell^{\infty}}^2\left[\sqrt{2}\frac{\sqrt{\Delta t}}{\sqrt{\Delta x}}+\frac{\Delta t^2}{\Delta x^2}\right]+1+2c^2\frac{\Delta t}{\Delta x},\\
C_{b}&=\left(\frac{\Delta x}{6}D_+\left(e\right)^n-c\boldsymbol{1}\right)\left(\Delta x-c\Delta t\right)+(1-\theta)\Delta t||D_+\left(u_{\Delta}\right)^n||_{\ell^{\infty}}\boldsymbol{1},\\
C_{c}&=||(u_{\Delta})^n||_{\ell^{\infty}}^2+\left\{||e^n||_{\ell^{\infty}}^2\left[1+\Delta x\right]+2||e^n||_{\ell^{\infty}}||(u_{\Delta})^n||_{\ell^{\infty}}+\frac{2c}{3}||e^n||_{\ell^{\infty}}\right\}-c^2,\\
C_{d}&=\frac{4}{\Delta x^2}\left\{\Delta t\left[(1-2\theta)+\frac{(1-\theta)\Delta x^2}{2}\left[c+\frac{\Delta x^{\frac{1}{2}-\gamma}+||e^n||_{\ell^{\infty}}+9||e^n||_{\ell^{\infty}}^2\Delta x^{\gamma-\frac{1}{2}}}{2}\right]\right.\right.\nonumber\\
&\hspace*{0.7cm}\left.\left.+\Delta t(1-\theta)||D_+\left(u_{\Delta}\right)^n||_{\ell^{\infty}}+\frac{(1-\theta)\Delta x^2}{4}\left\{||D_+(u_{\Delta})^n||_{\ell^{\infty}}+\frac{\Delta x}{2}||D_-\left(u_{\Delta}\right)^n||_{\ell^{\infty}}\right\}\right]-\frac{\Delta x^3}{4}\right\},\label{Cd3}\\
C_{e}&=2(1-\theta)\Delta t\left\{||\left(u_{\Delta}\right)^n||_{\ell^{\infty}}+||e^n||_{\ell^{\infty}}+\left[\frac{\Delta x^{\frac{1}{2}-\gamma}+||e^n||_{\ell^{\infty}}+9||e^n||_{\ell^{\infty}}^2\Delta x^{\gamma-\frac{1}{2}}}{2}\right]\right\}-\frac{4\Delta t}{\Delta x^2}\left\{(1-2\theta)\textcolor{white}{\frac{1}{1}}\right.\nonumber\\
&\left.\hspace*{1.3cm}+\frac{(1-\theta)\Delta x^2}{2}\left[c+\frac{\Delta x^{\frac{1}{2}-\gamma}+||e^n||_{\ell^{\infty}}+9||e^n||_{\ell^{\infty}}^2\Delta x^{\gamma-\frac{1}{2}}}{2}\right]+\Delta t(1-\theta)||D_+\left(u_{\Delta}\right)^n||_{\ell^{\infty}}\right\}.
\end{align}
\end{subequations}

\label{Cor_petit}
\end{Corollary}
\begin{Remark}
The variables $E_{a}$ are identical in both previous corollaries. It is noticed that Corollary \ref{Cor_petit} is valid for all $\theta$ but thereafter, it will be mainly used for $\theta<\frac{1}{2}$.
\end{Remark}
\begin{proof}
We choose $\sigma=1$ in Inequality \eqref{latest_ineq}.\\
$\bullet$ From Relation \eqref{EQ_2_LEMMA_4}, we transform the $A_{f}$-term in Inequality \eqref{latest_ineq} into
\begin{equation*}\Delta tA_{f}\left|\left|D_+D_+D_-e^n\right|\right|_{\ell^2_{\Delta}}^2=\Delta tA_{f}\left[\frac{4}{\Delta x^2}\left|\left|D_+D_-e^n\right|\right|_{\ell^2_{\Delta}}^2-\frac{4}{\Delta x^2}\left|\left|D_+De^n\right|\right|_{\ell^2_{\Delta}}^2\right].
\end{equation*} 
$\bullet$ We upper bound $||\mathcal{A}_{-(1-\theta)}e^n||_{\ell^{2}_{\Delta}}^2$ by $||\mathcal{A}_{\theta}e^n||_{\ell^{2}_{\Delta}}^2$ thanks to Proposition \ref{eq_a_et_b}, to obtain, instead of Inequality \eqref{latest_ineq},
\begin{equation*}
\begin{split}
||\mathcal{A}_{\theta}e^{n+1}||_{\ell^2_{\Delta}}^2&\leq ||\mathcal{A}_{\theta}e^n||_{\ell^2_{\Delta}}^2\left[1+A_{a}\Delta t+\Delta t\right]+\Delta t||\epsilon^n||^2_{\ell^2_{\Delta}}\left\{1+4\frac{\Delta t}{\Delta x}+\Delta t\right\}\\
&+\Delta t\left\langle A_{b}, \left[D_+(e)^n\right]^2\right\rangle+\Delta t^2A_{c}||D(e)^n||_{\ell^2_{\Delta}}^2+\Delta t\left\{A_{d}+\frac{4A_f}{\Delta x^2}\right\}||D_+D_-(e)^n||_{\ell^2_{\Delta}}^2\\
&+\Delta t\left\{A_{e}-\frac{4A_{f}}{\Delta x^2}\right\}||D_+D(e)^n||_{\ell^2_{\Delta}}^2.
\end{split}
\end{equation*}
We note $C_{a}:=A_{a}+1$ and verify $C_{a}\leq E_{a}$. Finally, we fix $C_{b}:=A_{b}$ with $\sigma=1$, $C_{c}:=A_{c}$, $C_{d}:=A_{d}+\frac{4A_{f}}{\Delta x^2}$ with $\sigma=1$ and $C_{e}:=A_{e}-\frac{4A_{f}}{\Delta x^2}$.

\end{proof}
In the following, we will have to show that $B_i$ and $C_i$ are non-positive to loop the estimates.
\subsection{Induction method}
\label{SUBSECTION_42}

We are now able to prove, by induction, the main result for a smooth initial datum: Theorem \ref{THEOREM_MAIN111}.

\begin{proof}[Proof of Theorem \ref{THEOREM_MAIN111}]
Let $T>0$ and $s\geq6$ with $u_0\in H^s(\mathbb{R})$. Let the Rusanov coefficient $c$ be such that
\eqref{DEF_c_INTRO} is true. This choice is possible because of Theorem \ref{theobourgain} which proves that the exact solution belongs to $L^{\infty}_x$ for $t\in[0,T]$. 
\begin{Remark}\label{def_alpha_0_8_nov}
Thanks to Hypothesis \eqref{DEF_c_INTRO} : $\underset{t\in[0,T]}{\mathrm{sup}}||u(t,\cdot)||_{L^{\infty}(\mathbb{R})}<c$, there exists a constant $\alpha_0>0$ such that, for all $\Delta t>0$, $\Delta x>0$ and for all $n\in\llbracket0,N\rrbracket$,
\begin{equation}\label{alpha0}
||(u_{\Delta})^n||_{\ell^{\infty}(\mathbb{Z})}+\alpha_0\leq ||u_{\Delta}||_{\ell^{\infty}(\llbracket0,N\rrbracket;\ell^{\infty}(\mathbb{Z}))} 
+ \alpha_0\leq \underset{t\in[0,T]}{\mathrm{sup}}||u(t,\cdot)||_{L^{\infty}(\mathbb{R})}+\alpha_0\leq c.
\end{equation}
\end{Remark}

\noindent Let $\beta_0\in(0,1)$, $\theta\in[0,1]$ and $\gamma\in(0,\frac{1}{2})$.
We define $\tilde{\omega}_0>0$ as 
\begin{equation}
\begin{split}
\tilde{\omega}_0=\left[\Lambda_{T, \|u_0\|_{H^{\frac{3}{4}}}}\left(1+\|u_0\|^2_{H^{\frac{1}{2}+\eta}}\right)\left(\frac{\|u_0\|_{H^6}}{c+\frac{1}{2}}+\|u_0\|_{H^4}+\|u_0\|_{H^{\frac{3}{2}+\eta}}\|u_0\|_{H^1}\right)\right]^{-\frac{1}{\gamma}},
\end{split}
\label{EQ_CONDITION}
\end{equation}
with $\Lambda_{T, \|u_0\|_{H^{\frac{3}{4}}}}$ defined in \eqref{LAMBDA_0}. \\
We also fix $\omega_0>0$ such that the following inequalities \eqref{eq_20_NOV_17} and \eqref{24_oct_cond_c_25}-\eqref{24_oct_disp_CFL2_25} if $\theta\geq\frac{1}{2}$ and the following inequalities \eqref{eq_20_NOV_17} and \eqref{COND_ON_c}-\eqref{DISP_CFL} if $\theta<\frac{1}{2}$ are verified: 
\begin{equation}\label{eq_20_NOV_17}
\omega_0^{\frac{1}{2}-\gamma}\leq3 c,
\end{equation}
\begin{itemize}
\item for $\theta\geq \frac{1}{2}$,
\begin{subequations}
\begin{numcases}{}
\omega_0^{\frac{1}{4}-\frac{\gamma}{2}}\sqrt{\left[\omega_0^{\frac{1}{2}-\gamma}+\omega_0^{\frac{3}{2}-\gamma}\right]+2\underset{t\in[0,T]}{\mathrm{sup}}||u(t,\cdot)||_{L^{\infty}(\mathbb{R})}+\frac{2c}{3}}\leq \alpha_0\label{24_oct_cond_c_25},\\
\frac{13(1-\beta_0)}{2c+1}\omega_0^{\frac{1}{2}-\gamma}\leq \beta_0\label{24_oct_hyp_CFL_25},\\
(1-2\theta)+\frac{(1-\theta)\omega_0^2}{2}\left[c+\frac{1}{2}+\frac{11}{2}\omega_0^{\frac{1}{2}-\gamma}\right]\leq 0, \ \ \ \text{\ if\ }\theta>\frac{1}{2}\label{24_oct_disp_CFL1_25},\\
\frac{11(1-\beta_0)}{2c+1}\omega_0^{\frac{1}{2}-\gamma}\leq\beta_0, \hspace*{3.8cm} \text{\ if\ }\theta=\frac{1}{2}\label{24_oct_disp_CFL2_25},
\end{numcases}
\end{subequations}
 \item for $\theta<\frac{1}{2}$,
\begin{subequations}
\begin{numcases}{\hspace*{-1cm}}
\omega_0^{\frac{1}{4}-\frac{\gamma}{2}}\sqrt{\left[\omega_0^{\frac{1}{2}-\gamma}+\omega_0^{\frac{3}{2}-\gamma}\right]+2\underset{t\in[0,T]}{\mathrm{sup}}||u(t,\cdot)||_{L^{\infty}(\mathbb{R})}+\frac{2c}{3}}\leq \alpha_0, \label{COND_ON_c}\\
12\omega_0^{\frac{1}{2}-\gamma}\leq \alpha_0,\label{24_oct_cond_c_2_25}\\\frac{(1-\theta)(1-\beta_0)}{2(1-2\theta)c}||(u_{\Delta})^n||_{\ell^{\infty}}\omega_0+\frac{(1-\beta_0)}{3c+\frac{3}{2}}\omega_0^{\frac{1}{2}-\gamma}+\frac{\omega_0^{\frac{1}{2}-\gamma}}{3c}\leq \beta_0,\label{24_oct_hyp_CFL_2_25}\\
\frac{(1-\theta)(1-\beta_0)}{2(1-2\theta)}\omega_0^2\left[c+\frac{11}{2}\omega_0^{\frac{1}{2}-\gamma}\right]+(1-\theta)||(u_{\Delta})^n||_{\ell^{\infty}}\frac{(1-\beta_0)}{(1-2\theta)}\left[\frac{(1-\beta_0)}{2(1-2\theta)}\omega_0^2+\frac{\omega_0(2+\omega_0)}{4}\right]\leq \beta_0.\label{DISP_CFL}
\end{numcases}
\end{subequations}
\end{itemize}
\begin{Remark}
These conditions on $\omega_0$ are very likely not optimal.
\end{Remark}

\noindent Let us prove by induction on $n\in\llbracket0,N\rrbracket$ that\begin{equation*}
\text{if\ }\Delta x\leq \min(\tilde{\omega}_0, \omega_0)\text{\ and\ if\ CFL\ conditions\ }\eqref{CFL_1}-\eqref{CFL_2}\text{\ hold,\ one\ has\ }||e^{n}||_{\ell^{\infty}}\leq \Delta x^{\frac{1}{2}-\gamma},\text{\ for\ all\ }n\in\llbracket0,N\rrbracket
\end{equation*}
\textbf{Initialization :} For $n=0$, the inequality
$
||e^0||_{\ell^{\infty}}\leq \Delta x^{\frac{1}{2}-\gamma}
$ is true because Expressions \eqref{EQ_33333_bis_bis} and \eqref{EQ_82_barre_FRED_L} imply \begin{equation*}
e_j^0=0,\ \ j\in\mathbb{Z}.
\end{equation*}
\textbf{Heredity :} Let us assume that 
\begin{equation}
\text{if\ }\Delta x\leq \min(\tilde{\omega}_0, \omega_0) \text{\ and\ if\ CFL\ conditions\ }\eqref{CFL_1}-\eqref{CFL_2}\text{\ hold,\ one\ has\ }||e^k||_{\ell^{\infty}}\leq \Delta x^{\frac{1}{2}-\gamma}, \text{\ for\ all\ } k\leq n.
\label{HR}
\end{equation} Then our goal is to prove that 
\begin{equation*}
\text{if\ }\Delta x\leq \min(\tilde{\omega}_0, \omega_0)\text{\ and\ if\ CFL\ conditions\ }\eqref{CFL_1}-\eqref{CFL_2}\text{\ hold,\ one\ has\ }||e^{n+1}||_{\ell^{\infty}}\leq \Delta x^{\frac{1}{2}-\gamma}.
\end{equation*}

\paragraph{Step 1 : simplification of Corollaries \ref{Cor_Grand} and \ref{Cor_petit}.}
Let us prove in this first step that $\Delta x\leq\min(\tilde{\omega}_0, \omega_0)$ and CFL conditions \eqref{CFL_1}-\eqref{CFL_2} imply the non-positivity of extra terms $B_i$ and $C_i$ in Corollaries \ref{Cor_Grand} and \ref{Cor_petit}. We dissociate two cases according to the value of $\theta$.\\

\fbox{\textsc{case $\theta\geq \frac{1}{2}$ :}}\\
We show the non-positivity of coefficients $B_{i}$ in Corollary \ref{Cor_Grand}, for $i\in\{b, c, e, f\}$.
\begin{itemize}
\item \textbf{Sign of $B_{b}$}:
We get by developing $D_+(e)_j^n$
\begin{equation*}
\frac{\Delta x}{6}D_+\left(e\right)_j^n\leq \frac{||e^n||_{\ell^{\infty}}}{3}.
\end{equation*}
However, by induction hypothesis, one has $\Delta x\leq \omega_0$ (with $\omega_0$ verifying, among others, Inequality \eqref{eq_20_NOV_17}) and $||e^n||_{\ell^{\infty}}\leq \Delta x^{\frac{1}{2}-\gamma}$. It gives
\begin{equation*}
\frac{||e^n||_{\ell^{\infty}}}{3}\leq \frac{\Delta x^{\frac{1}{2}-\gamma}}{3}\leq \frac{\omega_0^{\frac{1}{2}-\gamma}}{3}\leq c.
\end{equation*}
Due to the CFL condition \eqref{CFL_2}, one has
\begin{equation*}
\Delta x-c\Delta t\geq0.
\end{equation*} 
Thus, $B_{b}\leq 0$.

\item \textbf{Sign of $B_{c}$}:
For the term $B_{c}$, thanks to the hypothesis $||e^n||_{\ell^{\infty}}\leq \Delta x^{\frac{1}{2}-\gamma}$, we obtain
\begin{equation*}
\begin{split}
B_{c}\leq ||(u_{\Delta})^n||_{\ell^{\infty}}^2+\left\{\left[\Delta x^{1-2\gamma}+\Delta x^{2-2\gamma}\right]+2\Delta x^{\frac{1}{2}-\gamma}||(u_{\Delta})^n||_{\ell^{\infty}}+\frac{2c\Delta x^{\frac{1}{2}-\gamma}}{3}\right\}-c^2.
\end{split}
\end{equation*}
As $c\geq\alpha_0+||\left(u_{\Delta}\right)^n||_{\ell^{\infty}}$ (see Remark \ref{def_alpha_0_8_nov}) and $\Delta x\leq \omega_0$ (with $\omega_0$ satisfying Inequality \eqref{24_oct_cond_c_25}) by induction hypothesis, one has
\begin{equation*}
\begin{split}
B_{c}\leq ||(u_{\Delta})^n||_{\ell^{\infty}}^2+\left\{\left[\omega_0^{1-2\gamma}+\omega_0^{2-2\gamma}\right]+2\omega_0^{\frac{1}{2}-\gamma}||(u_{\Delta})^n||_{\ell^{\infty}}+\frac{2c\omega_0^{\frac{1}{2}-\gamma}}{3}\right\}-c^2\leq0.
\end{split}
\end{equation*}

\item \textbf{Sign of $B_{e}$}: since we suppose $||e^n||_{\ell^{\infty}}\leq \Delta x^{\frac{1}{2}-\gamma}$, the term $B_{e}$ satisfies
\begin{equation*}
B_{e}\leq2(1-\theta)\Delta t\left\{||(u_{\Delta})^n||_{\ell^{\infty}}+\frac{1}{2}+\frac{13}{2}\Delta x^{\frac{1}{2}-\gamma}\right\}-\Delta x.
\end{equation*}
As $\theta\geq \frac{1}{2}$, then $2(1-\theta)\leq 1$, and, thanks to the choice of $c$ \eqref{DEF_c_INTRO}, one has
\begin{equation*}
B_{e}\leq \Delta t\left\{c+\frac{1}{2}+\frac{13}{2}\Delta x^{\frac{1}{2}-\gamma}\right\}-\Delta x=\Delta x\left\{\frac{\Delta t}{\Delta x}\left[c+\frac{1}{2}\right]-1+\frac{13}{2}\frac{\Delta t}{\Delta x}\Delta x^{\frac{1}{2}-\gamma}\right\}.
\end{equation*}
Using $\Delta x\leq \omega_0$ and using hyperbolic CFL \eqref{CFL_2}, one has 
$$\frac{13}{2}\frac{\Delta t}{\Delta x}\Delta x^{\frac{1}{2}-\gamma}\leq \frac{13}{2}\frac{(1-\beta_0)}{c+\frac{1}{2}}\Delta x^{\frac{1}{2}-\gamma}\leq \frac{13(1-\beta_0)}{2c+1}\omega_0^{\frac{1}{2}-\gamma}$$
which is less than $\beta_0$ thanks to Inequality \eqref{24_oct_hyp_CFL_25}. Thus one has
$$B_{e}\leq 0.$$
\item \textbf{Sign of $B_{f}$}: the dispersive CFL-type condition \eqref{CFL_1} together with hypothesis $||e^n||_{\ell^{\infty}}\leq\Delta x^{\frac{1}{2}-\gamma}$ give
\begin{equation*}
B_{f}\leq \Delta t\left\{(1-2\theta)+\frac{(1-\theta)\Delta x^2}{2}\left[c+\frac{1}{2}+\frac{11}{2}\Delta x^{\frac{1}{2}-\gamma}\right]\right\}-\frac{\Delta x^3}{4},
\end{equation*}
which is non-positive if $\Delta x\leq \omega_0$. Indeed,
\begin{itemize}
\item if $\theta>\frac{1}{2}$, one has chosen $\omega_0$ such that $$(1-2\theta)+\frac{(1-\theta)}{2}\Delta x^2\left[c+\frac{1}{2}+\frac{11}{2}\Delta x^{\frac{1}{2}-\gamma}\right]\leq (1-2\theta)+\frac{(1-\theta)}{2}\omega_0^2\left[c+\frac{1}{2}+\frac{11}{2}\omega_0^{\frac{1}{2}-\gamma}\right]\leq0,$$ thanks to Inequality \eqref{24_oct_disp_CFL1_25},
\item if $\theta=\frac{1}{2}$,
\begin{equation*}
B_{f}\leq \frac{\Delta t\Delta x^2}{4}\left[c+\frac{1}{2}+\frac{11}{2}\Delta x^{\frac{1}{2}-\gamma}\right]-\frac{\Delta x^3}{4}=\frac{\Delta x^3}{4}\left\{\frac{\Delta t}{\Delta x}\left[c+\frac{1}{2}\right]-1+\frac{11\Delta t}{2\Delta x}\Delta x^{\frac{1}{2}-\gamma}\right\},
\end{equation*}
and Condition \eqref{CFL_2} together with $\Delta x\leq \omega_0$ for $\omega_0$ verifying Inequality \eqref{24_oct_disp_CFL2_25} enable us to conclude about the non-positivity of $B_{f}$.
\end{itemize}
\end{itemize}

\fbox{\textsc{case $\theta< \frac{1}{2}$ :}}\\
In the same way, from Corollary \ref{Cor_petit}, we show the non-positivity of $C_{i}$, for $i\in\{b, c, d, e\}$.
\begin{itemize}
\item \textbf{Sign of $C_{b}$}: one has, by definition of $C_{b}$ and by hypothesis $||e^n||_{\ell^{\infty}}\leq \Delta x^{\frac{1}{2}-\gamma}$
\begin{equation*}
\begin{split}
C_{b}&\leq \left(\frac{\Delta x}{6}D_+\left(e\right)_j^n-c\right)\left(\Delta x-c\Delta t\right)+2(1-\theta)\frac{\Delta t}{\Delta x}||(u_{\Delta})^n||_{\ell^{\infty}}\\
&\leq \frac{\Delta x||e^n||_{\ell^{\infty}}}{3}+\frac{c\Delta t||e^n||_{\ell^{\infty}}}{3}-c\Delta x+c^2\Delta t+2(1-\theta)\frac{\Delta t}{\Delta x}||(u_{\Delta})^n||_{\ell^{\infty}}\\
&\leq c\left[c\Delta t\left(1+\frac{\Delta x^{\frac{1}{2}-\gamma}}{3c}\right)-\Delta x\left(1-\frac{\Delta x^{\frac{1}{2}-\gamma}}{3c}-2(1-\theta)\frac{\Delta t}{\Delta x^2c}||(u_{\Delta})^n||_{\ell^{\infty}}\right)\right]\\
&\leq c\Delta x\left[c\frac{\Delta t}{\Delta x}+\frac{\Delta t}{\Delta x}\frac{\Delta x^{\frac{1}{2}-\gamma}}{3}-1+\frac{\Delta x^{\frac{1}{2}-\gamma}}{3c}+2(1-\theta)\frac{\Delta t}{\Delta x^2c}||(u_{\Delta})^n||_{\ell^{\infty}}\right].
\end{split}
\end{equation*}
The hyperbolic CFL condition \eqref{CFL_2} and the dispersive one \eqref{CFL_1} (we recall that $1-2\theta>0$ in that case) imply
\begin{equation*}
\begin{split}
C_{b}&\leq c\Delta x\left[1-\beta_0+\frac{(1-\beta_0)\Delta x^{\frac{1}{2}-\gamma}}{3c+\frac{3}{2}}-1+\frac{\Delta x^{\frac{1}{2}-\gamma}}{3c}+(1-\theta)\frac{\Delta x(1-\beta_0)}{2c(1-2\theta)}||(u_{\Delta})^n||_{\ell^{\infty}}\right].
\end{split}
\end{equation*}
The choice of $\omega_0$ small enough to satisfy Inequalities \eqref{24_oct_hyp_CFL_2_25} implies $C_{b}\leq 0$. 

\item \textbf{Sign of $C_{c}$}: since $C_c=B_c$, we follow exactly the same proof as for $\theta\geq \frac{1}{2}$ to show $C_{c}\leq 0$.
\item \textbf{Sign of $C_{d}$}: thanks to Definition \eqref{Cd3}, one has\begin{equation*}
\begin{split}
C_{d}&=\frac{4}{\Delta x^2}\left\{\Delta t\left[(1-2\theta)+\frac{(1-\theta)\Delta x^2}{2}\left[c+\frac{\Delta x^{\frac{1}{2}-\gamma}+||e^n||_{\ell^{\infty}}+9||e^n||_{\ell^{\infty}}^2\Delta x^{\gamma-\frac{1}{2}}}{2}\right]\right.\right.\\
&\left.\left.+\Delta t(1-\theta)||D_+\left(u_{\Delta}\right)^n||_{\ell^{\infty}}+\frac{(1-\theta)\Delta x^2}{4}\left\{||D_+(u_{\Delta})^n||_{\ell^{\infty}}+\frac{\Delta x}{2}||D_-\left(u_{\Delta}\right)^n||_{\ell^{\infty}}\right\}\right]-\frac{\Delta x^3}{4}\right\}
\end{split}
\end{equation*}
Since $||e^n||_{\ell^{\infty}}\leq \Delta x^{\frac{1}{2}-\gamma}$, it becomes, thanks to the dispersive CFL \eqref{CFL_1},

\begin{equation*}
\begin{split}
C_{d}&=\Delta x\left\{\frac{4\Delta t}{\Delta x^3}(1-2\theta)+\frac{2\Delta t}{\Delta x}(1-\theta)\left[c+\frac{11\Delta x^{\frac{1}{2}-\gamma}}{2}\right]\right.\\
&\left.\hspace*{2cm}+8\frac{\Delta t^2}{\Delta x^4}(1-\theta)||u_{\Delta}^n||_{\ell^{\infty}}+2(1-\theta)\frac{\Delta t}{\Delta x^2}||u_{\Delta}^n||_{\ell^{\infty}}+(1-\theta)\frac{\Delta t}{\Delta x}||u_{\Delta}^n||_{\ell^{\infty}}-1\right\}\\
&\leq \Delta x\left\{\frac{4\Delta t}{\Delta x^3}(1-2\theta)+\frac{\Delta x^2(1-\beta_0)}{2(1-2\theta)}(1-\theta)\left[c+\frac{11\Delta x^{\frac{1}{2}-\gamma}}{2}\right]+\frac{(1-\beta_0)^2\Delta x^2}{2(1-2\theta)^2}(1-\theta)||u_{\Delta}^n||_{\ell^{\infty}}\right.\\
&\left.\hspace*{2cm}+(1-\theta)\frac{(1-\beta_0)\Delta x}{2(1-2\theta)}||u_{\Delta}^n||_{\ell^{\infty}}+(1-\theta)\frac{\Delta x^2(1-\beta_0)}{4(1-2\theta)}||u_{\Delta}^n||_{\ell^{\infty}}-1\right\}\\
&= \Delta x\left\{\frac{4\Delta t}{\Delta x^3}(1-2\theta)+\frac{\Delta x^2(1-\beta_0)}{2(1-2\theta)}(1-\theta)\left[c+\frac{11\Delta x^{\frac{1}{2}-\gamma}}{2}\right]\right.\\
&\left.\hspace*{2cm}+(1-\theta)||u_{\Delta}^n||_{\ell^{\infty}}\frac{(1-\beta_0)}{(1-2\theta)}\left[\frac{(1-\beta_0)}{2(1-2\theta)}\Delta x^2+\frac{\Delta x(2+\Delta x)}{4}\right]-1\right\}
\end{split}
\end{equation*} 
Thanks to $\Delta x\leq\omega_0$, with $\omega_0$ verifying \eqref{DISP_CFL} and thanks to the CFL condition \eqref{CFL_1}, one has 
\begin{equation*}
C_{d}\leq 0.
\end{equation*}
\item \textbf{Sign of $C_{e}$}: we develop $C_{e}$ to obtain
\begin{equation*}
\begin{split}
C_{e}&\leq2(1-\theta)\Delta t\left\{||(u_{\Delta})^n||_{\ell^{\infty}}+\frac{13}{2}\Delta x^{\frac{1}{2}-\gamma}\right\}-\frac{4\Delta t}{\Delta x^2}(1-2\theta)-2(1-\theta)\Delta t\left[c-\frac{11\Delta x^{\frac{1}{2}-\gamma}}{2}\right]\\
&-\frac{8\Delta t^2}{\Delta x^3}(1-\theta)||\left(u_{\Delta}\right)^n||_{\ell^{\infty}}\\
&\leq 2(1-\theta)\Delta t\left\{||(u_{\Delta})^n||_{\ell^{\infty}}+12\Delta x^{\frac{1}{2}-\gamma}-c\right\}-\frac{4\Delta t}{\Delta x^2}\left[(1-2\theta)+\frac{2\Delta t}{\Delta x}(1-\theta)||(u_{\Delta})^n||_{\ell^{\infty}}\right].
\end{split}
\end{equation*}
Since $\theta<\frac{1}{2}$, one has $1-2\theta>0$ then $-\frac{4\Delta t}{\Delta x^2}\left[(1-2\theta)+\frac{2\Delta t}{\Delta x}(1-\theta)||(u_{\Delta})^n||_{\ell^{\infty}}\right]\leq 0$.
The hypothesis $\Delta x\leq \omega_0$, with $\omega_0$ satisfying \eqref{24_oct_cond_c_2_25} and the choice of $c$ \eqref{DEF_c_INTRO} give $C_{e}\leq0$.
\end{itemize}

\fbox{\textsc{all in all :}}\\
We have proved that, under the induction hypothesis, the following equality holds, for all $\theta\in[0,1]$
\begin{equation}
\hspace*{-2.2cm}||\mathcal{A}_{\theta}e^{n+1}||_{\ell^2_{\Delta}}^2\leq ||\mathcal{A}_{\theta}e^n||_{\ell^2_{\Delta}}^2\left\{1+\Delta tE_a\right\}+\Delta t||\epsilon^n||_{\ell^2_{\Delta}}^2\left\{1+4\frac{\Delta t}{\Delta x}+\Delta t\right\},
\label{latest_ineq_simplifiee}
\end{equation}
with $E_a$ defined by \eqref{DEF_Ba}.

\paragraph{Step 2 : From $e^n$ to $e^{n+1}$ thanks to a discrete Gr\"onwall lemma.}
By splitting $E_{a}$ and using the first inequality of \eqref{BORNE_2222222222} to upper bound $\Delta t||D_+\left(u_{\Delta}\right)^n||_{\ell^{\infty}}$ and $\Delta t||D_+\left(u_{\Delta}\right)^n||^2_{\ell^{\infty}}$, Inequality \eqref{latest_ineq_simplifiee} becomes
\begin{equation*}
||\mathcal{A}_{\theta}e^{n+1}||_{\ell^2_{\Delta}}^2\leq ||\mathcal{A}_{\theta}e^n||_{\ell^2_{\Delta}}^2\left\{1+\Delta tE_b^n+\sum_{i=1}^{2}\left(\int_{t^n}^{t^{n+1}}||\partial_x u(s, .)||_{L^{\infty}_x}^ids\right)E_{c, i}^n\right\}+\Delta t||\epsilon^n||_{\ell^2_{\Delta}}^2\left\{1+4\frac{\Delta t}{\Delta x}+\Delta t\right\},
\end{equation*}
with
\begin{equation*}
E_b^n=\left[||u_{\Delta}^n||_{\ell^{\infty}}^2\left(1+\frac{\Delta t}{\Delta x}\right)+1+2c^2\frac{\Delta t}{\Delta x}\right]\leq \left[1+||u_{\Delta}||_{\ell_n^{\infty}\ell^{\infty}}^2(1+\frac{\Delta t}{\Delta x})+2\frac{\Delta t}{\Delta x}c^2\right]
\end{equation*}
and
\begin{equation*}
E_{c, 1}^n=\left[7+\frac{\Delta t}{\Delta x}\left(2c+\frac{2}{3}\Delta x^{\frac{1}{2}-\gamma}+\frac{3}{2}||(u_{\Delta})^n||_{\ell^{\infty}}\right)\right]\leq \left[7+\frac{\Delta t}{\Delta x}\left(2c+\frac{2}{3}\Delta x^{\frac{1}{2}-\gamma}+\frac{3}{2}||u_{\Delta}||_{\ell^{\infty}\ell^{\infty}_n}\right)\right]
\end{equation*}
and
\begin{equation*}
E_{c, 2}^n=\left[\sqrt{2}\frac{\sqrt{\Delta t}}{\sqrt{\Delta x}}+\frac{\Delta t^2}{\Delta x^2}\right].
\end{equation*}
Due to the CFL condition, we have, denoting by $C$ a number independent of $c$, $u_{\Delta }^n$, $\Delta t$ and $\Delta x$
\begin{equation}
E_b^n\leq C\left(1+c^2\left(1+\frac{\Delta t}{\Delta x}\right)\right)=:E_b,
\label{Cg}
\end{equation}
\begin{equation}
E_{c,1}^n\leq C\left(1+\frac{\Delta t}{\Delta x}\left[1+c\right]\right)=:E_{c,1}
\label{Ch}
\end{equation}
and
\begin{equation}
E_{c,2}^n=\left[\sqrt{2}\frac{\sqrt{\Delta t}}{\sqrt{\Delta x}}+\frac{\Delta t^2}{\Delta x^2}\right]=:E_{c,2}.
\label{Ch2}
\end{equation}

We can now apply a discrete Gr\"onwall Lemma (noticing that $e^0_j=0,\ \ j\in\mathbb{Z}$). It provides, for every $n\in\llbracket0,N-1\rrbracket$,
\begin{equation}
||\mathcal{A}_{\theta}e^{n+1}||^2_{\ell^2_{\Delta}}\leq \exp\left(t^{n+1}E_b+\sum_{i=1}^{2}\int_{0}^{t^{n+1}}||\partial_xu(s,.)||^i_{L^{\infty}_x(\mathbb{R})}E_{c,i}\right)\underset{n\in\llbracket0,N\rrbracket}{\mathrm{sup}}||\epsilon^n||_{\ell^2_{\Delta}}^2T\left\{1+4\frac{\Delta t}{\Delta x}+\Delta t\right\}.
\label{EQ_--}
\end{equation}
Finally, Theorem \ref{theobourgain} and Proposition \ref{LEMMA_ON_ESPILON_AAAAAA} give, for $0<\eta\leq 6-\frac{3}{2}$,
\begin{equation}\label{ineq_18_NOV_17}
||\mathcal{A}_{\theta}e^{n+1}||^2_{\ell^2_{\Delta}}\leq M^2\left(1+\|u_0\|^2_{H^{\frac{1}{2}+\eta}}\right)^2\left\{\Delta t^2\|u_0\|^2_{H^6}+\Delta x^2\left[\|u_0\|_{H^4}^2+\|u_0\|_{H^{\frac{3}{2}+\eta}}^2\|u_0\|_{H^1}^2\right]\textcolor{white}{\frac{1}{1}}\right\},
\end{equation}
with
\begin{equation*}
\begin{split}
M^2=& \exp\left(TE_b+\|u_0\|_{H^{\frac{3}{4}}}C_{\frac{3}{4}}e^{\kappa_{\frac{3}{4}}T}\left[E_{c,1}T^{\frac{3}{4}}+E_{c,2}T^{\frac{1}{2}}\right]\right)C^2e^{2\kappa T}T\left\{1+4\frac{\Delta t}{\Delta x}+\Delta t\right\}\\
\leq&\exp\left(C\left(1+c^2\right)\left(1+\frac{\Delta t^2}{\Delta x^2}\right)\left(T+(T^{\frac{3}{4}}+T^{\frac{1}{2}})||u_0||_{H^{\frac{3}{4}}}e^{\kappa_{\frac{3}{4}} T}\right)\right)C^2e^{2\kappa T}T\left\{1+\frac{\Delta t}{\Delta x}\right\},
\end{split}
\end{equation*}
with $C$ independent of $u_0$ and $\kappa$, $\kappa_{\frac{3}{4}}$ dependent only on $||u_0||_{L^2}$.
Thanks to the CFL condition \eqref{CFL_2}, an upper bound for $M$ is
\begin{equation*}
M^2\leq \Lambda^2_{T,||u_0||_{H^{\frac{3}{4}}}}
\end{equation*}
with
$$\Lambda^2_{T,||u_0||_{H^{\frac{3}{4}}}}= \exp\left(C\left(1+c^2\right)\left(1+\frac{(1-\beta_0)^2}{(c+\frac{1}{2})^2}\right)\left(T+(T^{\frac{3}{4}}+T^{\frac{1}{2}})||u_0||_{H^{\frac{3}{4}}}e^{\kappa_{\frac{3}{4}} T}\right)\right)C^2e^{2\kappa T}T\left\{1+\frac{1-\beta_0}{c+\frac{1}{2}}\right\}.$$
Since $||e^{n+1}||_{\ell^2_{\Delta}}^2\leq||\mathcal{A}_{\theta}e^{n+1}||^2_{\ell^2_{\Delta}}$ (Proposition \ref{PROP_CONTINOUS_INVERSE}), Inequality \eqref{ineq_18_NOV_17} gives
\begin{equation}
\begin{split}
||e^{n+1}||^2_{\ell^2_{\Delta}}&\leq \Lambda^2_{T, \|u_0\|_{H^{\frac{3}{4}}}}\left(1+\|u_0\|^2_{H^{\frac{1}{2}+\eta}}\right)^2\left\{\Delta t^2\|u_0\|^2_{H^6}+\Delta x^2\left[\|u_0\|_{H^4}^2+\|u_0\|_{H^{\frac{3}{2}+\eta}}^2\|u_0\|_{H^1}^2\right]\right\}\\
&\leq \Lambda^2_{T, \|u_0\|_{H^{\frac{3}{4}}}}\left(1+\|u_0\|^2_{H^{\frac{1}{2}+\eta}}\right)^2\left(\frac{\|u_0\|^2_{H^6}}{\left(c+\frac{1}{2}\right)^2}+\|u_0\|_{H^4}^2+\|u_0\|_{H^{\frac{3}{2}+\eta}}^2\|u_0\|_{H^1}^2\right)\Delta x^2,
\end{split}
\label{EQ_*}
\end{equation}
where the last inequality is obtained thanks to the CFL condition \eqref{CFL_2}.\\

\textbf{Conclusion :} It remains to verify the induction hypothesis \eqref{HR} at step $n+1$. The definition of the $\ell^2_{\Delta}$-norm, Identity \eqref{NORME_L2}, together with the inclusion $\ell^2\subset\ell^{\infty}$, holds
\begin{equation*}
||e^n||_{\ell^{\infty}}\leq \frac{||e^n||_{\ell^2_{\Delta }}}{\sqrt{\Delta x}}.
\end{equation*}
According to the upper bound \eqref{EQ_*}, the $\ell^{\infty}$-norm is bounded as follow
\begin{equation*}
\begin{split}
||e^{n+1}||_{\ell^{\infty}}\leq \Lambda_{T, \|u_0\|_{H^{\frac{3}{4}}}}\left(1+\|u_0\|^2_{H^{\frac{1}{2}+\eta}}\right)\left(\frac{\|u_0\|_{H^6}}{c+\frac{1}{2}}+\|u_0\|_{H^4}+\|u_0\|_{H^{\frac{3}{2}+\eta}}\|u_0\|_{H^1}\right)\sqrt{\Delta x}.
\end{split}
\end{equation*}

The choice of a small $\Delta x$ satisfying $\Delta x\leq\min(\tilde{\omega}_0, \omega_0)$ with $\tilde{\omega}_0$ defined in \eqref{EQ_CONDITION} implies thus
$||e^{n+1}||_{\ell^{\infty}}\leq \Delta x^{\frac{1}{2}-\gamma}$. The induction hypothesis is then true for $n+1$. \end{proof}

Thus, we have proved Equation \eqref{errorestimatetheo} with $\Lambda_{T, ||u_0||_{H^{\frac{3}{4}}}}$ defined by \eqref{LAMBDA_0} and $\widehat{\omega_0}=\min(\omega_0, \tilde{\omega}_0)$.

\begin{Remark}
The choice of a time average in the definition of $u_{\Delta}$, Equation \eqref{NOTATION_BAR}, is dictated by the discrete Gr\"onwall Lemma on \eqref{EQ_--}. Indeed, applying discrete Gr\"onwall Lemma introduces the following term $\sum_{n=0}^N\Delta t||D_+\left(u_{\Delta}\right)^n||^i_{\ell^{\infty}}$ which is controlled thanks to the estimate \eqref{BORNE_2222222222}, where the time integral plays a crucial role.\\
Regarding the space average in the definition of $u_{\Delta}$, its necessity comes from controlling the sum on $j\in\mathbb{Z}$ in the consistency estimates \eqref{EQ_4|}.\\
\end{Remark}

\begin{Remark}
This method is a process to find the CFL condition which suits also for the Airy equation \begin{equation*}
\partial_t u(t,x)+\partial_x^3 u(t,x)=0, \ \ (t, x)\in [0,T]\times \mathbb{R},
\end{equation*}
with the finite difference scheme
\begin{equation}
\frac{v_j^{n+1}-v_j^n}{\Delta t}+\theta\frac{v_{j+2}^{n+1}-3v_{j+1}^{n+1}+3v_{j}^{n+1}-v_{j-1}^{n+1}}{\Delta x^3}+(1-\theta)\frac{v_{j+2}^{n}-3v_{j+1}^{n}+3v_{j}^{n}-v_{j-1}^{n}}{\Delta x^3}=0.
\label{Schema_Airy}
\end{equation}
The analogue of Equation \eqref{latest_ineq} is here

\begin{equation*}
\left|\left|\mathcal{A}_{\theta}e^{n+1}\right|\right|_{\ell^2_{\Delta}}^2\leq \left\{1+\Delta t\right\}\left|\left|\mathcal{A}_{\theta}e^n\right|\right|_{\ell^2_{\Delta}}^2+\Delta t\left\{1+\Delta t\right\}\left|\left|\epsilon^n\right|\right|_{\ell^2_{\Delta}}^2+\Delta t\underbrace{\left\{1+\Delta t\right\}\left\{(1-2\theta)\Delta t-\frac{\Delta x^3}{4}\right\}}_{B_f^{\mathrm{Airy}}}\left|\left|D_+D_+D_-\left(e\right)^n\right|\right|_{\ell^2_{\Delta}}^2.
\end{equation*}

Imposing $B_f^{\mathrm{Airy}}\leq0$ (which corresponds to Step 1 in the previous proof of Theorem \ref{THEOREM_MAIN111}) leads to
\begin{equation*}
\Delta t(1-2\theta)\leq\frac{\Delta x^3}{4}.
\end{equation*}
This so-called Courant-Friedrichs-Lewy condition, in the case $\theta=0$, is exactly the one which is obtained in \cite{Mengzhao_1983} with a computation of the zeros of the amplification factor in \cite{Mengzhao_1983} and the one obtained by the Fourier method.
Indeed, the amplification factor obtained by Fourier analysis on Airy equation is
\begin{equation*}
\frac{1-8\frac{(1-\theta)\Delta t}{\Delta x^3}\sin^4(\pi\xi)-8i\frac{(1-\theta)\Delta t}{\Delta x^3}\sin^3(\pi\xi)\cos(\pi\xi)}{1+8\frac{\theta\Delta t}{\Delta x^3}\sin^4(\pi\xi)+8i\frac{\theta\Delta t}{\Delta x^3}\sin^3(\pi\xi)\cos(\pi\xi)}, \ \ \ \xi \in (0,1).
\end{equation*}
Requiring that its modulus is less than 1 yields
\begin{equation*}
\Delta t\sin^2(\pi\xi)(1-2\theta)\leq \frac{\Delta x^3}{4}, \mathrm{\ for\ all\ }\xi\in (0,1).
\end{equation*}
\label{REMARK_AIRY_STABLE}
\end{Remark}

\begin{Remark}
For a Rusanov finite difference scheme applied to the non-linear term of the KdV equation: the Burgers equation
\begin{equation*}
\partial_t u(t, x)+\partial_x\left(\frac{u^2}{2}\right)(t,x)=0, \ \ (t,x)\in[0,T]\times\mathbb{R},
\end{equation*}
which corresponds to the discrete equation
\begin{equation}
\frac{v_j^{n+1}-v_j^n}{\Delta t}+\frac{\left(v_{j+1}^n\right)^2-\left(v_{j-1}^n\right)^2}{4\Delta x}=c\left(\frac{v_{j+1}^n-2v_{j}^n+v_{j-1}^n}{2\Delta x}\right), \ \ \ (n,j)\in\llbracket0,N\rrbracket\times\mathbb{Z},\label{Schema_Burgers}
\end{equation}
the analogue of Equation \eqref{latest_ineq} would be
\begin{multline*}
||e^{n+1}||_{\ell^2_{\Delta}}^2\leq ||e^n||_{\ell^2_{\Delta}}^2\left\{1+\Delta tE_a^{\mathrm{Burgers}}\right\}+\Delta t\left\{4\frac{\Delta t}{\Delta x}+\Delta t\right\}\left|\left|\epsilon^n\right|\right|_{\ell^2_{\Delta}}^2+\Delta t\left\langle B_b^{\mathrm{Burgers}}, \left[D_+\left(e\right)^n\right]^2\right\rangle\\+\Delta t^2B_c^{\mathrm{Burgers}}\left|\left|D\left(e\right)^n\right|\right|_{\ell^2_{\Delta}}^2,
\end{multline*}
with 
\begin{multline*}
E_a^{\mathrm{Burgers}}=||u_{\Delta}^n||_{\ell^{\infty}}^2+||D_+\left(u_{\Delta}\right)^n||_{\ell^{\infty}}\left(1+\frac{\Delta t}{\Delta x}\left[2c+\frac{2}{3}||e^n||_{\ell^{\infty}}+\frac{3}{2}||u_{\Delta}^n||_{\ell^{\infty}}\right]\right)\\+\frac{\Delta t^2}{\Delta x^2}||D_+\left(u_{\Delta}\right)^n||_{\ell^{\infty}}^2+\frac{\Delta t}{\Delta x}\left(||(u_{\Delta})^n||_{\ell^{\infty}}^2+2c^2\right),
\end{multline*}
\begin{equation*}
B_b^{\mathrm{Burgers}}=\left(\frac{\Delta x}{6}D_+\left(e\right)^n-c\boldsymbol{1}\right)\left(\Delta x-c\Delta t\right),
\end{equation*}
and
\begin{equation*}
B_c^{\mathrm{Burgers}}=||e^n||_{\ell^{\infty}}^2\left[1+\Delta x\right]+||u_{\Delta}^n||_{\ell^{\infty}}^2-c^2+2||e^n||_{\ell^{\infty}}||u_{\Delta}^n||_{\ell^{\infty}}+\frac{2c}{3}||e^n||_{\ell^{\infty}}.
\end{equation*}

Therefore, for $u_0\in H^{\frac{3}{2}}(\mathbb{R})$ and for $\Delta x$ small enough, the well-known CFL condition is verified
\begin{equation*}
c\Delta t\leq\Delta x,
\end{equation*}
(thanks to the condition $B_b^{\mathrm{Burgers}}\leq 0$) and the well-known condition for the Rusanov coefficient is verified
\begin{equation*}
||u_{\Delta}^n||_{\ell^{\infty}}<c,
\end{equation*} 
(thanks to the condition $B_c^{\mathrm{Burgers}}\leq 0$).
\label{REMARK_BURGERS_STABLE}
\end{Remark}

\begin{Remark}
For Burgers equation, we know a natural bound for the convergence error: thanks to the maximum principle one has $||e^n||_{\ell^{\infty}}\leq 2||u_0||_{L^{\infty}}$.
\end{Remark}
%
%
%
%
\section{Convergence for less smooth initial data}
\label{SECTION_NON_REGULIERE}
In this section, we relax the hypothesis $u_0\in H^6(\mathbb{R})$ and adapt the previous proof for any solution in $H^{\frac{3}{4}}(\mathbb{R})$ to obtain Theorem \ref{MAIN_THEOREM_445454545454}. When $u_0$ is not smooth enough to verify $u_0\in H^6(\mathbb{R})$, we regularize it thanks to mollifiers $\left(\varphi^{\delta}\right)_{\delta>0}$, as explained in Introduction. Recall that we denote the mollifiers by $(\varphi^{\delta})_{\delta>0}$, whose construction is based on $\chi$ a $\mathcal{C}^{\infty}$-function such that $\chi\equiv1$ on $[-\frac{1}{2}, \frac{1}{2}]$, $\chi$ is supported in $[-1, 1]$ and $\chi(\xi)=\chi(-\xi)$. We denote the exact solution from $u_0$ by $u$, the exact solution from $u_0\star\varphi^{\delta}$ by $u^{\delta}$ and the numerical solution from \eqref{init_peu_reg} by $(v_j^n)_{(n,j)\in\llbracket0,N\rrbracket\times\mathbb{Z}}$.

\subsection{Approximation results}
\label{SUBSECTION_1_SECTION_NON_REGULIERE}
We need to quantify the dependence of the Sobolev norms of the solution $u^{\delta}$ on $\delta$. That result is gathered in Proposition \ref{COR_1_AA} whose proof needs the following lemma.

\begin{Lemma}
Assume $(m,s) \in \mathbb{R}^2$ with $m\geq s\geq0$. There exists a constant $C>0$ such that, if $u_0\in H^{s}(\mathbb{R})$ and $\delta >0$ and $u_0^{\delta}$ is such as $u_0^{\delta}=u_0\star\varphi^{\delta}$, then 
\begin{equation}
||u_0^{\delta}||_{H^m(\mathbb{R})}\leq \frac{C}{\delta^{m-s}}|| u_0||_{H^{s}(\mathbb{R})}.
\label{equation_derive_taux_cv}
\end{equation}
\label{lemme_derive-taux-cv}
\end{Lemma}

\begin{proof} 
According to \eqref{REM_1_BISBIS}, the $H^m(\mathbb{R})$-norm of $u_0^{\delta}$ verifies
\begin{equation*}
\begin{split}
||u_0\star\varphi^{\delta}||_{H^m\left(\mathbb{R}\right)}^2&=\int_{\mathbb{R}}\left(1+|\xi|^2\right)^m|\chi\left(\delta\xi\right)|^2|\widehat{u_0}\left(\xi\right)|^2d\xi
\leq\int_{\mathbb{R}}\left(1+|\xi|^2\right)^{s}|\widehat{u_0}|^2\left(1+|\xi|^2\right)^{m-s}|\chi\left(\delta\xi\right)|^2d\xi.
\end{split}
\end{equation*}
By hypothesis on $\chi$ and its support, one has $|\chi\left(\delta\xi\right)|\leq1$ and there exists a constant $C>0$ such that $\left(1+|\xi|^2\right)^{m-s}|\chi(\delta\xi)|^2\leq\frac{C}{\delta^{2(m-s)}}$, which concludes the proof.
\end{proof}

We are now able to estimate the Sobolev norms of $u^{\delta}$.
\begin{Proposition}
Assume $m\geq s\geq0$ and $u_0\in H^{s}(\mathbb{R})$ then, 
\begin{equation*}
\underset{t\in[0,T]}{\sup}||u^{\delta}(t,.)||_{H^m(\mathbb{R})}\leq Ce^{\kappa_mT}\frac{||u_0||_{H^{s}(\mathbb{R})}}{\delta^{m-s}},
\end{equation*}
where $C$ is a number which depends on $m$ and $\kappa_m$ depends on $\|u_0\|_{L^2}$ and $m$. Both are independent of $\delta$.
\label{COR_1_AA}
\end{Proposition}

\begin{proof}
We combine Theorem \ref{theobourgain} and Lemma \ref{lemme_derive-taux-cv}.
\end{proof}

We need then to know the rate of convergence of $u_0^{\delta}$ toward $u_0$ with respect to $\delta$ (as $\delta$ tends to 0), which is summarized as follows.
\begin{Lemma}
Assume $u_0 \in H^{s}(\mathbb{R})$ with $0\leq \ell\leq s$, then, there exists a number $C$ independent of $\delta$ such that $$||u_0-u^{\delta}_0||_{H^{\ell}(\mathbb{R})}\leq C \delta^{s-\ell}||u_0||_{H^{s}(\mathbb{R})}.$$
\label{lemme_2}
\end{Lemma}

\begin{proof}
By definition of the $H^{\ell}(\mathbb{R})$-norm, we have, for $s\geq \ell$ :
\begin{equation*}
\begin{split}
||u_0-u^{\delta}_0||^2_{H^{\ell}(\mathbb{R})}&=\int_{\mathbb{R}}(1+|\xi|^{2})^{\ell}|\widehat{u_0}(\xi)|^2\left(1-\chi(\delta\xi)\right)^2d\xi
=\delta^{2(s-\ell)}\int_{\mathbb{R}}(1+|\xi|^2)^{\ell}|\widehat{u_0}(\xi)|^2\left(\frac{1-\chi(\delta\xi)}{(\delta\xi)^{s-\ell}}\right)^2\xi^{2(s-\ell)}d\xi.
\end{split}
\end{equation*}

Hypothesis on $\chi$ implies that $\underset{z\in\mathbb{R}}{\mathrm{sup}}\left|\frac{1-\chi(z)}{z^{s-\ell}}\right|\leq C_2$ for a certain constant $C_2$.
Hence, by using the inequality $(1+|\xi|^2)^{\ell}|\xi|^{2(s-\ell)}\leq C(1+|\xi|^2)^{s}$, with $C$ a constant,
\begin{equation*}
\begin{split}
||u_0-u^{\delta}_0||^2_{H^{\ell}(\mathbb{R})}&\leq \delta^{2(s-\ell)}CC_2^2\int_{\mathbb{R}}\left(1+|\xi|^2\right)^{s}|\widehat{u_0}(\xi)|^2d\xi 
\leq CC_2^2\delta^{2(s-\ell)}||u_0||^2_{H^{s}(\mathbb{R})}.
\end{split}
\end{equation*}

\end{proof}

%
%
\subsection{Proof of Theorem \ref{MAIN_THEOREM_445454545454}}
\label{SUBSECTION_2_SECTION_NON_REGULIERE}
Let $s\geq\frac{3}{4}$. Assume $u_0\in H^{s}(\mathbb{R})$, $T>0$ and $c$ such that \eqref{DEF_c_INTRO} is true, which implies the existence of $\alpha_0$ as in \eqref{alpha0} in Remark \ref{def_alpha_0_8_nov}. We construct $u_0^{\delta}=u_0\star\varphi^{\delta}$ as previously.\\
Let $\beta_0\in(0,1)$, $\theta\in[0,1]$ and $(v_j^n)_{(n,j)\in\llbracket0,N\rrbracket\times\mathbb{Z}}$ the unknown of the numerical scheme \eqref{EQ_SCHEME}-\eqref{init_peu_reg}.
Thanks to Theorem \ref{THEOREM_MAIN111}, there exists $\widehat{\omega_0}>0$ such that for every $\Delta x\leq \widehat{\omega_0}$ and $\Delta t$ satisfying CFL conditions \eqref{CFL_1}-\eqref{CFL_2}, one has

\begin{equation*}
||v^n-(u_{\Delta }^{\delta})^n||_{\ell^2_{\Delta}}\leq \Lambda_{T, \|u_0^{\delta}\|_{H^{\frac{3}{4}}}}\left(1+\|u_0^{\delta}\|_{H^{\frac{1}{2}+\eta}}^2\right)\left(\frac{\|u_0^{\delta}\|_{H^6}}{c+\frac{1}{2}}+\|u_0^{\delta}\|_{H^4}+\|u_0^{\delta}\|_{H^{\frac{3}{2}+\eta}}\|u_0^{\delta}\|_{H^1}\right)\Delta x,
\end{equation*}

with $ \Lambda_{T, \|u_0^{\delta}\|_{H^{\frac{3}{4}}}}$ defined by \eqref{LAMBDA_0}.

\begin{Remark}\label{omegatilde}

For the bound on $\Delta x$, $\widehat{\omega_0}$ in Theorem \ref{THEOREM_MAIN111}, $\min(\tilde\omega_0^\delta,\omega_0)$ is convenient, where, for $\gamma \in (0, 1/2)$, 
\begin{equation}\label{eq_omega_tilde_20_NOV_17}
\begin{split}
\tilde{\omega}_0^\delta=\left[\Lambda_{T, \|u_0^{\delta}\|_{H^{\frac{3}{4}}}}\left(1+\|u_0^{\delta}\|^2_{H^{\frac{1}{2}+\eta}}\right)\left(\frac{\|u_0^{\delta}\|_{H^6}}{c+\frac{1}{2}}+\|u_0^{\delta}\|_{H^4}+\|u_0^{\delta}\|_{H^{\frac{3}{2}+\eta}}\|u_0^{\delta}\|_{H^1}\right)\right]^{-\frac{1}{\gamma}},
\end{split}
\end{equation}
with $\Lambda_{T, \|u_0^{\delta}\|_{H^{\frac{3}{4}}}}$ defined in \eqref{LAMBDA_0}, and $\omega_0$ satisfies 
\eqref{eq_20_NOV_17} and \eqref{24_oct_cond_c_25}-\eqref{24_oct_disp_CFL2_25} if $\theta\geq\frac{1}{2}$ and \eqref{eq_20_NOV_17} and \eqref{COND_ON_c}-\eqref{DISP_CFL} if $\theta<\frac{1}{2}$. The point here is that these inequalities satisfied by $\omega_0$ are valid independently of $\delta$ because $|| u_0^\delta ||_{L^\infty(\mathbb{R})} \leq || u_0 ||_{L^\infty(\mathbb{R})}$. The fact that $\tilde{\omega}_0^\delta$ depends on $\delta$ will bring some difficulty.

\end{Remark}
%

By using a triangle inequality between the analytical solution starting from $u_0$ and the one starting from $u^{\delta}_0$, the global error is upper bounded by 
\begin{equation*}
||e^n||_{\ell^2_{\Delta}}=||v^n-(u_{\Delta})^n||_{\ell^2_{\Delta}}\leq \sqrt{[\Xi_1]^n}+\sqrt{[\Xi_2]^n},
\end{equation*}
with
 \begin{equation*}
[\Xi_1]^n= \left|\left|\left(u_{\Delta}\right)^n-\left[u_{\Delta}^{\delta}\right]^n\right|\right|_{\ell^2_{\Delta}}^2=\sum_{j\in\mathbb{Z}}\Delta x \left(\frac{1}{\Delta x[\min(t^{n+1}, T)-t^n]} \int_{t^n}^{\min(t^{n+1}, T)}\int_{x_j}^{x_{j+1}}u(s,x)-u^{\delta}(s,x)dxds\right)^2,
 \end{equation*}
 with the notation \eqref{NOTATION_BAR},
 and
 \begin{equation*}
 [\Xi_2]^n=\left|\left| \left[u_{\Delta}^{\delta}\right]^n-v^n\right|\right|_{\ell^2_{\Delta}}^2=\sum_{j\in\mathbb{Z}}\Delta x \left(\frac{1}{\Delta x[\min(t^{n+1}, T)-t^n]}\int_{t^{n}}^{\min(t^{n+1}, T)}\int_{x_j}^{x_{j+1}}u^{\delta}(s,x)dxds-v^n_j\right)^2.
 \end{equation*}

Let us first focus on term $[\Xi_1]^n$. The Cauchy-Schwarz inequality implies $[\Xi_1]^n\leq\underset{t\in[0,T]}{\sup}||u(t,.)-u^{\delta}(t,.)||_{L^2(\mathbb{R})}^2$, which leads to study the difference between $u$ and $u^{\delta}$.\\
Since $u$ and $u^{\delta}$ are two solutions of the initial equation \eqref{EQ_INIT}, one has
\begin{equation*}
\partial_t\left(u-u^{\delta}\right)+\partial_x^3\left(u-u^{\delta}\right)+u\partial_x\left(u-u^{\delta}\right)+\left(u-u^{\delta}\right)\partial_xu^{\delta}=0.
\end{equation*}
Multiplying by $\left(u-u^{\delta}\right)$, integrating the equation and changing $u^{\delta}$ in $u-(u-u^{\delta})$ in the latest term yield
\begin{multline*}
\frac{d}{dt}\int_{\mathbb{R}}\frac{\left(u(t,x)-u^{\delta}(t,x)\right)^2}{2}dx-\int_{\mathbb{R}}\partial_x u(t,x)\frac{\left(u(t,x)-u^{\delta}(t,x)\right)^2}{2}dx\\+\int_{\mathbb{R}}\left(u(t,x)-u^{\delta}(t,x)\right)^2\partial_x\left[u(t,x)-\left(u(t,x)-u^{\delta}(t,x)\right)\right]dx=0,
\end{multline*}
thus
\begin{equation*}
\frac{d}{dt}\frac{||u(t,.)-u^{\delta}(t,.)||_{L^2(\mathbb{R})}^2}{2}\leq\frac{||\partial_xu(t,.)||_{L^{\infty}(\mathbb{R})}}{2}||u(t,.)-u^{\delta}(t,.)||_{L^2(\mathbb{R})}^2.
\end{equation*}
The previous inequality looks like the 'weak-strong uniqueness' of DiPerna \cite{DiPerna_1979} or Dafermos \cite{Dafermos_1979, Dafermos_Livre}. The $L^{2}(\mathbb{R})$-norm of the difference $u-u^{\delta}$ is then upper bounded by
\begin{equation*}
\begin{split}
||u(t,.)-u^{\delta}(t,.)||_{L^2(\mathbb{R})}^2&\leq\text{exp}\left(\int_0^t\frac{||\partial_xu(s,.)||_{L^{\infty}(\mathbb{R})}}{2}ds\right)||u_0-u_0^{\delta}||_{L^2(\mathbb{R})}^2\\
&\leq \text{exp}\left(\frac{T^{\frac{3}{4}}C_{\frac{3}{4}}e^{\kappa_{\frac{3}{4}}T}}{2}\|u_0\|_{H^{\frac{3}{4}}}\right)||u_0-u_0^{\delta}||_{L^2(\mathbb{R})}^2,
\end{split}
\end{equation*}
where $\kappa_{\frac{3}{4}}$ and $C_{\frac{3}{4}}$ are defined in Theorem \ref{theobourgain}. Then
\begin{equation*}
[\Xi_1]^n\leq\underset{t\in[0,T]}{\sup}||u(t,.)-u^{\delta}(t,.)||_{L^2(\mathbb{R})}^2\leq \text{exp}\left(\frac{T^{\frac{3}{4}}C_{\frac{3}{4}}e^{\kappa_{\frac{3}{4}}T}}{2}\|u_0\|_{H^{\frac{3}{4}}}\right)||u_0-u_0^{\delta}||_{L^2(\mathbb{R})}^2.
\end{equation*}
Lemma \ref{lemme_2} implies
\begin{equation}
[\Xi_1]^n\leq C^2\delta^{2s}||u_0||_{H^s(\mathbb{R})}^2\text{exp}\left(\frac{T^{\frac{3}{4}}C_{\frac{3}{4}}e^{\kappa_{\frac{3}{4}}T}}{2}\|u_0\|_{H^{\frac{3}{4}}}\right).
\label{UPPER_BOUND_ON_ALPHA00000}
\end{equation}
In the other hand, the term $[\Xi_2]^n$ corresponds to the estimate \eqref{EQ_*} derived in Subsection \ref{SUBSECTION_42} with a smooth initial datum. It remains us to quantify the dependency of its upper bound with respect to $\delta$.
Thanks to Theorem \ref{THEOREM_MAIN111}, one has
\begin{equation*}
\sqrt{[\Xi_2]^n}\leq \Lambda_{T, \|u_0^{\delta}\|_{H^{\frac{3}{4}}}}\left(1+\|u_0^{\delta}\|_{H^{\frac{1}{2}+\eta}}^2\right)\left(\frac{\|u_0^{\delta}\|_{H^6}}{c+\frac{1}{2}}+\|u_0^{\delta}\|_{H^4}+\|u_0^{\delta}\|_{H^{\frac{3}{2}+\eta}}\|u_0^{\delta}\|_{H^1}\right)\Delta x,
\end{equation*}

with $ \Lambda_{T, \|u_0^{\delta}\|_{H^{\frac{3}{4}}}}$ defined by \eqref{LAMBDA_0}. As $u_0$ belongs to $H^s(\mathbb{R})$ with $s\geq \frac{3}{4}$, then $||u_0^{\delta}||_{H^{\frac{3}{4}}}=||u_0||_{H^{\frac{3}{4}}}$ and $||u_0^{\delta}||_{H^{\frac{1}{2}+\eta}}=||u_0||_{H^{\frac{1}{2}+\eta}}.$

\begin{Lemma}
For every $s\geq\frac{3}{4}$, there exists $C$, depending only on $s$ and on $\|u_0\|_{L^2}$, such that, if $u_0\in H^s(\mathbb{R})$, 
\begin{equation*}
\frac{\|u_0^{\delta}\|_{H^6}}{c+\frac{1}{2}}+\|u_0^{\delta}\|_{H^4}+\|u_0^{\delta}\|_{H^{\frac{3}{2}+\eta}}\|u_0^{\delta}\|_{H^1}\leq \frac{\|u_0\|_{H^s}}{\delta^{6-s}}C\left(\frac{1}{c+\frac{1}{2}}+1+\|u_0\|_{H^{\min(1,s)}}\right).
\end{equation*}
\end{Lemma}
\begin{proof}
We apply Lemma \ref{lemme_derive-taux-cv} with $s=6, 4, \frac{3}{2}+\eta, 1$ and the biggest power of $\delta$ is $\frac{1}{\delta^{6-s}}$.
\end{proof}
Thus, an upper bound for $[\Xi_2]^n$ is
\begin{equation*}
\sqrt{[\Xi_2]^n}\leq \Lambda_{T,\|u_0\|_{H^{\frac{3}{4}}}}\left(1+\|u_0\|_{H^{\frac{1}{2}+\eta}}^2\right)\left(\frac{1}{c+\frac{1}{2}}+1+\|u_0\|_{H^{\min(1,s)}}\right) C\frac{\|u_0\|_{H^s}}{\delta^{6-s}}\Delta x.
\end{equation*}

For Theorem \ref{THEOREM_MAIN111} to be applied, we need to choose a small $\Delta x$ such that $\Delta x \leq \min(\tilde{\omega}_0^\delta,\omega_0)$ (see Remark \ref{omegatilde}). With the above lemma, this condition rewrites
\begin{equation}
\Delta x \leq 
\min\left(\left(\frac{\tilde{C}}{\delta^{6-s}}\right)^{-\frac{1}{\gamma}},\omega_0\right) =: \widehat{\omega_0^\delta}.
\label{CONST}
\end{equation}
If this condition is satisfied, and if CFL conditions \eqref{CFL_1}-\eqref{CFL_2} are verified, the convergence error $(e_j^n)_{(n,j)}$ is upper bounded by 
\begin{small}
\begin{multline}
||e^n||_{\ell^2_{\Delta}}\\\leq C\left[ \Lambda_{T, \|u_0\|_{H^{\frac{3}{4}}}}\left(1+\|u_0\|_{H^{\frac{1}{2}+\eta}}^2\right)\left(\frac{1}{c+\frac{1}{2}}+1+\|u_0\|_{H^{\min(1,s)}}\right) + \text{exp}\left(\frac{T^{\frac{3}{4}}C_{\frac{3}{4}}e^{\kappa_{\frac{3}{4}}T}}{4}\|u_0\|_{H^{\frac{3}{4}}}\right)\right]\|u_0\|_{H^s} \left[\frac{\Delta x}{\delta^{6-s}}+\delta^s\right], 
\label{EGALISATION}
\end{multline}
\end{small}
for $n\in\llbracket0,N\rrbracket.$\\

The final key point is to find the optimal $\delta$, in other words, the parameter $\delta$ which makes both terms $\delta^s$ 
(coming from $\sqrt{[\Xi_1]^n}$) and $\frac{\Delta x}{\delta^{6-s}}$ (coming from $\sqrt{[\Xi_2]^n}$) in \eqref{EGALISATION} equal while respecting the constraint \eqref{CONST}. Defining $\delta=\Delta x^{a}$ summarizes the problem in the following system
\begin{equation*}
\left\{
\begin{split}
&\mathrm{Find\ }a\mathrm{\ such\ that\ :\ }\Delta x^{as}=\frac{\Delta x}{\Delta x^{a(6-s)}},\\
&\mathrm{under\ the\ constraint\ :\ }\frac{1}{\Delta x^{a(6-s)}}<\frac{1}{\Delta x^{\gamma}} \mbox{ and } \Delta x \leq \omega_0.
\end{split}
\right.
\end{equation*}

Three cases have to be considered:
\begin{itemize}
\item if $\frac{3}{4}\leq s\leq 6-6\gamma$, the constraint is binding and we have to choose $a$ which transforms the constraint inequality in an equality : $a=\frac{\gamma}{6-s}$. In that case, the rate of convergence is given by the smallest term between $\Delta x^{as}$ and $\frac{\Delta x}{\Delta x^{a(6-s)}}$ \textit{i.e.} $\Delta x^{\frac{\gamma s}{6-s}}$.
\item If $6-6\gamma\leq s\leq 6$, $a=\frac{1}{6}$ enables both terms $\Delta x^{as}$ and $\frac{\Delta x}{\Delta x^{a(6-s)}}$ to be equal without violating the constraint. This choice of $a$ gives a rate of convergence of $\Delta x^{\frac{s}{6}}$.
\item If $s\geq6$, the result of the Theorem \ref{THEOREM_MAIN111} applies.
\end{itemize}
Since $\gamma$ is in $(0,\frac{1}{2})$ (cf. Lemma \ref{Prop_B,5} and induction hypothesis \eqref{HR}), we take the optimal $\gamma$ : $\gamma=\frac{1}{2}-\eta$ with $\eta$ small and $\eta>0$.
The conclusion of the theorem is straightforward consequence.
\begin{Remark}
The choice of $\delta$ is independent of the regularity $s$ of the initial datum, if $3\leq s\leq6$.
\end{Remark}

\begin{Remark}
Notice that in the latter result, the error is defined as the difference between the exact solution and the numerical solution obtained with a smoothed initial condition with a certain parameter $\delta$. To be more complete and estimate the error between the exact solution and the numerical one would require some stability estimate for the scheme that would allow to compare two numerical solutions with different initial data, in the spirit of he stability estimate recalled in Remark \ref{estimstab}. This precise result seems very difficult to state. 
\end{Remark}

%
%
%
%
\section{Numerical results}
\label{NUMERICAL_RESULTS}

In this section, the previous results are illustrated numerically by some examples and the numerical convergence rates are computed for the KdV equation. 

\subsection{Convergence rates}
Through the rest of the paper, the computations are performed with an implicit scheme $\theta=1$ in order to avoid the dispersive CFL condition. Our purpose is to gauge the relevance of our theoretical results on the rate of convergence with respect to $\Delta x$. To this end, the time step is chosen according to the hyperbolic CFL condition. More precisely, $c$ is numerically chosen such that $c^n=\underset{j\in\llbracket1,J\rrbracket}{\max|v_j^k|}$ and $\Delta t^n=\frac{\Delta x}{c^n}$. This choice seems surprising related to the CFL of Theorems \ref{THEOREM_MAIN111} and \ref{MAIN_THEOREM_445454545454} but, as explained in Remark \ref{REM_INtro_}, the condition $[c+\frac{1}{2}]\Delta t<\Delta x$ seems technical and may be replaced with the classical one $c\Delta t\leq \Delta x$. Eventually, we fix the final time $T=0.1$.

We can not simulate numerical solutions on $\mathbb{Z}$ as done in the theoretical results. We have to take into account numerical boundaries: we use periodic boundaries. We fix the space domain to $[0,L]$ with $L=50$ (except for the cnoidal wave where $L=1$) and fix $J \in \mathbb{N}^*$ and $\Delta x = L/J$.
\begin{Remark}
Notice that the theoretical results do not apply rigorously since the solutions do not belong to $H^s(\mathbb{R})$ because of their periodicity.
\end{Remark}

When the exact solution is known (e.g. for the cnoidal-wave solution), the variable $E_J$ denotes the error with $J$ cells and is defined as 
\begin{equation*}
E_J=\underset{n\in\llbracket0,N\rrbracket}{\mathrm{sup}}||\left(e^n_j\right)_{j\in \llbracket0,J\rrbracket}||_{\ell^2_{\Delta}}=\underset{n\in\llbracket0,N\rrbracket}{\mathrm{sup}}\left|\left|\left(v^n_j\right)_{j\in \llbracket0,J\rrbracket}-\left(\left[u_{\Delta}\right]^n_j\right)_{j\in \llbracket0,J\rrbracket}\right|\right|_{\ell^2_{\Delta}},
\end{equation*}
with $(v_j^n)_{j\in\llbracket0,J\rrbracket}$ the numerical solution computed with $J$ cells in space and $\left(\left[u_{\Delta}\right]^n_j\right)_{j\in \llbracket0,J\rrbracket}$ the $J$-piecewise constant function from the analytical solution. \\
When the exact solution is not known, the convergence error is computed from two numerical solutions with different meshes, $v$ with $J$ cells and $\overline{v}$ with $2J$ cells, and $E_J$ is replaced with the following $\tilde{E}_J$: 
\begin{equation*}
\tilde{E}_J=\underset{n\in\llbracket0,N\rrbracket}{\mathrm{sup}}\left|\left|\left(v^n_j\right)_{j\in \llbracket0,J\rrbracket}-\left(\tilde{v}^n_j\right)_{j\in \llbracket0,J\rrbracket}\right|\right|, 
\end{equation*}
where $\tilde{v}_j^n = \overline{v}_{2j}^n$ for any $j$ and any $n$. 
In that case, $\left(\tilde{v}^n_j\right)_{j\in \llbracket0,J\rrbracket}$, computed from the refined numerical solution $\left(w^n_j\right)_{j\in \llbracket0,2J\rrbracket}$, plays the role of the exact one $\left(\left[u_{\Delta}\right]^n_j\right)_{j\in \llbracket0,J\rrbracket}$.

The "convergence rate" $r_J$ is computed as 
\begin{equation*}
r_J = \frac{\log\left(E_J\right)-\log\left(E_{2J}\right)}{\log(2)}, \mbox{ or } r_J = \frac{\log\left(\tilde{E}_J\right)-\log\left(\tilde{E}_{2J}\right)}{\log(2)}
\end{equation*}


\subsection{Smooth initial data}

To assess the optimality of Theorem \ref{THEOREM_MAIN111}, the corresponding test cases are carried out with two smooth periodic initial data, either the sinusoidal initial datum
\begin{equation*}
u_0(x)=\cos\left(\frac{2\pi}{L}x\right),
\end{equation*}
or the so-called cnoidal-wave initial datum. This cnoidal-wave solution represents a periodic solitary wave solution of the Korteweg-de Vries equation whose analytical expression is known as follow: 
\begin{equation*}
u(t,x)=\frac{1}{\mu^{\frac{1}{5}}}a\mathrm{cn}^2\left(4K(m)\left(\mu^{\frac{2}{5}}\left(x-\frac{L}{2}\right)-v\mu^{\frac{1}{5}}t\right)\right),
\end{equation*}
where $\mu=\frac{1}{24^2}$ and $\mathrm{cn}(z)=\mathrm{cn}(z:m)$ is the Jacobi elliptic function with modulus $m\in(0,1)$ (we choose $m=0.9$) and the parameters have the values $a=192m\mu K(m)^2$ and $v=64\mu(2m-1)K(m)^2$. $K(m)$ is the complete elliptic integral of the first kind (cf \cite{Bona_Chen_Karakashian_2013}).\\ 
Both results are gathered in Figure \ref{sinus} for sinusoidal solution and Figure \ref{cnoidal} for cnoidal-wave solution. We display the values of $r$ with respect to $J$ in the left table and post the corresponding graph in logarithmic scale on the right. The first order is confirmed for both initial data whether in tables or in graphs.

\begin{minipage}{1\linewidth}
\begin{minipage}[r]{0.55\linewidth}
\begin{tabular}{|c|c|c|c|}
\hline
&&error in&numerical\\
$J$&$\Delta x$&$\ell^{\infty}(0,T, \ell^2_{\Delta}(\mathbb{Z}))$&order\\
&&computed with $E_J$&\\
\hline
1600&$3,1250.10^{-2}$&$6,2062.10^{-5}$&\\
3200&$1,5625.10^{-2}$&$3,1033.10^{-5}$&$0.9999$\\
6400&$7,8125.10^{-3}$&$1,5517.10^{-5}$&$0.9999$\\
12800&$3,9063.10^{-3}$&$8,0795.10^{-6}$&$0.9415$\\
25600&$1,9531.10^{-3}$&$4,1435.10^{-6}$&$0.9634$\\
51200&$9,7656.10^{-4}$&$1,9974.10^{-6}$&$1.0527$\\
\hline
\end{tabular}
\end{minipage} \hfill
\begin{minipage}[l]{0.55\linewidth}
\includegraphics[width=55mm]{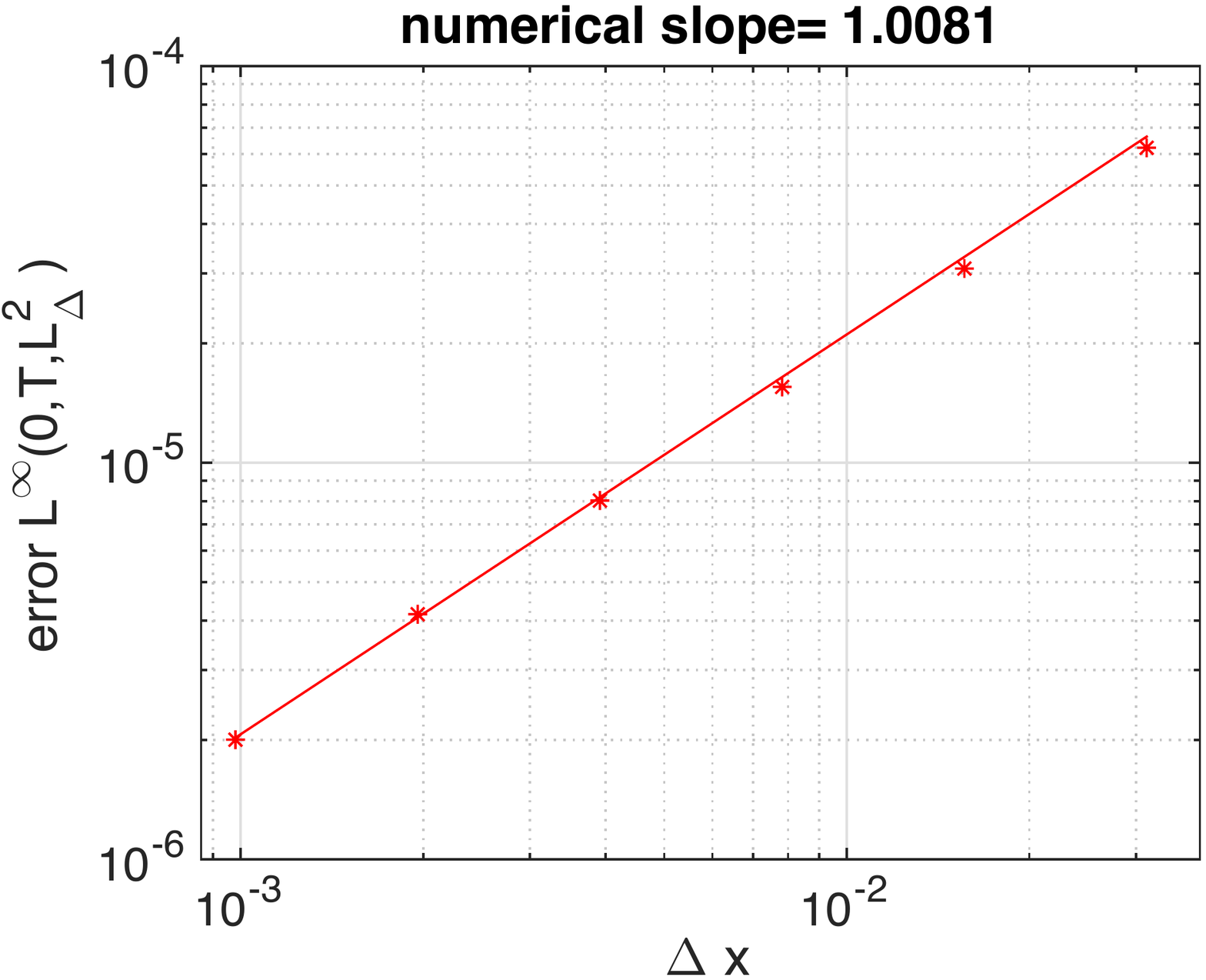}
\end{minipage}
\captionof{figure}{Experimental rate of convergence for sinusoidal solution}
\label{sinus}
\end{minipage}

\begin{minipage}{1\linewidth}
\begin{minipage}[r]{0.55\linewidth}
\begin{tabular}{|c|c|c|c|}
\hline
&&error in&numerical\\
$J$&$\Delta x$&$\ell^{\infty}(0,T, \ell^2_{\Delta} (\mathbb{Z}))$&order\\&&computed with $E_J$&\\
\hline
1600&$6.2500.10^{-4}$&$8.9875.10^{-4}$&\\
3200&$3.1250.10^{-4}$&$4.5253.10^{-4}$&$0.9899$\\
6400&$1.5625.10^{-4}$&$2.2636.10^{-4}$&$0.9994$\\
12800&$7.8125.10^{-5}$&$1.1292.10^{-4}$&$1.0034$\\
25600&$3.9062.10^{-5}$&$5.7102.10^{-5}$&$0.9837$\\
\hline
\end{tabular}
\end{minipage} \hfill
\begin{minipage}[l]{0.55\linewidth}
\includegraphics[width=55mm]{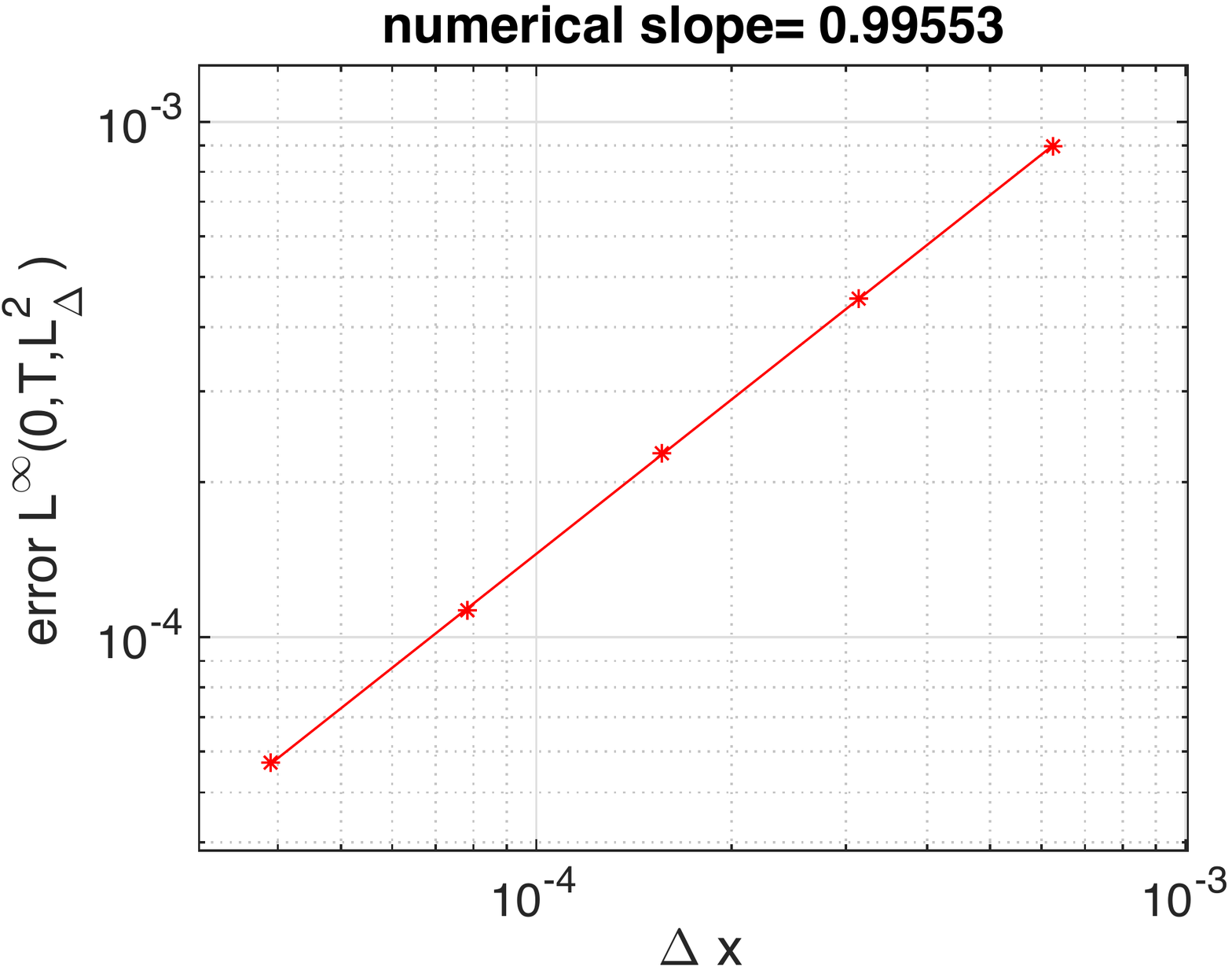}
\end{minipage}
\captionof{figure}{Experimental rate of convergence for cnoidal-wave solution}
\label{cnoidal}
\end{minipage}

\subsection{Less smooth initial data} 

To illustrate numerically Theorem \ref{MAIN_THEOREM_445454545454}, 
we here initialize the scheme with a less regular initial datum. We test two kinds of periodic data in 
$H^s([0, L])$, with $s\geq0$. We will test both integer and half-integer values of $s$. 

\textbf{Tests achieved with {\em half-integer} $s$, from the indicator function.} 
Since the indicator function $\mathds{1}_{[0,\frac{L}{2}]}$ belongs to $H^{s}([0,L])$ for all $s<\frac{1}{2}$, an idea to construct a periodic function in $H^{s+\ell}\left([0,L]\right)$, with $s<\frac{1}{2}$ and $\ell\in\mathbb{N^*}$ is to integrate $\ell$ times the periodic indicator function. 
For instance, after a first integration, the initial datum
\begin{equation*}
u_0(x)=x\mathds{1}_{[0,\frac{L}{2}]}+(L-x)\mathds{1}_{[\frac{L}{2},L]}
\end{equation*}
is periodic and "almost" in $H^{\frac{3}{2}}([0,L])$. By reiterating the process of periodization and integration, we obtain initial data in $H^{s}([0,L])$, with $s=\frac{7}{2}^-, \frac{9}{2}^-,\frac{11}{2}^-$... 

\textbf{Tests achieved with {\em integer} $s$, from the square root function.} 
Since the square root function is in $H^{1-}([0,L])$ we construct a $H^{s-}([0,L])$ 
function by integrating the square root function $s-1$ times. However, we need, in addition, a periodic 
initial datum, this is why we add the beginning of a Taylor expansion for the function and its derivatives 
up to $(s-1)$-th to be continuous and periodic.
More precisely, we search the coefficients $b_i$, $i\in\llbracket1,s\rrbracket$ such that the function
\begin{equation*}
x^{s-1+\frac{1}{2}}-b_1x-\frac{b_2}{2}x^2-\frac{b_3}{3!}x^3...-\frac{b_{s}}{s!}x^{s}
\end{equation*}
and all its derivatives up to $(s-1)$-th be equal for $x=0$ and for $x=L.$
To find those coefficients, we just have to solve a triangular linear system.

Theoretically, the necessity to bound $\int_{0}^{T}||\partial_xu(s,.)||^i_{L^{\infty}(\mathbb{R})}ds$ in \eqref{EQ_--} 
forces to choose $s\geq\frac{3}{4}$. In addition, the necessity to bound $||e^n||_{\ell^{\infty}}$ in $F_a$ in 
\eqref{DEF_Ba} in order to apply the Gr\"onwall lemma leads to choose $\Delta x$ such that Equation 
\eqref{EQ_CONDITION} is true, which leads to the constraint $\frac{1}{\delta^{6-s}}<\frac{1}{\Delta x^{\gamma}}$ 
in \eqref{CONST}. However, those restrictions may be only technical and the rate of convergence seems to be 
$\Delta x^{\frac{s}{6}}\mathrm{\ for\ all\ } s\in[0,3)$, as the following numerical results indicate. 

Figures \ref{H12} and \ref{H1} below report the experiments done for $s = 0.5^-$ and $s = 1^-$. Table \ref{ocv} gives the results we have obtained with the same technique, for various $s$ values between $0.5^-$ and $8^-$. The results are compared with the results proved in the present paper and the conjectures ones. 

\begin{minipage}{1\linewidth}
\begin{minipage}[r]{0.55\linewidth}
\begin{tabular}{|c|c|c|c|}
\hline
&&error in&numerical\\
$J$&$\Delta x$&$\ell^{\infty}(0,T, \ell^2_{\Delta} (\mathbb{Z}))$&order\\&&computed with $\tilde{E}_J$&\\
\hline
3200&$1.5625.10^{-2}$&$1.0567.10^{-2}$&\\
6400&$7.8125.10^{-3}$&$9.8843.10^{-3}$&$0.0964$\\
12800&$3.9063.10^{-3}$&$9.2992.10^{-3}$&$0.0880$\\
25600&$1.9531.10^{-3}$&$8.7490.10^{-3}$&$0.0879$\\
51200&$9.7656.10^{-4}$&$8.2289.10^{-3}$&$0.0885$\\
102400&$4.8828.10^{-4}$&$7.7468.10^{-3}$&$0.0871$\\
\hline
\end{tabular}
\end{minipage} \hfill
\begin{minipage}[l]{0.55\linewidth}
\includegraphics[width=55mm]{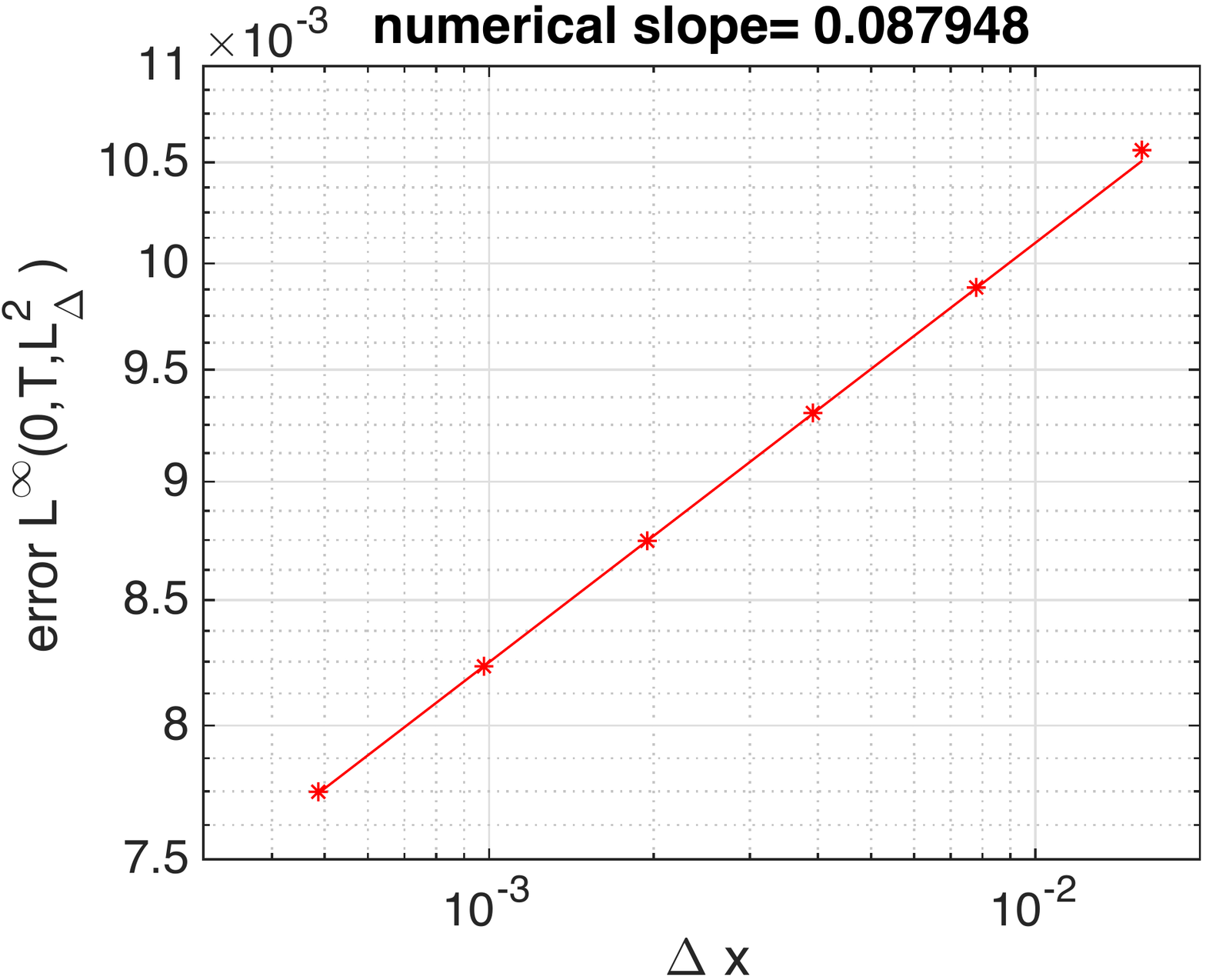}
\end{minipage}
\captionof{figure}{Experimental rate of convergence for $u_0\in H^{\frac{1}{2}-}([0,L])$}
\label{H12}
\end{minipage}

\begin{minipage}{1\linewidth}
\begin{minipage}[r]{0.55\linewidth}
\begin{tabular}{|c|c|c|c|}
\hline
&&error in&numerical\\
$J$&$\Delta x$&$\ell^{\infty}(0,T, \ell^2_{\Delta} (\mathbb{Z}))$&order\\&&computed with $\tilde{E}_J$&\\
\hline
1600&$3.1250.10^{-2}$&$2.6762.10^{-2}$&\\
3200&$1.5625.10^{-2}$&$2.3501.10^{-2}$&$0.18748$\\
6400&$7.8125.10^{-3}$&$2.0793.10^{-2}$&$0.17660$\\
12800&$3.9063.10^{-3}$&$1.8595.10^{-2}$&$0.16119$\\
25600&$1.9531.10^{-3}$&$1.6602.10^{-2}$&$0.16360$\\
51200&$9.7656.10^{-4}$&$1.4787.10^{-2}$&$0.16701$\\
\hline
\end{tabular}
\end{minipage} \hfill
\begin{minipage}[l]{0.55\linewidth}
\includegraphics[width=55mm]{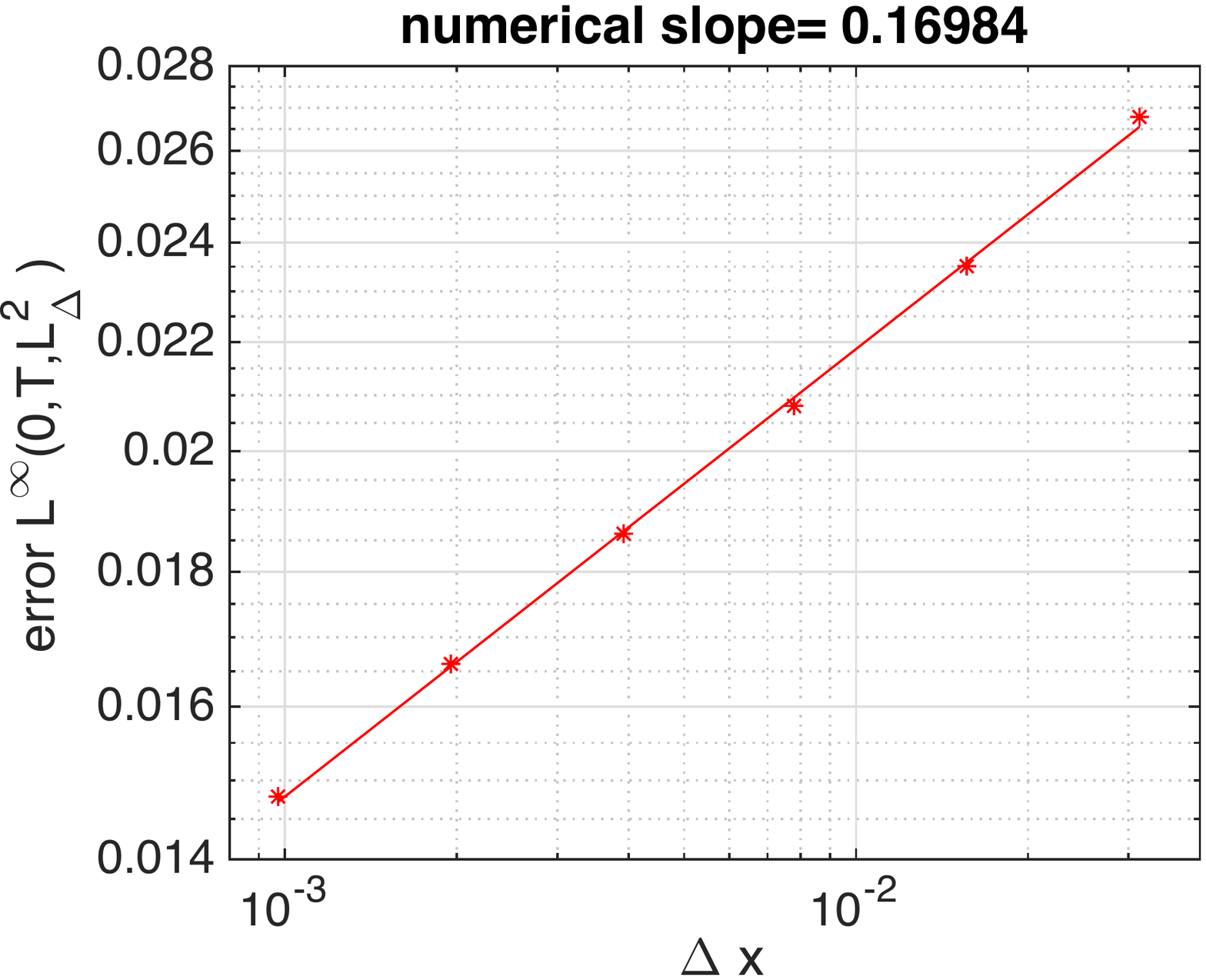}
\end{minipage}
\captionof{figure}{Experimental rate of convergence for $u_0\in H^{1-}([0,L])$}
\label{H1}
\end{minipage}

\begin{table}
\begin{tabular}{|c|c|c|c|}
\hline
Sobolev index & Proved convergence rate & Experimental convergence rate & Conjectured experimental rate\\
\hline
$0.5^-$ & 0.0455 & 0.08795 & 0.08333 \\
\hline
$1^-$ & 0.1000 & 0.16984 & 0.16667 \\
\hline
$1.5^-$ & 0.1667 & 0.25500 & 0.25000 \\
\hline
$2^-$ & 0.2500 & 0.33806 & 0.33333 \\
\hline
$2.5^-$ & 0.3571 & 0.42595 & 0.41667 \\
\hline
$3^-$ & 0.5000 & 0.50173 & 0.50000 \\
\hline
$3.5^-$ & 0.58333 & 0.66016 & cf. proved\\
\hline
$4^-$ & 0.66667 & 0.67225 & cf. proved \\
\hline
$4.5^-$ & 0.75000 & 0.78307 & cf. proved\\
\hline
$5^-$ & 0.83333 & 0.86032 & cf. proved\\
\hline
$5.5^-$ & 0.91667 & 0.97340 & cd. proved\\
\hline
$6^-$ & 1.0000 & 0.98708 & cf. proved\\
\hline
$7^-$ & 1.0000 & 0.99485 & cf. proved\\
\hline
$8^-$ & 1.0000 & 1.0060 & cf. proved\\
\hline
\end{tabular} 
\caption{Convergence order with respect to regularity. }\label{ocv}
\end{table}

Remark the relative error between the experimental rate and the theoretical one is sometimes significant, for example, this relative error is more than $12\%$ in the case $s=\frac{7}{2}-$. However, the theoretical rate is an {\em asymptotic} result for $\Delta x$ and $\Delta t$ small enough. We do not think the difference is significant here. 

We summarize the theoretical and numerical results in Figure \ref{Ordre_fct_regularite}. The blue line corresponds to the proved rate of convergence, the dashed line matches the conjectured rate and the red dots stand for the numerical rates of convergence. Both are intertwined, which validates the rate of convergence of $\frac{\mathrm{min}(s,6)}{6}$ with $s$ the Sobolev regularity of the initial value.\\

\begin{figure}[h]
\begin{center}
\includegraphics[height=100mm]{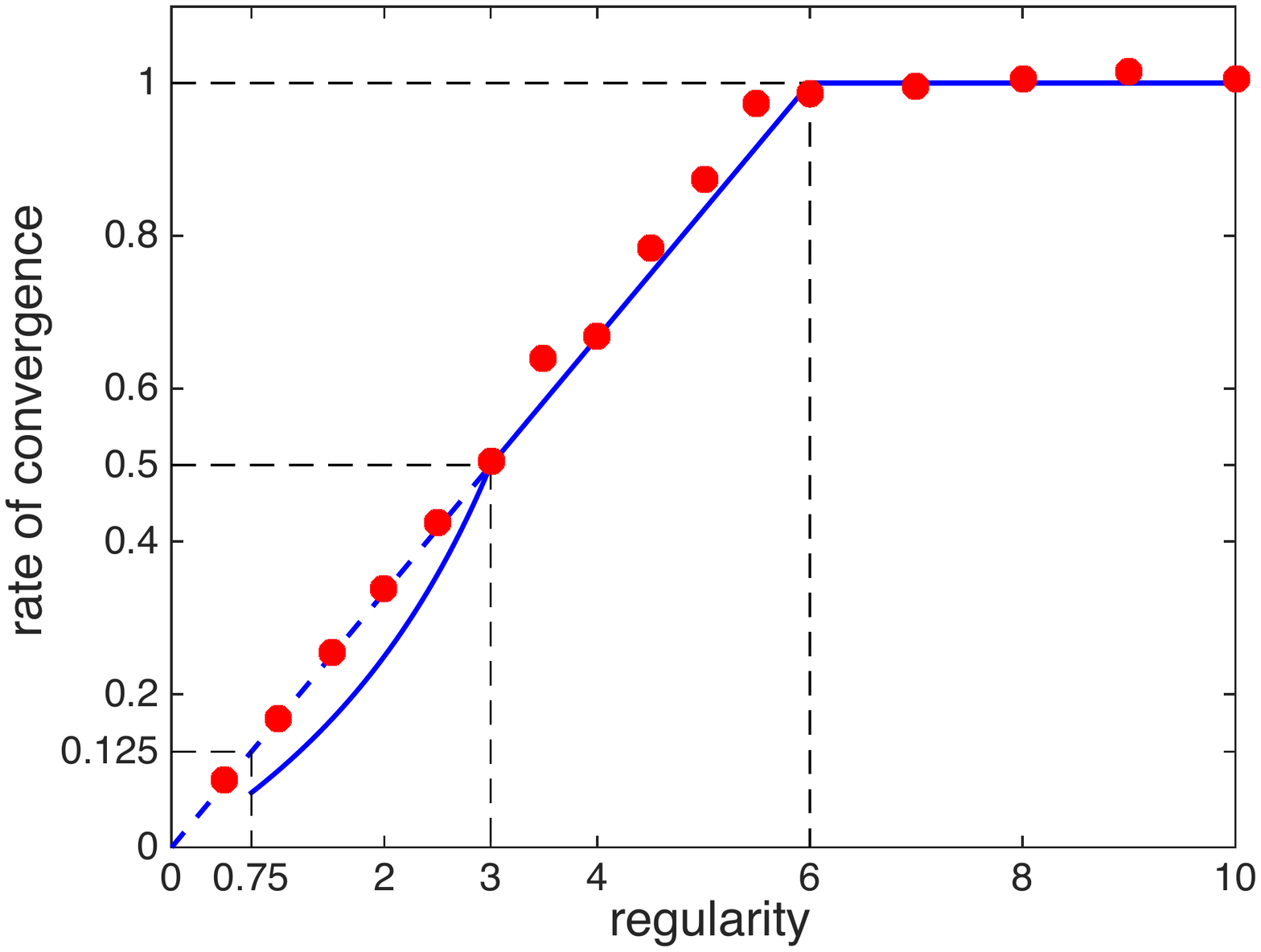}
\end{center}
\caption{Rates of convergence according to the Sobolev regularity of $u_0$. -- Rates proved in this paper (solid line) \textit{versus} experimental rates (dots)}
\label{Ordre_fct_regularite}
\end{figure}

%
%
%
%
\appendix
\section{Appendix :  proof of Proposition \ref{propositionun} on the consistency error}
\label{APPENDIX_A}
Let us recall that the consistency error is defined by \eqref{CONSISTENCY_DEF}. 

The  main technical part of the proof will be to establish that the consistency error satisfies the following inequality 
\begin{equation}
\label{appendixconsistance}
||\epsilon^n||_{\ell^{\infty}(\llbracket0,N\rrbracket;\ell^2_{\Delta})}\leq B_1 \left\{\Delta t\underset{t\in[0,T]}{\mathrm{sup}}\left[\left(1+||u||_{L^{\infty}_x}^2\right)||u||_{H^{6}_x}\right]+\Delta x\underset{t\in[0,T]}{\mathrm{sup}}\left[\left(1+||u||_{L^{\infty}_x}\right)||u||_{H^{4}_x}+||\partial_xu||_{L^{\infty}_x}||u||_{H^1_x}\right]\right\},
\end{equation}
where $B_1$ is a constant that does not depend on $u$, $u_0$, $T$, $\Delta t$ nor $\Delta x$.

Assuming that \eqref{appendixconsistance} is established, we can first easily finish the proof of Proposition
 \ref{propositionun}. Indeed, 
by using the Sobolev embedding $H^{\frac{1}{2}+\eta}(\mathbb{R})\hookrightarrow L^{\infty}(\mathbb{R})$, with $\eta>0$, we obtain
\begin{equation*}
||\epsilon^n||_{\ell^{\infty}(\llbracket0,N\rrbracket;\ell^2_{\Delta})}\leq B_1\left\{ \Delta t\underset{t\in[0,T]}{\mathrm{sup}}\left[\left(1+||u||_{H^{\frac{1}{2}+\eta}_x}^2\right)||u||_{H^{6}_x}\right]+\Delta x\underset{t\in[0,T]}{\mathrm{sup}}\left[\left(1+||u||_{H^{\frac{1}{2}+\eta}_x}\right)||u||_{H^{4}_x}+||u||_{H^{\frac{3}{2}+\eta}_x}||u||_{H^1_x}\right]\right\}.
\end{equation*}

Theorem \ref{theobourgain} enables to rewrite
\begin{multline*}
||\epsilon^n||_{\ell^{\infty}(\llbracket0,N\rrbracket;\ell^2_{\Delta})}\leq \Delta t\ B_1C_{6}C^2_{\frac{1}{2}+\eta}e^{(2\kappa_{\frac{1}{2}+\eta}+\kappa_6) T}\left[\left(1+||u_0||^2_{H^{\frac{1}{2}+\eta}}\right)||u_0||_{H^{6}}\right]\\
+\Delta x\ \overline{C}e^{\overline{\kappa} T}\left[\left(1+||u_0||_{H^{\frac{1}{2}+\eta}}\right)||u_0||_{H^{4}}+||u_0||_{H^{\frac{3}{2}+\eta}}||u_0||_{H^1}\right],
\end{multline*}
with
$\overline{C}=\max\left(B_1C_{\frac{1}{2}+\eta}C_4, B_1C_{\frac{3}{2}+\eta}C_1, B_1C_4\right)$ and $\overline{\kappa}=\max\left(\kappa_{\frac{1}{2}+\eta}+\kappa_4, \kappa_{\frac{3}{2}+\eta}+\kappa_1, \kappa_4\right)$.

\noindent Inequality \eqref{BORNE_33333333333} follows from the fact that there exists a constant $B_2$ (for example $B_2=\frac{1}{2\sqrt{2} - 2}$) such that $$\left(1+||u_0||_{H^{\frac{1}{2}+\eta}}\right)\leq B_2 \left(1+||u_0||^2_{H^{\frac{1}{2}+\eta}}\right).$$
We fix $C=\max\left(B_1C_6C^2_{\frac{1}{2}+\eta}, B_2\overline{C}\right)$ and $\kappa=\max\left(2\kappa_{\frac{1}{2}+\eta}+\kappa_6, \overline{\kappa}\right)$.

It remains to prove \eqref{appendixconsistance}.

For the sake of simplicity, we here assume that $t^{n+1} \leq T$. 
Note that $\epsilon_{j}^n$ can be rewritten as
\begin{equation}
\begin{split}
\epsilon_j^n&=\frac{1}{\Delta t^2\Delta x}\int_{t^n}^{t^{n+1}}\int_{x_j}^{x_{j+1}}u(s+\Delta t,y)-u(s,y)dyds\\
&+\frac{1}{4\Delta x}\left[\left(\frac{1}{\Delta t\Delta x}\int_{t^n}^{t^{n+1}}\int_{x_j}^{x_{j+1}}u(s,y+\Delta x)dyds\right)^2-\left(\frac{1}{\Delta x\Delta t}\int_{t^n}^{t^{n+1}}\int_{x_j}^{x_{j+1}}u(s,y-\Delta x)dyds\right)^2\right]\\
&+\frac{1-\theta}{\Delta t\Delta x^4}\int_{t^n}^{t^{n+1}}\int_{x_j}^{x_{j+1}}u(s,y+2\Delta x)-3u(s,y+\Delta x)+3u(s,y)-u(s,y-\Delta x)dyds\\
&+\frac{\theta}{\Delta t\Delta x^4}\int_{t^{n+1}}^{t^{n+2}}\int_{x_j}^{x_{j+1}}u(s,y+2\Delta x)-3u(s,y+\Delta x)+3u(s,y)-u(s,y-\Delta x)dyds\\
&-c\left(\frac{1}{2\Delta t\Delta x^2}\int_{t^n}^{t^{n+1}}\int_{x_j}^{x_{j+1}}u(s,y+\Delta x)-2u(s,y)+u(s,y-\Delta x)dyds\right).
\label{CONSISTENCY_ERROR}
\end{split}
\end{equation}
We only give details for  the expansion of  the nonlinear term (the other terms are easier and can be handled by similar arguments) :
$$
NL:=\left[\left(\frac{1}{\Delta t\Delta x}\int_{t^n}^{t^{n+1}}\int_{x_j}^{x_{j+1}}u(s,y+\Delta x)dyds\right)^2-\left(\frac{1}{\Delta x\Delta t}\int_{t^n}^{t^{n+1}}\int_{x_j}^{x_{j+1}}u(s,y-\Delta x)dyds\right)^2\right].
$$
Let us introduce, for $\nu$ in $\mathbb{R}$
\begin{equation*}
K(\nu):=\left(\frac{1}{\Delta x\Delta t}\int_{x_j}^{x_{j+1}}\int_{t^n}^{t^{n+1}}u(s, y+\nu\Delta x)dsdy\right)^2.
\end{equation*}
The nonlinear term in Equation \eqref{CONSISTENCY_ERROR} rewrites
\begin{equation*}
NL=K(1)-K(-1)=2K'(0)+\int_{0}^{1}K''(w)(1-w)dw+\int_{0}^{1}K''(-w)(-1+w)dw.
\end{equation*}
A straightforward computation yields
\begin{equation*}
\begin{split}
K'(0)&=\frac{2}{\Delta x\Delta t^2}\int_{x_j}^{x_{j+1}}\int_{t^n}^{t^{n+1}}\int_{x_j}^{x_{j+1}}\int_{t^n}^{t^{n+1}} \partial_xu(\bar{s},\bar{y})u(s,y)d\bar{s}d\bar{y}dsdy\\
&=\frac{2}{\Delta x\Delta t^2}\int_{x_j}^{x_{j+1}}\int_{t^n}^{t^{n+1}}\int_{x_j}^{x_{j+1}}\int_{t^n}^{t^{n+1}}\left[\partial_xu(s,y)+\int_{y}^{\bar{y}}\partial_x^2u(s, v)dv+\int_{s}^{\bar{s}}\partial_{xt}u(\tau, \bar{y})d\tau\right]u(s,y)d\bar{s}d\bar{y}dsdy\\
&=\frac{2}{\Delta t}\int_{x_j}^{x_{j+1}}\int_{t^n}^{t^{n+1}}u(s, y)\partial_xu(s, y)dsdy+\frac{2}{\Delta t\Delta x}\int_{x_j}^{x_{j+1}}\int_{x_j}^{x_{j+1}}\int_{t^n}^{t^{n+1}}u(s, y)\int_{y}^{\bar{y}}\partial_x^2u(s, v)dvdsd\bar{y}dy\\
&\hspace*{5cm}+\frac{2}{\Delta t^2\Delta x}\int_{x_j}^{x_{j+1}}\int_{x_j}^{x_{j+1}}\int_{t^n}^{t^{n+1}}\int_{t^n}^{t^{n+1}}u(s, y)\int_{s}^{\bar{s}}\partial_{xt}u(\tau, \bar{y})d\tau d\bar{s}dsd\bar{y}dy,
\end{split}
\end{equation*}
and thanks to the Cauchy-Schwarz inequality, we obtain
\begin{multline*}
|K''(\nu)|^2\leq C\left[\frac{\Delta x^3}{\Delta t^2}\int_{t^n}^{t^{n+1}}||u(\bar{s},.)||_{L^{\infty}_x}^2\int_{x_j}^{x_{j+1}}\int_{t^n}^{t^{n+1}}\left(\partial_x^2u(s, y+\nu\Delta x)\right)^2dsdyd\bar{s}\right.\\\left.+\left(\frac{2\Delta x}{\Delta t}\int_{t^n}^{t^{n+1}}\int_{x_j}^{x_{j+1}}\left(\partial_xu(s, y+\nu\Delta x)\right)^2dsdy\right)^2\right].
\end{multline*}

By using similar expansions for the other terms in \eqref{appendixconsistance} 
 and the fact that  $u$ satisfies \eqref{EQ_INIT}, we deduce by using  the Cauchy-Schwarz inequality 
 to estimate the remainders that 
\begin{equation}
\begin{split}
&||\epsilon^n||_{\ell^2_{\Delta}}^2\leq C\left[\Delta t^2\underset{t\in[0,T]}{\mathrm{sup}}||\partial_t^2u(t,.)||_{L^2_x}^2+\Delta x^2\underset{t\in[0,T]}{\mathrm{sup}}||u(t,.)||_{L^{\infty}_x}^2\underset{t\in[0,T]}{\mathrm{sup}}||\partial_x^2u(t,.)||_{L^2_x}^2+\Delta x^2\underset{n\in\llbracket0,N\rrbracket}{\mathrm{sup}}||\partial_x^4u||_{L^2_x}^2\right.\\
&\left.+\Delta t^2\underset{t\in[0,T]}{\mathrm{sup}}||u(t,.)||_{L^{\infty}_x}^2\underset{t\in[0,T]}{\mathrm{sup}}||\partial_{xt}u(t,.)||_{L^2_x}^2+\Delta x^2\underset{t\in[0,T]}{\mathrm{sup}}||\partial_xu(t,.)||_{L^{2}_x}^2\underset{t\in[0,T]}{\mathrm{sup}}||\partial_xu(t,.)||_{L^{\infty}_x}^2+\Delta x^2\underset{n\in\llbracket0,N\rrbracket}{\mathrm{sup}}||\partial_x^2u||_{L^2_x}^2\right].
\label{EQ_4|}
\end{split}
\end{equation}

Let us then compute $||\partial_t^2u||_{L^2_x}$ in \eqref{EQ_4|}.
 Thanks to the KdV equation, the time derivative is equal to
 \begin{equation*}
 \partial_t^2u=2u\left(\partial_xu\right)^2+u^2\partial_x^2u+5\partial_xu\partial_x^3u+2u\partial_x^4u+3\left(\partial_x^2u\right)^2+\partial_x^6u.
 \end{equation*}
 For the term $\partial_xu\partial_x^3u$, we use then the relation, for all $u$ and $v$ in $H^{\alpha+\beta}(\mathbb{R})$
 \begin{equation}
 \left|\left|\partial_x^{\alpha}u\partial_x^{\beta}v\right|\right|_{L^2(\mathbb{R})}\leq C\left[\left|\left|u\right|\right|_{L^{\infty}(\mathbb{R})}\left|\left|v\right|\right|_{H^{\alpha+\beta}(\mathbb{R})}+\left|\left|v\right|\right|_{L^{\infty}(\mathbb{R})}\left|\left|u\right|\right|_{H^{\alpha+\beta}(\mathbb{R})}\right].
 \label{relation_sept_2016}
 \end{equation}
  Hence
 \begin{equation*}
 \hspace*{-0.4cm}||\partial_t^2u||_{L^2_x}\leq C\left[ ||u||_{L^{\infty}_x}||\partial_xu||_{L^4_x}^2+||u||_{L^{\infty}_x}^2||\partial_x^2u||_{L^2_x}+||u||_{L^{\infty}_x}||\partial_x^4u||_{L^2_x}+||u||_{L^{\infty}_x}||\partial_x^4u||_{L^2_x}+||\partial_x^2u||_{L_x^4}^2+||\partial_x^6u||_{L^2_x}\right].
 \end{equation*}
 
 For the term $ \left|\left|\partial_xu\right|\right|_{L^4_x}$, we use an integration by parts and the Cauchy-Schwarz inequality to obtain
 \begin{equation*}
 \left|\left|\partial_xu\right|\right|_{L^4_x}^4=\int_{\mathbb{R}}\left(\partial_xu(x)\right)^3 \partial_xu(x)dx=-\int_{\mathbb{R}}3u(x)\partial_x^2u(x)\left(\partial_xu(x)\right)^2dx\leq 3\left|\left|u\right|\right|_{L^{\infty}_x}\left|\left|\partial_x^2u\right|\right|_{L^2_x}\left|\left|\partial_xu\right|\right|_{L^4_x}^2.
 \end{equation*}
 We thus conclude $\left|\left|\partial_xu\right|\right|_{L^4_x}^2\leq C\left|\left|u\right|\right|_{L^{\infty}_x}\left|\left|\partial_x^2u\right|\right|_{L^2_x}.$
 
For the term $||\partial_x^2u||_{L_x^4}^2$, we again use an integration by parts and the Cauchy-Schwarz inequality to write
 \begin{equation*}
 ||\partial_x^2u||_{L_x^4}^4=\int_{\mathbb{R}}\left(\partial_x^2u(x)\right)^3\partial_x^2u(x)dx=\int_{\mathbb{R}}-3\partial_x^3u(x)\left(\partial_x^2u(x)\right)^2\partial_xu(x)dx\leq 3||\partial_x^2u||_{L_x^4}^2\sqrt{\int_{\mathbb{R}}\left(\partial_x^3u(x)\right)^2\left(\partial_xu(dx)\right)^2dx},
 \end{equation*}
 which implies thanks to Relation \eqref{relation_sept_2016} $||\partial_x^2u||_{L_x^4}^2\leq C\left|\left|u\right|\right|_{L^{\infty}_x}\left|\left|\partial_x^4u\right|\right|_{L^2_x}$.
 For the $||\partial_{xt}u(t,\cdot)||_{L^2_x}$-term in \eqref{EQ_4|}, it holds
 \begin{equation*}
 \begin{split}
 ||\partial_{tx}u(t,\cdot)||_{L^2_x}^2&=||-\left(\partial_xu(t,\cdot)\right)^2-u(t,\cdot)\partial_x^2u(t,\cdot)-\partial_x^4u(t,\cdot)||_{L^2_x}^2\\
 &\leq C\left[ ||u(t,\cdot)||_{L^{\infty}_x}^2||\partial_x^2u(t,\cdot)||_{L^2_x}^2+||\partial_xu(t,\cdot)||_{L^4_x}^4+||\partial_x^4u(t,\cdot)||_{L^2_x}^2\right].
 \end{split}
 \end{equation*}
 
To conclude, we obtain with \eqref{EQ_4|}
\begin{equation*}
\begin{split}
&||\epsilon^n||_{\ell^{\infty}(\llbracket 0,N\rrbracket;\ell^2_{\Delta}(\mathbb{Z}))}\leq C\left[ \Delta t\underset{t\in[0,T]}{\mathrm{sup}}\left(||u||_{L^{\infty}_x}^2||u||_{H^{2}_x}+||u||_{L^{\infty}_x}||u||_{H^{4}_x}+||u||_{H^{6}_x}+||u||_{L^{\infty}_x}||u||_{H^{2}_x}+||u||_{H^{4}_x}\right)\right.\\
&\left.+\Delta x\underset{t\in[0,T]}{\mathrm{sup}}\left(||u||_{L^{\infty}_x}||u||_{H^{2}_x}+||\partial_xu||_{L^{\infty}_x}||u||_{H^{1}_x}+||u||_{H^{4}_x}+||u||_{H^{2}_x}\right)\right],
\end{split}
\end{equation*}
which  can be simplified into
\begin{multline*}
||\epsilon^n||_{\ell^{\infty}(\llbracket0,N\rrbracket;\ell^2_{\Delta}(\mathbb{Z}))}\leq C\left[ \Delta t\underset{t\in[0,T]}{\mathrm{sup}}\left(||u||_{L^{\infty}_x}^2||u||_{H^{2}_x}+||u||_{L^{\infty}_x}||u||_{H^{4}_x}+||u||_{H^{6}_x}\right)\right.\\\left.+\Delta x\underset{t\in[0,T]}{\mathrm{sup}}\left(||u||_{L^{\infty}_x}||u||_{H^{2}_x}+||\partial_xu||_{L^{\infty}_x}||u||_{H^{1}_x}+||u||_{H^{4}_x}\right)\right].
\end{multline*}
Thus the consistency error is upper bounded by
\begin{equation*}
||\epsilon^n||_{\ell^{\infty}(\llbracket0,N\rrbracket;\ell^2_{\Delta}(\mathbb{Z}))}\leq C\left\{\Delta t\underset{t\in[0,T]}{\text{sup}}\left[\left(1+||u||^2_{L^{\infty}_x}\right)||u||_{H^6_x}\right]+\Delta x\underset{t\in[0,T]}{\text{sup}}\left[\left(1+||u||_{L^{\infty}_x}\right)||u||_{H^4_x}+||\partial_xu||_{L^{\infty}_x}||u||_{H^1_x}\right]\right\}
\end{equation*}
as claimed in \eqref{appendixconsistance}.
This ends the proof of Proposition \ref{propositionun}.

%
%
\section{Appendix : Proof of  Proposition \ref{PROP_DISCRETE_RELATIVE_ENTROPY_INEQ}}
\label{PROOF_Stability_Ineq_10_avril}
This appendix is devoted to the proof of Proposition \ref{PROP_DISCRETE_RELATIVE_ENTROPY_INEQ} to obtain stability Inequality \eqref{latest_ineq}.\begin{proof}[Proof of Proposition \ref{PROP_DISCRETE_RELATIVE_ENTROPY_INEQ}]

Thanks to \eqref{SCHEMA_REMANIE}, one has
\begin{equation}
\label{jemenfousjesaisplus1}
\left|\left|\mathcal{A}_{\theta}e^{n+1}\right|\right|_{\ell^2_{\Delta}}^2=\left(\text{RHS}^n\right)_a+\left(\text{RHS}^n\right)_b+\left(\text{RHS}^n\right)_c
\end{equation}
with
\begin{small}
\begin{equation*}
\left(\text{RHS}^n\right)_a=\left|\left|e^n\right|\right|_{\ell^2_{\Delta}}^2+(1-\theta)^2\Delta t^2\left|\left|D_+D_+D_-\left(e\right)^{n}\right|\right|_{\ell^2_{\Delta}}^2+\Delta t^2\left|\left|D\left(\frac{e^2}{2}\right)^n\right|\right|_{\ell^2_{\Delta}}^2+\Delta t^2\left|\left|D\left(u_{\Delta}e\right)^n\right|\right|_{\ell^2_{\Delta}}^2+\frac{c^2\Delta t^2\Delta x^2}{4}\left|\left|D_+D_-\left(e\right)^n\right|\right|_{\ell^2_{\Delta}}^2,
\end{equation*}
\end{small}
\begin{small}
\begin{equation}
\begin{split}
\left(\text{RHS}^n\right)_b=&-2(1-\theta)\Delta t\left\langle e^n, D_+D_+D_-\left(e\right)^n\right\rangle-2\Delta t\left\langle e^n, D\left(\frac{e^2}{2}\right)^n\right\rangle-2\Delta t\left\langle e^n, D\left(u_{\Delta}e\right)^n\right\rangle +c\Delta x\Delta t\left\langle e^n, D_+D_-\left(e\right)^n\right\rangle \\
&+2(1-\theta)\Delta t^2\left\langle D_+D_+D_-\left(e\right)^n, D\left(u_{\Delta}e\right)^n\right\rangle+2(1-\theta)\Delta t^2\left\langle D_+D_+D_-\left(e\right)^n, D\left(\frac{e^2}{2}\right)^n\right\rangle\\
&-c\Delta x\Delta t^2(1-\theta)\left\langle D_+D_+D_-\left(e\right)^n, D_+D_-\left(e\right)^n\right\rangle+2\Delta t^2\left\langle D\left(\frac{e^2}{2}\right)^n, D\left(u_{\Delta}e\right)^n\right\rangle-c\Delta x\Delta t^2\left\langle D\left(\frac{e^2}{2}\right)^n, D_+D_-\left(e\right)^n\right\rangle\\
&-c\Delta x\Delta t^2\left\langle D\left(u_{\Delta}e\right)^n, D_+D_-\left(e\right)^n\right\rangle,
\end{split}
\label{DEF_b}
\end{equation}
\end{small}
and
\begin{small}
\begin{multline*}
\left(\text{RHS}^n\right)_c=-2\Delta t\left\langle e^n-(1-\theta)\Delta tD_+D_+D_-\left(e\right)^n, \epsilon^n\right\rangle+2\Delta t^2\left\langle D\left(\frac{e^2}{2}\right)^n, \epsilon^n\right\rangle+2\Delta t^2\left\langle D\left(u_{\Delta}e\right)^n, \epsilon^n\right\rangle -c\Delta x\Delta t^2\left\langle D_+D_-\left(e\right)^n, \epsilon^n\right\rangle\\+\Delta t^2\left|\left|\epsilon^n\right|\right|_{\ell^2_{\Delta}}^2.
\end{multline*}
\end{small}


\paragraph{Right-hand side $\left(\text{RHS}^n\right)_a$ }

We here will bound $\left(\text{RHS}^n\right)_a$. 
\begin{itemize}
\item To this aim, we use the discrete integrations by parts formulas of Subsection \ref{SECTION_2}, to see that, thanks to Identity \eqref{EQ_3_LEMMA_3},
$$
\Delta t^2\left|\left|D\left(\frac{e^2}{2}\right)^n\right|\right|_{\ell^2_{\Delta}}^2=\Delta t^2\left|\left|D\left(e\right)^n\left(\frac{\mathcal{S}^{+}e^n+\mathcal{S}^{-}e^n}{2}\right)\right|\right|_{\ell^2_{\Delta}}^2.
$$
\item To bound $\Delta t^2\left|\left| D\left(u_{\Delta}e\right)^n\right|\right|_{\ell^2_{\Delta}}^2$, we shall  use
 the following lemma.
 \begin{Lemma}
Let $\left(a_j\right)_{j\in\mathbb{Z}}$ and $\left(b_j\right)_{j\in\mathbb{Z}}$ be two sequences in $\ell_{\Delta}^2(\mathbb{Z})$. For any 
$\Delta t > 0$ one has 
\begin{small}
\begin{equation}
\left|\left|D\left(ab\right)\right|\right|_{\ell^2_{\Delta}}^2\leq \left\langle b^2+\frac{\Delta t}{2}\left[\left(D_+b\right)^2+\left(D_-b\right)^2\right], \left(Da\right)^2\right\rangle+\frac{1}{2}\left\langle \frac{\left(\mathcal{S}^{-}b\right)^2+\left(\mathcal{S}^{+}b\right)^2}{\Delta t}+\frac{3}{4}\left(D_+b\right)^2+\frac{3}{4}\left(D_-b\right)^2, a^2\right\rangle.
\label{EQ_2_LEMMA_2}
\end{equation}
\end{small}
\label{PROP_3333_3333_3333_2}
\end{Lemma}
The proof of this lemma is postponed to the end of the section.

Relation \eqref{EQ_2_LEMMA_2} gives 
\begin{multline*}
\Delta t^2\left|\left|D\left(u_{\Delta}e\right)^n\right|\right|_{\ell^2_{\Delta}}^2\leq \Delta t^2\left\langle\left(\left[u_{\Delta}\right]^n\right)^2+\frac{\Delta t}{2}\left(D_+\left(u_{\Delta}\right)^n\right)^2+\frac{\Delta t}{2}\left(D_-\left(u_{\Delta}\right)^n\right)^2, \left(De^n\right)^2\right\rangle\\+\frac{\Delta t}{2}\left\langle \left(\mathcal{S}^{-}\left[u_{\Delta}\right]^n\right)^2+\left(\mathcal{S}^{+}\left[u_{\Delta}\right]^n\right)^2+\frac{3\Delta t}{4}\left(D_+\left(u_{\Delta}\right)^n\right)^2+\frac{3\Delta t}{4}\left(D_-\left(u_{\Delta}\right)^n\right)^2, \left(e^n\right)^2\right\rangle.
\end{multline*}
We turn our attention to the term $\frac{\Delta t^3}{2}\left\langle \left(D_+(u_{\Delta})^n\right)^2+\left(D_-(u_{\Delta})^n\right)^2, \left(De^n\right)^2\right\rangle$ in the  first line of the above expression.  By using  the definition of $De_j^n$, we obtain that
\begin{equation*}
\begin{split}
\frac{\Delta t^3}{2}\left\langle\left(D_+(u_{\Delta})^n\right)^2+\left(D_-(u_{\Delta})^n\right)^2, \left(De^n\right)^2\right\rangle&\leq \frac{\Delta t^3}{\Delta x^2}||D_+(u_{\Delta})^n||_{\ell^{\infty}}^2||e^n||_{\ell^2_{\Delta}}^2.
\end{split}
\end{equation*}

\item Thanks to Relation \eqref{EQ_2_LEMMA_3}, one has $$\frac{c^2\Delta t^2\Delta x^2}{4}\left|\left|D_+D_-(e)^n\right|\right|_{\ell^2_{\Delta}}^2=c^2\Delta t^2\left|\left|D_+\left(e\right)^n\right|\right|_{\ell^2_{\Delta}}^2-c^2\Delta t^2\left|\left|D\left(e\right)^n\right|\right|_{\ell^2_{\Delta}}^2.$$
\end{itemize}
All this yields
\begin{small}
\begin{equation*}
\begin{split}
&\left(\text{RHS}^n\right)_a\leq\Delta t^2\left|\left|D_+D_+D_-\left(e\right)^{n}\right|\right|_{\ell^2_{\Delta}}^2\left(\theta^2+(1-2\theta)\right)+c^2\Delta t^2\left|\left|D_+\left(e\right)^n\right|\right|_{\ell^2_{\Delta}}^2\\
&+\Delta t^2\left\langle \left[D\left(e\right)^n\right]^2, \left(\frac{\mathcal{S}^{+}e^n+\mathcal{S}^{-}e^n}{2}\right)^2+\left[\left(u_{\Delta}\right)^n\right]^2-c^2\boldsymbol{1}\right\rangle\\
&+\left\langle \left(e^n\right)^2, \boldsymbol{1}+\frac{\Delta t}{2}\left[\left(\mathcal{S}^{-}\left[u_{\Delta}\right]^n\right)^2+\left(\mathcal{S}^{+}\left[u_{\Delta}\right]^n\right)^2+\frac{3\Delta t}{4}\left(D_+\left(u_{\Delta}\right)^n\right)^2+\frac{3\Delta t}{4}\left(D_-\left(u_{\Delta}\right)^n\right)^2+2\frac{\Delta t^2}{\Delta x^2}||D_+(u_{\Delta})^n||_{\ell^{\infty}}^2\boldsymbol{1}\right]\right\rangle.
\end{split}
\end{equation*}
\end{small}

\paragraph{Right-hand side $\left(\text{RHS}^n\right)_b$ }
We next focus on $\left(\text{RHS}^n\right)_b$ and on its different ten terms. \begin{itemize}
\item By Relations \eqref{EQ_3_LEMMA_BASIC} and \eqref{EQ_5_LEMMA_BASIC}, one sees that
\begin{equation*}
\begin{split}
-2(1-\theta)\Delta t\left\langle e^n, D_+D_+D_-(e)^n\right\rangle&=2\theta\Delta t\left\langle e^n, D_+D_+D_-\left(e\right)^n\right\rangle+ 2\Delta t\left\langle D_-\left(e\right)^n, D_+D_-\left(e\right)^n\right\rangle,\\
&=2\theta\Delta t\left\langle e^n, D_+D_+D_-\left(e\right)^n\right\rangle-\Delta t\Delta x\left|\left|D_+D_-\left(e\right)^n\right|\right|_{\ell^2_{\Delta}}^2.
\end{split}
\end{equation*}
Equality \eqref{EQ_2_LEMMA_4} enables to write \begin{equation*}
-2(1-\theta)\Delta t\left\langle e^n, D_+D_+D_-(e)^n\right\rangle=2\theta\Delta t\left\langle e^n, D_+D_+D_-\left(e\right)^n\right\rangle-\frac{\Delta t\Delta x^3}{4}\left|\left|D_+D_+D_-\left(e\right)^n\right|\right|_{\ell^2_{\Delta}}^2-\Delta t\Delta x\left|\left|D_+D\left(e\right)^n\right|\right|_{\ell^2_{\Delta}}^2.
\end{equation*}
\item Thanks to Identity \eqref{EQ_0_LEMMA_3}, one has $$-2\Delta t\left\langle e^n, D\left(\frac{e^2}{2}\right)^n\right\rangle =\frac{\Delta x^2\Delta t}{6}\left\langle D_+\left(e\right)^n, \left(D_+\left(e\right)^n\right)^2\right\rangle.$$
\item Identity \eqref{EQ_1_LEMMA_1} gives $$-2\Delta t\left\langle e^n, D\left(u_{\Delta}e\right)^n\right\rangle =-\Delta t\left\langle D_+\left(u_{\Delta}\right)^n, e^n\mathcal{S}^{+}e^n\right\rangle\leq \Delta t||D_+\left(u_{\Delta}\right)^n||_{\ell^{\infty}}\left|\left|e^n\right|\right|_{\ell^2_{\Delta}}^2.$$
\item Moreover, Relations \eqref{EQ_1_LEMMA_BASIC} and \eqref{EQ_3_LEMMA_BASIC} imply $$c\Delta x\Delta t\left\langle e^n, D_+D_-(e)^n\right\rangle=-c\Delta x\Delta t\left|\left|D_+\left(e\right)^n\right|\right|_{\ell^2_{\Delta}}^2.$$

\item To bound $2(1-\theta)\Delta t^2\left\langle D_+D_+D_-(e)^n, D\left(u_{\Delta}e\right)^n\right\rangle$, we use the
following lemma.
\begin{Lemma}
Let $\left(a_j\right)_{j\in\mathbb{Z}}$, $\left(b_j\right)_{j\in\mathbb{Z}}$ be two sequences in $\ell_{\Delta}^2(\mathbb{Z})$ and $\sigma\in\{0,1\}$. One has
\begin{multline}
\left\langle D_+D_+D_-\left(a\right), D\left(ab\right)\right\rangle \leq \frac{\Delta t}{4}\left\langle |D_+\left(b\right)|+|D_-\left(b\right)|, \left(D_+D_+D_-\left(a\right)\right)^2\right\rangle+\frac{1}{4\Delta t}\left\langle |D_-\left(b\right)|+|D_+\left(b\right)|, a^2\right\rangle\\+\frac{1}{2}\left\langle ||D_+(b)||_{\ell^{\infty}}^{\sigma}\boldsymbol{1}-\frac{\Delta x}{2}D_-\left(b\right), \left(D_+D_-\left(a\right)\right)^{2}\right\rangle+\frac{1}{2}||D_+\left(b\right)||_{\ell^{\infty}}^{2-\sigma}\left|\left|D_+\left(a\right)\right|\right|_{\ell^2_{\Delta}}^2-\left\langle b, \left(D_+D\left(a\right)\right)^2\right\rangle.
\end{multline}
\label{PROP_32_je-ne-sais-plus}
\end{Lemma}
Again, we postpone the proof of this lemma until the end of the section.

Thanks to this lemma applied with $a_j=e_j^n$ and $b_j=\left(u_{\Delta }\right)_j^n$, one has
\begin{small}
\begin{equation*}
\begin{split}
2(1-\theta)\Delta t^2&\left\langle D_+D_+D_-(e)^n, D\left(u_{\Delta}e\right)^n\right\rangle \leq\frac{\Delta t^3}{2}(1-\theta)\left\langle |D_+\left(u_{\Delta}\right)^n|+|D_-\left(u_{\Delta}\right)^n|, \left(D_+D_+D_-\left(e\right)^n\right)^2\right\rangle\\
&+\frac{\Delta t}{2}(1-\theta)\left\langle|D_-\left(u_{\Delta}\right)^n|+|D_+\left(u_{\Delta}\right)^n|, \left(e^n\right)^2\right\rangle\\
&+(1-\theta)\Delta t^2\left\langle ||D_+(u_{\Delta})^n||_{\ell^{\infty}}^{\sigma}\boldsymbol{1}-\frac{\Delta x}{2}D_-\left(u_{\Delta}\right)^n, \left(D_+D_-\left(e\right)^n\right)^2\right\rangle\\
&+(1-\theta)\Delta t^{2}||D_+\left(u_{\Delta }\right)^n||_{\ell^{\infty}}^{2-\sigma}\left|\left|D_+\left(e\right)^n\right|\right|_{\ell^2_{\Delta}}^2-2(1-\theta)\Delta t^2\left\langle \left(u_{\Delta}\right)^n, \left(D_+D\left(e\right)^n\right)^2\right\rangle,
\end{split}
\end{equation*}
\end{small}
for $\sigma\in\{0, 1\}$.

\item To bound $2(1-\theta)\Delta t^2\left\langle D_+D_+D_-(e)^n, D\left(\frac{e^2}{2}\right)^n\right\rangle$, we use  the following lemma. 
\begin{Lemma}
Let $\left(a_j\right)_{j\in\mathbb{Z}}$ be a sequence in $\ell_{\Delta}^2(\mathbb{Z})$ and $\gamma \in[0,\frac{1}{2})$, one has
\begin{equation*}
\left\langle D_+D_+D_-(a), D\left(\frac{a^2}{2}\right)\right\rangle\leq \frac{\Delta x^{\frac{1}{2}-\gamma}+||a||_{\ell^{\infty}}+9||a||^2_{\ell^{\infty}}\Delta x^{\gamma-\frac{1}{2}}}{2}\left|\left|D_+D_-(a)\right|\right|_{\ell^2_{\Delta}}^2
+||a||_{\ell^{\infty}}\left|\left|D_+D(a)\right|\right|_{\ell^2_{\Delta}}^2.
\end{equation*}
\label{Prop_B,5}
\end{Lemma}
The proof is postponed to the end of the section.

Applying Lemma \ref{Prop_B,5} to $a_j=e_j^n$, one gets
\begin{multline*}
2(1-\theta)\Delta t^2\left\langle D_+D_+D_-(e)^n, D\left(\frac{e^2}{2}\right)^n\right\rangle \\\leq\Delta t^2(1-\theta)\left(\Delta x^{\frac{1}{2}-\gamma}+||e^n||_{\ell^{\infty}}+9||e^n||_{\ell^{\infty}}^2\Delta x^{\gamma-\frac{1}{2}}\right)\left|\left|D_+D_-\left(e\right)^n\right|\right|_{\ell^2_{\Delta}}^2\\+2(1-\theta)\Delta t^2||e^n||_{\ell^{\infty}}\left|\left|D_+D(e)^n\right|\right|_{\ell^2_{\Delta}}^2.
\end{multline*}
Once again, Relation \eqref{EQ_2_LEMMA_4} transforms $\Delta t^2(1-\theta)\left(\Delta x^{\frac{1}{2}-\gamma}+||e^n||_{\ell^{\infty}}+9||e^n||_{\ell^{\infty}}^2\Delta x^{\gamma-\frac{1}{2}}\right)\left|\left|D_+D_-\left(e\right)^n\right|\right|_{\ell^2_{\Delta}}^2$ to obtain
\begin{multline*}
2(1-\theta)\Delta t^2\left\langle D_+D_+D_-(e)^n, D\left(\frac{e^2}{2}\right)^n\right\rangle \\\leq\Delta t^2(1-\theta)\left[\Delta x^{\frac{1}{2}-\gamma}+||e^n||_{\ell^{\infty}}+9||e^n||_{\ell^{\infty}}^2\Delta x^{\gamma-\frac{1}{2}}\right]\left|\left|D_+D\left(e\right)^n\right|\right|_{\ell^2_{\Delta}}^2\\\hspace*{2cm}+(1-\theta)\frac{\Delta t^2\Delta x^2}{4}\left[\Delta x^{\frac{1}{2}-\gamma}+||e^n||_{\ell^{\infty}}+9||e^n||_{\ell^{\infty}}^2\Delta x^{\gamma-\frac{1}{2}}\right]\left|\left|D_+D_+D_-(e)^n\right|\right|_{\ell^2_{\Delta}}^2\\+2(1-\theta)\Delta t^2||e^n||_{\ell^{\infty}}\left|\left|D_+D(e)^n\right|\right|_{\ell^2_{\Delta}}^2.
\end{multline*}
\begin{Remark}
Thereafter, $a_j$ will be replaced by the unknown $e_j^n$ whereas $b_j$ will be replaced by the exact solution $[u_{\Delta}]_j^n$. We could not use Lemma \ref{PROP_32_je-ne-sais-plus} with $b_j=\frac{a_j}{2}$ instead of Lemma \ref{Prop_B,5} because $D_+(b)_j$ in Lemma \ref{PROP_32_je-ne-sais-plus} will be replaced by $D_+(\frac{a}{2})_j=D_+(\frac{e}{2})^n_j$ which is always unknown.
\end{Remark}
\item Relation \eqref{EQ_5_LEMMA_BASIC} gives $$-c\Delta x\Delta t^2(1-\theta)\left\langle D_+D_+D_-(e)^n, D_+D_-(e)^n\right\rangle=(1-\theta)c\frac{\Delta x^2\Delta t^2}{2}\left|\left|D_+D_+D_-\left(e\right)^n\right|\right|_{\ell^2_{\Delta}}^2.$$

\item To deal with $2\Delta t^2\left\langle D\left(\frac{e^2}{2}\right)^n, D\left(u_{\Delta}e\right)^n\right\rangle $, we use the next lemma  whose proof is left to the reader.  
\begin{Lemma}
Let $\left(a_j\right)_{j\in\mathbb{Z}}$ and $\left(b_j\right)_{j\in\mathbb{Z}}$ be two sequences in $\ell_{\Delta}^2(\mathbb{Z})$, then one has
\begin{equation}
\left\langle D\left(ab\right), D\left(\frac{a^2}{2}\right)\right\rangle =\left\langle \left[D\left(a\right)\right]^2, \frac{\mathcal{S}^{+}a\mathcal{S}^{+}b+\mathcal{S}^{-}a\mathcal{S}^{-}b}{2}\right\rangle-\frac{4\Delta x^2}{3}\left\langle D\left(b\right), \left[D\left(a\right)\right]^3\right\rangle-\frac{1}{3}\left\langle DD\left(b\right), a^3\right\rangle.
\label{EQ_1_LEMMA_2}
\end{equation}
\label{PROP_3333_3333_3333_1}
\end{Lemma}

Identity \eqref{EQ_1_LEMMA_2} with $a_j=e_j^n$ and $b_j=\left(u_{\Delta }\right)_j^n$ gives
\begin{multline*}
2\Delta t^2\left\langle D\left(\frac{e^2}{2}\right)^n, D\left(u_{\Delta}e\right)^n\right\rangle=\Delta t^2\left\langle \left[D\left(e\right)^n\right]^2, \mathcal{S}^{+}\left(u_{\Delta}\right)^n\mathcal{S}^{+}e^{n}+\mathcal{S}^{-}\left(u_{\Delta}\right)^n\mathcal{S}^{-}e^n\right\rangle\\-\frac{8\Delta x^2\Delta t^2}{3}\left\langle D\left(u_{\Delta}\right)^n, \left[D\left(e\right)^n\right]^3\right\rangle -\frac{2\Delta t^2}{3}\left\langle DD\left(u_{\Delta}\right)^n, \left(e^n\right)^3\right\rangle.
\end{multline*}
\item Relation \eqref{EQ_1_LEMMA_3} yields $$-c\Delta x\Delta t^2\left\langle D\left(\frac{e^2}{2}\right)^n, D_+D_-(e)^n\right\rangle=-\frac{c\Delta x\Delta t^2}{6}\left\langle D_+\left(e\right)^n, \left(D_+\left(e\right)^n\right)^2\right\rangle+\frac{2c\Delta x\Delta t^2}{3}\left\langle D\left(e\right)^n, \left(D\left(e\right)^n\right)^2\right\rangle.$$
\item Relation \eqref{EQ_2_LEMMA_1} implies $$-c\Delta x\Delta t^2\left\langle D\left(u_{\Delta}e\right)^n, D_+D_-(e)^n\right\rangle=\frac{c\Delta t^2}{\Delta x}\left\langle D_+\left(u_{\Delta}\right)^n, e^n\mathcal{S}^{+}e^n\right\rangle-\frac{c\Delta t^2}{\Delta x}\left\langle D\left(u_{\Delta}\right)^n, \mathcal{S}^{-}e^n\mathcal{S}^{+}e^n\right\rangle.$$
Thus, thanks to the Cauchy-Schwarz inequality, we get $$-c\Delta x\Delta t^2\left\langle D\left(u_{\Delta}e\right)^n, D_+D_-(e)^n\right\rangle\leq \frac{c\Delta t^2}{\Delta x}||D_+\left(u_{\Delta}\right)^n||_{\ell^{\infty}}\left|\left|e^n\right|\right|_{\ell^2_{\Delta}}^2+\frac{c\Delta t^2}{\Delta x}||D\left(u_{\Delta}\right)^n||_{\ell^{\infty}}\left|\left|e^n\right|\right|_{\ell^2_{\Delta}}^2.$$
\end{itemize}

Gathering all these relations yields the following inequality, for $\sigma\in\{0, 1\}$.
\begin{small}
\begin{equation*}
\begin{split}
&\left(\text{RHS}^n\right)_b\leq 2\theta\Delta t\left\langle e^n, D_+D_+D_-\left(e\right)^n\right\rangle+(1-\theta)\Delta t^2\left[||D_+(u_{\Delta})^n||_{\ell^{\infty}}^{\sigma}+\frac{\Delta x}{2}||D_-\left(u_{\Delta}\right)^n||_{\ell^{\infty}}\right]\left|\left|D_+D_-e^n\right|\right|_{\ell^2_{\Delta}}^2\\
&+\Delta t\left\langle ||D_+u_{\Delta}^n||_{\ell^{\infty}}\boldsymbol{1}-\frac{2\Delta t}{3}DD\left(u_{\Delta}\right)^ne^n+\frac{c\Delta t}{\Delta x}||D_+u_{\Delta}^n||_{\ell^{\infty}}\boldsymbol{1}+\frac{c\Delta t}{\Delta x}||Du_{\Delta}^n||_{\ell^{\infty}}\boldsymbol{1}+\frac{(1-\theta)}{2}\left[|D_+u_{\Delta}^n|+|D_-u_{\Delta}^n|\right], \left(e^n\right)^2\right\rangle\\
&+\Delta t\left\langle \frac{\Delta x^2}{6}D_+\left(e\right)^n-c\Delta x\boldsymbol{1}-\frac{c\Delta t\Delta x}{6}D_+\left(e\right)^n+(1-\theta)\Delta t||D_+\left(u_{\Delta}\right)^n||_{\ell^{\infty}}^{2-\sigma}\boldsymbol{1}, \left[D_+\left(e\right)^n\right]^2\right\rangle\\
&+\Delta t^2\left\langle \left(D\left(e\right)^n\right)^2, \mathcal{S}^{+}\left(u_{\Delta}\right)^n\mathcal{S}^{+}e^{n}+\mathcal{S}^{-}\left(u_{\Delta}\right)^n\mathcal{S}^{-}e^n-\frac{8\Delta x^2}{3}D\left(u_{\Delta}\right)^nD\left(e\right)^n+\frac{2c\Delta x}{3}D\left(e\right)^n\right\rangle\\
&+\Delta t\left\langle-\Delta x\boldsymbol{1}-2(1-\theta)\Delta t\left(u_{\Delta}\right)^n+2(1-\theta)\Delta t||e^n||_{\ell^{\infty}}\boldsymbol{1}+\Delta t(1-\theta)\left[\Delta x^{\frac{1}{2}-\gamma}+||e^n||_{\ell^{\infty}}+9||e^n||_{\ell^{\infty}}^2\Delta x^{\gamma-\frac{1}{2}}\right]\boldsymbol{1}, \left(D_+De^n\right)^2\right\rangle\\
&+\Delta t\left\langle-\frac{\Delta x^3}{4}\boldsymbol{1}+c\frac{(1-\theta)\Delta x^2\Delta t}{2}\boldsymbol{1}+\frac{\Delta t^2(1-\theta)}{2}\left[|D_+\left(u_{\Delta}\right)^n|+|D_-\left(u_{\Delta}\right)^n|\right]\right.\\
&\left.+(1-\theta)\frac{\Delta t\Delta x^2}{4}\left(\Delta x^{\frac{1}{2}-\gamma}+||e^n||_{\ell^{\infty}}+9||e^n||_{\ell^{\infty}}^2\Delta x^{\gamma-\frac{1}{2}}\right)\boldsymbol{1}, \left[D_+D_+D_-\left(e\right)^n\right]^2\right\rangle.
\end{split}
\end{equation*}
\end{small}

\paragraph{Right-hand side $\left(\text{RHS}^n\right)_c$ }
Let us now focus on 
$\left(\text{RHS}^n\right)_c$ and its four different terms.

\begin{itemize}
\item From Young's inequality, 
\begin{equation*}-2\Delta t\left\langle e^n-(1-\theta)\Delta tD_+D_+D_-\left(e\right)^n, \epsilon^n\right\rangle \leq\Delta t\left|\left|\mathcal{A}_{-(1-\theta)}e^n\right|\right|_{\ell^2_{\Delta}}^2+\Delta t\left|\left|\epsilon^n\right|\right|_{\ell^2_{\Delta}}^2.
\end{equation*}

\item Once again, we apply Young's inequality to obtain \begin{equation*}\begin{split}2\Delta t^2\left\langle D\left(\frac{e^2}{2}\right)^n, \epsilon^n\right\rangle \leq \frac{\Delta t^2}{\Delta x}\left|\left|\epsilon^n\right|\right|_{\ell^2_{\Delta}}^2+\Delta t^2\Delta x\left|\left|D\left(\frac{e^2}{2}\right)^n\right|\right|_{\ell^2_{\Delta}}^2.
\end{split}
\end{equation*}

Then, Identity \eqref{EQ_3_LEMMA_3} gives
\begin{equation*}
\begin{split}2\Delta t^2\left\langle D\left(\frac{e^2}{2}\right)^n, \epsilon^n\right\rangle \leq \frac{\Delta t^2}{\Delta x}\left|\left|\epsilon^n\right|\right|_{\ell^2_{\Delta}}^2+\Delta t^2\Delta x\left|\left| D\left(e\right)^n\left(\frac{\mathcal{S}^{+}e^n+\mathcal{S}^{-}e^n}{2}\right)\right|\right|_{\ell^2_{\Delta}}^2,
\end{split}
\end{equation*}
\item One also has
\begin{equation*}
\begin{split}
2\Delta t^2\left\langle D\left(u_{\Delta}e\right)^n, \epsilon^n\right\rangle \leq \frac{\Delta t^2}{\Delta x}||(u_{\Delta})^n||_{\ell^{\infty}}^2\left|\left|e^n\right|\right|_{\ell^2_{\Delta}}^2+\frac{\Delta t^2}{\Delta x}\left|\left|\epsilon^n\right|\right|_{\ell^2_{\Delta}}^2.
\end{split}
\end{equation*}
\item Finally, we see that, thanks to Young's inequality, $$-c\Delta x\Delta t^2\left\langle D_+D_-\left(e\right)^n, \epsilon^n\right\rangle\leq 2c^2\frac{\Delta t^2}{\Delta x}\left|\left|e^n\right|\right|_{\ell^2_{\Delta}}^2+2\frac{\Delta t^2}{\Delta x}\left|\left|\epsilon^n\right|\right|_{\ell^2_{\Delta}}^2.$$

\end{itemize}

Thus, we have
\begin{equation*}
\begin{split}
\left(\text{RHS}^n\right)_c&\leq \Delta t||e^n||^2_{\ell^2_{\Delta}}\left\{\frac{\Delta t}{\Delta x}\left[||(u_{\Delta})^n||_{\ell^{\infty}}^2+2c^2\right]\right\}+\Delta t||\epsilon^n||^2_{\ell^2_{\Delta}}\left\{1+4\frac{\Delta t}{\Delta x}+\Delta t\right\}\\
&+\Delta t\left|\left|\mathcal{A}_{-(1-\theta)}e^n\right|\right|_{\ell^2_{\Delta}}^2+\Delta t^2\Delta x\left|\left|D\left(e\right)^n\left(\frac{\mathcal{S}^{+}e^n+\mathcal{S}^{-}e^n}{2}\right)\right|\right|_{\ell^2_{\Delta}}^2.
\end{split}
\end{equation*}

\paragraph{Final inequality}
Gathering the previous estimates on the right hand-side of \eqref{jemenfousjesaisplus1}, the convergence error satisfies the following inequality
\begin{equation*}
\begin{split}
&\left|\left|\mathcal{A}_{\theta}e^{n+1}\right|\right|_{\ell^2_{\Delta}}^2\leq \left|\left|\mathcal{A}_{\theta}e^n\right|\right|_{\ell^2_{\Delta}}^2+\Delta t\left\langle \left(e^n\right)^2, F_{a}\right\rangle+\Delta t\left|\left| \mathcal{A}_{-(1-\theta)}e^n\right|\right|_{\ell^2_{\Delta}}^2+\Delta t||\epsilon^n||^2_{\ell^2_{\Delta}}\left\{1+4\frac{\Delta t}{\Delta x}+\Delta t\right\}\\
&+\Delta t\left\langle F_{b}, \left[D_+\left(e\right)^n\right]^2\right\rangle+\Delta t^2\left\langle F_{c}, \left[D\left(e\right)^n\right]^2\right\rangle+\Delta tF_{d}\left|\left|D_+D_-\left(e\right)^n\right|\right|_{\ell^2_{\Delta}}^2+\Delta tF_{e}\left|\left|D_+D\left(e\right)^n\right|\right|_{\ell^2_{\Delta}}^2\\
&+\Delta tF_{f}\left|\left|D_+D_+D_-\left(e\right)^n\right|\right|_{\ell^2_{\Delta}}^2,
\end{split}
\end{equation*}
with
\begin{multline*}
\hspace*{-0.45cm}F_{a}=\frac{\left(\mathcal{S}^{-}\left[u_{\Delta}\right]^{n}\right)^2}{2}+\frac{\left(\mathcal{S}^{+}\left[u_{\Delta}\right]^{n}\right)^2}{2}+\frac{\Delta t}{2}\left[\frac{3}{4}\left(D_-\left(u_{\Delta}\right)^n\right)^2+\frac{3}{4}\left(D_+\left(u_{\Delta}\right)^n\right)^2\right]+\frac{\Delta t^2}{\Delta x^2}||D_+(u_{\Delta})^n||_{\ell^{\infty}}^2\boldsymbol{1}\\
+\frac{(1-\theta)}{2}\left[|D_-\left(u_{\Delta}\right)^n|+|D_+\left(u_{\Delta}\right)^n|\right]+||D_+\left(u_{\Delta}\right)^n||_{\ell^{\infty}}\left(1+\frac{c\Delta t}{\Delta x}\right)\boldsymbol{1}\\
+\frac{c\Delta t}{\Delta x}||D\left(u_{\Delta}\right)^n||_{\ell^{\infty}}\boldsymbol{1}-\frac{2\Delta t}{3}DD\left(u_{\Delta}\right)^ne^n+\frac{\Delta t}{\Delta x}\left(||(u_{\Delta})^n||_{\ell^{\infty}}^2+2c^2\right)\boldsymbol{1},
\end{multline*}
\begin{equation*}
\hspace*{-4.1cm}F_{b}=c^2\Delta t\boldsymbol{1}+\frac{\Delta x^2}{6}D_+\left(e\right)^n-c\Delta x\boldsymbol{1}-\frac{c\Delta x\Delta t}{6}D_+\left(e\right)^n+(1-\theta)\Delta t||D_+\left(u_{\Delta}\right)^n||_{\ell^{\infty}}^{2-\sigma}\boldsymbol{1},
\end{equation*}
\begin{equation*}
F_{c}=\left(\frac{\mathcal{S}^{+}e^n+\mathcal{S}^{-}e^n}{2}\right)^2\left[1+\Delta x\right]+\left(\left[u_{\Delta}\right]^n\right)^2-c^2\boldsymbol{1}+\mathcal{S}^{+}\left(u_{\Delta}\right)^n\mathcal{S}^{+}e^n+\mathcal{S}^{-}\left(u_{\Delta}\right)^n\mathcal{S}^{-}e^n-\frac{8\Delta x^2}{3}D\left(u_{\Delta}\right)^nD\left(e\right)^n+\frac{2c\Delta x}{3}D\left(e\right)^n,
\end{equation*}
\begin{equation*}
\hspace*{-8.6cm}F_{d}=(1-\theta)\Delta t\left[||D_+(u_{\Delta})^n||_{\ell^{\infty}}^{\sigma}+\frac{\Delta x}{2}||D_-\left(u_{\Delta}\right)^n||_{\ell^{\infty}}\right],
\end{equation*}
\begin{equation*}
\hspace*{-0.8cm}
F_{e}=2(1-\theta)\Delta t\left|\left|\left(u_{\Delta}\right)^n\right|\right|_{\ell^{\infty}}+2(1-\theta)\Delta t||e^n||_{\ell^{\infty}}-\Delta x+\Delta t(1-\theta)\left[\Delta x^{\frac{1}{2}-\gamma}+||e^n||_{\ell^{\infty}}+9||e^n||_{\ell^{\infty}}^2\Delta x^{\gamma-\frac{1}{2}}\right],
\end{equation*}
and
\begin{multline*}
\hspace*{-0.4cm}F_{f}=\Delta t\left[(1-2\theta)+\frac{c(1-\theta)\Delta x^2}{2}+\Delta t(1-\theta)\left|\left|D_+\left(u_{\Delta}\right)^n\right|\right|_{\ell^{\infty}}\right.\\\left.+(1-\theta)\frac{\Delta x^2}{4}\left(\Delta x^{\frac{1}{2}-\gamma}+||e^n||_{\ell^{\infty}}+9||e^n||_{\ell^{\infty}}^2\Delta x^{\gamma-\frac{1}{2}}\right)\right]-\frac{\Delta x^3}{4}.
\end{multline*}

\noindent $\bullet$ Since $||DD\left(u_{\Delta}\right)^n||_{\ell^{\infty}}\leq \frac{1}{\Delta x}||D\left(u_{\Delta}\right)^n||_{\ell^{\infty}}$, $||D\left(u_{\Delta}\right)^n||_{\ell^{\infty}}\leq ||D_+\left(u_{\Delta}\right)^n||_{\ell^{\infty}}$ and $\Delta t||D_+\left(u_{\Delta}\right)^n||_{\ell^{\infty}}\leq\frac{2\Delta t}{\Delta x}||u_{\Delta}^n||_{\ell^{\infty}}$, then
\begin{equation*}
F_{a}\leq A_{a},
\end{equation*}
where $A_{a}$ is defined by \eqref{Cf}.\\
$\bullet$ For $F_b$, we recognize the definition \eqref{Ca} of $A_{b}$.\\$\bullet$ For the term $F_{c}$, we have
\begin{equation*}
\begin{split}
F_c\leq ||e^n||_{\ell^{\infty}}^2\left[1+\Delta x\right]+||(u_{\Delta})^n||_{\ell^{\infty}}^2-c^2+\frac{1}{3}e_{j+1}^n\left(u_{\Delta}\right)_{j+1}^n+\frac{1}{3}e_{j-1}^n\left(u_{\Delta}\right)_{j-1}^n+\frac{2}{3}\left(u_{\Delta}\right)_{j+1}^ne_{j-1}^n+\frac{2}{3}\left(u_{\Delta}\right)_{j-1}^ne_{j+1}^n\\+\frac{2c}{3}||e^n||_{\ell^{\infty}}.
\end{split}
\end{equation*}
Thus, one has $F_{c}\leq A_{c}$ \eqref{Cc}.\\
$\bullet$ Furthermore, from \eqref{Cd} and \eqref{Ce}
\begin{equation*}
F_{d}=A_{d}
\end{equation*}
and
\begin{equation*}
F_{e}=A_{e}.
\end{equation*}
$\bullet$ At last,
we see that $F_{f}\leq A_{f}$ defined by \eqref{Cb}. This ends the proof.\end{proof}

It only  remains to prove the above  technical lemmas.

\subsection*{Proof of Lemma \ref{PROP_3333_3333_3333_2}}
\begin{proof}
Inequality \eqref{EQ_2_LEMMA_2} is based on Relation \eqref{EQ_2bis_LEMMA_BASIC}
\begin{equation*}
\begin{split}
\left|\left|D\left(ab\right)\right|\right|_{\ell^2_{\Delta}}^2&=\left|\left|bD\left(a\right)+\frac{\mathcal{S}^+a}{2}D_+\left(b\right)+\frac{\mathcal{S}^-a}{2}D_-\left(b\right)\right|\right|_{\ell^2_{\Delta}}^2\\
&=\left|\left|bD\left(a\right)\right|\right|_{\ell^2_{\Delta}}^2+\left\langle bD\left(a\right), \mathcal{S}^+aD_+\left(b\right)\right\rangle+\left\langle b D\left(a\right), \mathcal{S}^-aD_-\left(b\right)\right\rangle\\
&+\left|\left|\frac{\mathcal{S}^+a}{2}D_+\left(b\right)\right|\right|_{\ell^2_{\Delta}}^2+\frac{1}{2}\left\langle \mathcal{S}^+aD_+\left(b\right), \mathcal{S}^-aD_-\left(b\right)\right\rangle+\left|\left|\frac{\mathcal{S}^-a}{2}D_-\left(b\right)\right|\right|_{\ell^2_{\Delta}}^2.
\end{split}
\end{equation*}
We conclude from the Young inequality which yields 
\begin{equation*}
\begin{split}
&\left|\left| D\left(ab\right)\right|\right|_{\ell^2_{\Delta}}^2\leq \left|\left|bD\left(a\right)\right|\right|_{\ell^2_{\Delta}}^2+\frac{1}{2\Delta t}\left|\left|b\mathcal{S}^+a\right|\right|_{\ell^2_{\Delta}}^2+\frac{\Delta t}{2}\left|\left|D\left(a\right)D_+\left(b\right)\right|\right|_{\ell^2_{\Delta}}^2+\frac{1}{2\Delta t}\left|\left|b\mathcal{S}^-a\right|\right|_{\ell^2_{\Delta}}^2\\
&+\frac{\Delta t}{2}\left|\left|D\left(a\right)D_-\left(b\right)\right|\right|_{\ell^2_{\Delta}}^2+\frac{3}{2}\left|\left|\frac{\mathcal{S}^+a}{2}D_+\left(b\right)\right|\right|_{\ell^2_{\Delta}}^2+\frac{3}{2}\left|\left|\frac{\mathcal{S}^-a}{2}D_-\left(b\right)\right|\right|_{\ell^2_{\Delta}}^2.
\end{split}
\end{equation*}
\end{proof}
\subsection*{Proof of Lemma \ref{PROP_32_je-ne-sais-plus}}
We shall  start by establishing the following lemma. 
\begin{Lemma} Let $\left(a_j\right)_{j\in\mathbb{Z}}$ and $\left(b_j\right)_{j\in\mathbb{Z}}$ be two sequences in $\ell^2_{\Delta}(\mathbb{Z})$, $\sigma$ be in $\{0,1\}$ and $\nu$ be non negative. Then, it holds
\begin{multline}
\left\langle D_+D_+D_-\left(a\right), bD\left(a\right)\right\rangle \leq\frac{1}{2}\left\langle \Delta x^{\nu}\left(\frac{|D_-(b)|^{\sigma}}{2}+\frac{|D_-(b)|^{\sigma}}{2}\right)-\frac{\Delta x}{2}D_-b, \left(D_+D_-\left(a\right)\right)^2\right\rangle\\+\frac{1}{2\Delta x^{\nu}}\left\langle |D_+\left(b\right)|^{2-\sigma}, \left(D_+\left(a\right)\right)^2\right\rangle-\left\langle b, \left(D_+D\left(a\right)\right)^2\right\rangle.
\end{multline}
\label{LEMMA_JE_NE_SAIS_PLUS_COMBIEN}
\end{Lemma}

\begin{proof}[Proof of Lemma \ref{LEMMA_JE_NE_SAIS_PLUS_COMBIEN}]
By developing $D\left(a\right)_j$ and using the relation \eqref{EQ_3_LEMMA_BASIC}, it holds
\begin{equation*}
\begin{split}
\left\langle D_+D_+D_-\left(a\right), bD\left(a\right)\right\rangle &=\left\langle D_+D_+D_-\left(a\right), \frac{b}{2}D_+\left(a\right)\right\rangle+\left\langle D_+D_+D_-\left(a\right), \frac{b}{2}D_-\left(a\right)\right\rangle\\
&=-\left\langle D_+D_-\left(a\right), D_-\left(\frac{b}{2}D_+\left(a\right)\right)\right\rangle-\left\langle D_+D_-\left(a\right), D_-\left(\frac{b}{2}D_-\left(a\right)\right)\right\rangle.
\end{split}
\end{equation*}
We focus first on the term $-\left\langle D_+D_-\left(a\right), D_-\left(\frac{b}{2}D_+\left(a\right)\right)\right\rangle$. Equality \eqref{EQ_UUTTIILLEESS_22} gives
\begin{equation*}
\begin{split}
-\left\langle D_+D_-\left(a\right), D_-\left(\frac{b}{2}D_+\left(a\right)\right)\right\rangle=-\left\langle D_+D_-\left(a\right), \frac{D_-\left(b\right)}{2}D_-\left(a\right)+\frac{b}{2}D_+D_-\left(a\right)\right\rangle.
\end{split}
\end{equation*}
Eventually, Young inequality provides
\begin{multline*}
-\left\langle D_+D_-\left(a\right), D_-\left(\frac{b}{2}D_+\left(a\right)\right)\right\rangle\leq \frac{\Delta x^{\nu}}{4}\left\langle |D_-b|^{\sigma}, \left(D_+D_-\left(a\right)\right)^2\right\rangle+\frac{1}{4\Delta x^{\nu}}\left\langle |D_+\left(b\right)|^{2-\sigma}, \left(D_+\left(a\right)\right)^2\right\rangle \\-\left\langle \frac{b}{2}, \left(D_+D_-\left(a\right)\right)^2\right\rangle.
\end{multline*}
For the term $-\left\langle D_+D_-\left(a\right), D_-\left(\frac{b}{2}D_-\left(a\right)\right)\right\rangle$, one has thanks to Equality \eqref{EQ_UUTTIILLEESS_22},
\begin{equation*}
-\left\langle D_+D_-\left(a\right), D_-\left(\frac{b}{2}D_-\left(a\right)\right)\right\rangle=-\left\langle D_+D_-\left(a\right), \frac{D_-\left(b\right)}{2}D_-\left(a\right)+\frac{\mathcal{S}^-b}{2}D_-D_-\left(a\right)\right\rangle.
\end{equation*}
Hence, it holds (by Young inequality)
\begin{small}
\begin{equation*}
\begin{split}
&\hspace*{-1.8cm}-\left\langle D_+D_-\left(a\right), D_-\left(\frac{b}{2}D_-\left(a\right)\right)\right\rangle\\
&\leq \frac{\Delta x^{\nu}}{4}\left\langle |D_-b|^{\sigma}, \left(D_+D_-\left(a\right)\right)^2\right\rangle +\frac{1}{4\Delta x^{\nu}}\left\langle |D_+\left(b\right)|^{2-\sigma}, \left(D_+a\right)^2\right\rangle -\left\langle \frac{\mathcal{S}^-b}{2}D_-D_-\left(a\right), D_+D_-\left(a\right)\right\rangle\\
&\leq \frac{\Delta x^{\nu}}{4}\left\langle |D_-b|^{\sigma}, \left(D_+D_-\left(a\right)\right)^2\right\rangle+\frac{1}{4\Delta x^{\nu}}\left\langle |D_+\left(b\right)|^{2-\sigma}, \left(D_+a\right)^2\right\rangle\\
&-\left\langle \mathcal{S}^-b, \left(\frac{D_+D_-a+D_-D_-a}{2}\right)^2\right\rangle+\left\langle \frac{\mathcal{S}^-b}{4}, \left(D_+D_-a\right)^2\right\rangle+\left\langle \frac{\mathcal{S}^-b}{4}, \left(D_-D_-a\right)^2\right\rangle\\
&\leq \left\langle \frac{\Delta x^{\nu}|D_-b|^{\sigma}}{4}+\frac{\mathcal{S}^-b+b}{4}, \left(D_+D_-a\right)^2\right\rangle-\left\langle b, \left(D_+Da\right)^2\right\rangle+\frac{1}{4\Delta x^{\nu}}\left\langle |D_+\left(b\right)|^{2-\sigma}, \left(D_+\left(a\right)\right)^2\right\rangle.
\end{split}
\end{equation*}\end{small}

By collecting the previous results, one has
\begin{multline*}
\left\langle D_+D_+D_-\left(a\right), bD\left(a\right)\right\rangle \leq\left\langle \left\{\frac{\Delta x^{\nu}|D_-b|^{\sigma}}{4}+\frac{\Delta x^{\nu}|D_-b|^{\sigma}}{4}+\frac{\mathcal{S}^-b-b}{4}\right\}, \left(D_+D_-a\right)^2\right\rangle\\+\frac{1}{2\Delta x^{\nu}}\left\langle|D_+\left(b\right)|^{2-\sigma}, \left(D_+\left(a\right)\right)^2\right\rangle -\left\langle b, \left(D_+Da\right)^2\right\rangle.
\end{multline*}
Lemma \ref{LEMMA_JE_NE_SAIS_PLUS_COMBIEN} is then proved.
\end{proof}

We can then finish the proof of Lemma \ref{PROP_32_je-ne-sais-plus}.

We use relation \eqref{EQ_2bis_LEMMA_BASIC} to develop $D_+D_+D_-\left(a\right)_jD\left(ab\right)_j$ which gives (thanks to  the Young inequality)\begin{equation}
\begin{split}
\left\langle D_+D_+D_-\left(a\right), D\left(ab\right)\right\rangle&=\left\langle D_+D_+D_-\left(a\right), bD\left(a\right)+\frac{\mathcal{S}^+a}{2}D_+\left(b\right)+\frac{\mathcal{S}^-a}{2}D_-\left(b\right)\right\rangle\\
&\leq\left\langle D_+D_+D_-\left(a\right), bD\left(a\right)\right\rangle+\frac{\Delta t}{4}\left\langle \left(D_+D_+D_-\left(a\right)\right)^2, |D_+\left(b\right)|\right\rangle\\
&+\frac{1}{4\Delta t}\left\langle \left(\mathcal{S}^+a\right)^2, |D_+\left(b\right)|\right\rangle+\frac{\Delta t}{4}\left\langle \left(D_+D_+D_-\left(a\right)\right)^2, |D_-\left(b\right)|\right\rangle +\frac{1}{4\Delta t}\left\langle \left(\mathcal{S}^-a\right)^2, |D_-\left(b\right)|\right\rangle.
\end{split}
\label{eq_annexe_prop_3}
\end{equation}
The conclusion comes from Lemma \ref{LEMMA_JE_NE_SAIS_PLUS_COMBIEN} with $\nu=0$.

\subsection*{Proof of Lemma \ref{Prop_B,5}}
To prove Lemma \ref{Prop_B,5}, we first develop the left-hand side thanks to \eqref{EQ_2bis_LEMMA_BASIC}
\begin{equation*}
\left\langle D_+D_+D_-a, D\left(\frac{a^2}{2}\right)\right\rangle=\left\langle D_+D_+D_-\left(a\right), \left[\frac{a}{2}D\left(a\right)+\frac{\mathcal{S}^+a}{4}D_+\left(a\right)+\frac{\mathcal{S}^-a}{4}D_-\left(a\right)\right]\right\rangle.
\end{equation*}
$\bullet$
The first term $\left\langle D_+D_+D_-\left(a\right), \frac{a}{2}D\left(a\right)\right\rangle$ is treated with Lemma \ref{LEMMA_JE_NE_SAIS_PLUS_COMBIEN} above, 
with $\nu=\frac{1}{2}-\gamma$ and $\sigma=0$, which rewrites
\begin{equation*}
\left\langle D_+D_+D_-(a), \frac{a}{2}D(a)\right\rangle \leq \frac{1}{4}\left\langle \left\{\Delta x^{\frac{1}{2}-\gamma}\boldsymbol{1}-\frac{\Delta x}{2}D_-a\right\}, \left(D_+D_-a\right)^2\right\rangle +\frac{1}{4\Delta x^{\frac{1}{2}-\gamma}}\left|\left|D_+a\right|\right|_{\ell^4_{\Delta}}^4-\frac{1}{2}\left\langle a, \left(D_+Da\right)^2\right\rangle.
\end{equation*}
$\bullet$
For the second term, we integrate by parts thanks to \eqref{EQ_3_LEMMA_BASIC} and \eqref{EQ_UUTTIILLEESS_22}
\begin{equation*}
\begin{split}
\left\langle D_+D_+D_-\left(a\right), \frac{\mathcal{S}^+a}{4}D_+a\right\rangle&=-\left\langle D_+D_-\left(a\right), D_-\left(\frac{\mathcal{S}^+a}{4}D_+a\right)\right\rangle\\
&=-\left\langle D_+D_-\left(a\right), \frac{a}{4}D_+D_-\left(a\right)+\frac{\left(D_+\left(a\right)\right)^2}{4}\right\rangle.
\end{split}
\end{equation*}
Young inequality completes the upper bound
\begin{equation*}
\begin{split}
\left\langle D_+D_+D_-\left(a\right), \frac{\mathcal{S}^+a}{4}D_+a\right\rangle
&\leq-\left\langle \left(D_+D_-(a)\right)^2, \frac{a}{4}\right\rangle+\frac{\Delta x^{\frac{1}{2}-\gamma}}{8}\left|\left|D_+D_-a\right|\right|_{\ell^2_{\Delta}}^2+\frac{1}{8\Delta x^{\frac{1}{2}-\gamma}}\left|\left|D_+a\right|\right|_{\ell^4_{\Delta}}^4.
\end{split}
\end{equation*}
$\bullet$ For the third term, Relation \eqref{EQ_3_LEMMA_BASIC} together with \eqref{EQ_UUTTIILLEESS_11} gives
\begin{small}
\begin{equation*}
\begin{split}
&\left\langle D_+D_+D_-\left(a\right), \frac{\mathcal{S}^-a}{4}D_-a\right\rangle=-\left\langle D_+D_+\left(a\right), D_+\left(\frac{\mathcal{S}^-a}{4}D_-a\right)\right\rangle\\
&=-\left\langle D_+D_+\left(a\right), \frac{a}{4}D_+D_-\left(a\right)+\frac{\mathcal{S}^-D_+\left(a\right)}{4}D_-\left(a\right)\right\rangle\\
&= -\left\langle \frac{a}{2}, \left(\frac{D_+D_+\left(a\right)+D_+D_-\left(a\right)}{2}\right)^2\right\rangle+\left\langle \frac{a}{8}, \left(D_+D_+\left(a\right)\right)^2\right\rangle +\left\langle\frac{a}{8}, \left(D_+D_-\left(a\right)\right)^2\right\rangle-\left\langle D_+D_+\left(a\right), \frac{\left(D_-\left(a\right)\right)^2}{4}\right\rangle\\
&\leq -\left\langle \frac{a}{2}, \left(D_+D(a)\right)^2\right\rangle +\left\langle \frac{\mathcal{S}^-a+a}{8}, \left(D_+D_-(a)\right)^2\right\rangle+\frac{\Delta x^{\frac{1}{2}-\gamma}}{8}\left|\left|D_+D_-(a)\right|\right|_{\ell^2_{\Delta}}^2+\frac{1}{8\Delta x^{\frac{1}{2}-\gamma}}\left|\left|D_+(a)\right|\right|_{\ell^4_{\Delta}}^4.
\end{split}
\end{equation*}
\end{small}
Gathering all these results yields
\begin{small}
\begin{equation*}
\begin{split}
\left\langle D_+D_+D_-(a), D\left(\frac{a^2}{2}\right)\right\rangle \leq \left\langle \frac{\Delta x^{\frac{1}{2}-\gamma}}{2}\boldsymbol{1}-\frac{\Delta x}{8}D_-(a)+\frac{\mathcal{S}^-a-a}{8}, \left(D_+D_-(a)\right)^2\right\rangle+\frac{1}{2\Delta x^{\frac{1}{2}-\gamma}}\left|\left|D_+a\right|\right|_{\ell^4_{\Delta}}^4\\-\left\langle a, \left(D_+D(a)\right)^2\right\rangle.
\end{split}
\end{equation*}
\end{small}
To conclude this proof,  it suffices to  use the following lemma.
\begin{Lemma} Let $(a_j)_{j\in\mathbb{Z}}$ be a sequence in $\ell^2_{\Delta}(\mathbb{Z})$, then one has
\begin{equation*}
||D_+a||_{\ell^4_{\Delta}}\leq \sqrt{3||a||_{\ell^{\infty}}||D_+D_-a||_{\ell^{2}_{\Delta}}}.
\end{equation*}
\label{Norm_L^4}
\end{Lemma}
This result is a discrete version of a classical Gagliardo-Nirenberg inequality, thus we leave its proof to the reader.

\bibliography{biblio}

\end{document}